\documentclass[10pt,a4paper]{article}
\usepackage[a4paper]{geometry}
\usepackage{amssymb,latexsym,amsmath,amsfonts,amsthm,caption}
\usepackage{graphicx}

\usepackage{epsfig}
\usepackage{subfigure}

\DeclareMathOperator{\diag}{diag} 

\newcommand{\er}{\mathbb{R}}
\newcommand{\cee}{\mathbb{C}}
\newcommand{\enn}{\mathbb{N}}
\newcommand{\zet}{\mathbb{Z}}
\newcommand{\lam}{\lambda}

\newcommand{\bol}{\hfill\square\\}
\newcommand{\wtil}{\widetilde}
\newcommand{\what}{\widehat}
\newcommand{\til}{\tilde}
\renewcommand{\Re}{\mathrm{Re}\,}
\renewcommand{\Im}{\mathrm{Im}\,}

\newcommand{\vece}{\mathbf{e}}

\newcommand{\vecQ}{\mathbf{Q}}

\newcommand{\vecv}{\mathbf{v}}
\newcommand{\vecPsi}{\mathbf{\Psi}}
\newcommand{\om}{\omega}

\newcommand{\ud}{\,\mathrm{d}}
\newcommand{\ir}{\text{\rm{i}}}
\newcommand{\inter}{\text{\rm{int}}}

\newtheorem{theorem}{Theorem}[section]
\newtheorem{lemma}[theorem]{Lemma}
\newtheorem{proposition}[theorem]{Proposition}

\newtheorem{corollary}[theorem]{Corollary}

\theoremstyle{definition}
\newtheorem{definition}[theorem]{Definition}

\theoremstyle{remark}

\newtheorem{remark}[theorem]{Remark}

\numberwithin{equation}{section}

\title{High order three-term recursions, Riemann-Hilbert minors and Nikishin systems on star-like sets}

\date{\today}

\author{Steven Delvaux\footnotemark[1]\,  and Abey L\'opez\footnotemark[1]}

\begin{document}


\maketitle
\renewcommand{\thefootnote}{\fnsymbol{footnote}}
\footnotetext[1]{Department of Mathematics, University of Leuven (KU Leuven),
Celestijnenlaan 200B, B-3001 Leuven, Belgium. email: \{steven.delvaux,
abey.lopezgarcia\}\symbol{'100}wis.kuleuven.be. The authors are Postdoctoral
Fellows of the Fund for Scientific Research-Flanders (FWO), Belgium.}

\begin{abstract}
We study monic polynomials $Q_n(x)$ generated by a high order three-term
recursion $xQ_n(x)=Q_{n+1}(x)+a_{n-p} Q_{n-p}(x)$ with arbitrary $p\geq 1$ and
$a_n>0$ for all $n$. The recursion is encoded by a two-diagonal Hessenberg
operator $H$. One of our main results is that, for periodic coefficients $a_n$
and under certain conditions, the $Q_n$ are multiple orthogonal polynomials
with respect to a Nikishin system of orthogonality measures supported on
star-like sets in the complex plane. This improves a recent result of
Aptekarev-Kalyagin-Saff where a formal connection with Nikishin systems was
obtained in the case when $\sum_{n=0}^{\infty}|a_n-a|<\infty$ for some $a>0$.

An important tool in this paper is the study of \lq Riemann-Hilbert minors\rq,
or equivalently, the \lq generalized eigenvalues\rq\ of the Hessenberg matrix
$H$. We prove interlacing relations for the generalized eigenvalues by using
totally positive matrices. In the case of asymptotically
periodic coefficients $a_n$, we find weak and ratio asymptotics for
the Riemann-Hilbert minors and we obtain a connection with a vector equilibrium
problem.
We anticipate that in the future, the study of Riemann-Hilbert
minors may prove useful for more general classes of multiple orthogonal
polynomials.\smallskip

\textbf{Keywords:} Multiple orthogonal polynomial, Nikishin system, banded
Hessenberg matrix, block Toeplitz matrix, Riemann-Hilbert matrix, generalized
Poincar\'{e} theorem, ratio asymptotics, vector equilibrium problem,
interlacing, totally positive matrix.\smallskip

\textbf{MSC 2010:} Primary $42C05$; Secondary $15B05$,
$15B48$.

\end{abstract}

\setcounter{tocdepth}{2} \tableofcontents

\section{Introduction}


Let $(Q_n)_{n=0}^{\infty}$ be the sequence of monic polynomials generated by
the recurrence relation
\begin{equation}\label{recurrencerel} xQ_n(x)=Q_{n+1}(x)+a_{n-p}
Q_{n-p}(x),\qquad n\geq 0,\end{equation} for a fixed integer
$p\in\enn:=\{1,2,3,\ldots\}$, with
initial conditions
\begin{equation}\label{initialcond} Q_0(x)\equiv 1,\qquad Q_{-1}(x)\equiv\cdots\equiv Q_{-p}(x)\equiv 0.\end{equation}
The recurrence coefficients $a_n$ are assumed to be positive real numbers:
\begin{equation}\label{assumption:anpos}a_n>0,\qquad n\geq 0.\end{equation}
Note that for $p=1$, \eqref{recurrencerel} reduces to the standard three-term
recurrence relation for orthogonal polynomials on the real line, in the special
case of an even orthogonality measure. We will be interested in the case where
$p\geq 2$, which we refer to as a \emph{high order three-term recurrence}
\cite{AKS}.

The assumption \eqref{assumption:anpos} implies that the zeros of $Q_n$ are
located on the \emph{star} $S_{+} := \{ x\in\cee\mid x^{p+1}\in\er_{+}\},$ and
that they satisfy certain \emph{interlacing relations}. This was demonstrated
by Eiermann-Varga \cite{EV} and Romdhane \cite{Rom}; see also
Fig.~\ref{fig:plotLimit} and \ref{fig:interlace1} below for the
case $p=2$. In the present paper we will obtain more general interlacing
relations, in the context of so-called Riemann-Hilbert minors.

\smallskip
The polynomials $Q_n$ are studied in the literature under various assumptions
on the recurrence coefficients $a_n$. He and Saff \cite{HeSaff} show that the
Faber polynomials associated with the closed domain bounded by a $(p+1)$-cusped
hypocycloid satisfy the recursion \eqref{recurrencerel} with constant
coefficients $a_n=a=1/p$. Many properties of these Faber polynomials are
obtained in \cite{EV,HeSaff}.

More properties and applications for the polynomials $Q_n$ are obtained by Ben
Cheikh-Douak \cite{BenCheikh}, Douak-Maroni \cite{DouMar1}, Maroni
\cite{Maroni} and others \cite{Milov,Rom}. The polynomials $Q_n$ are often
called \emph{$d$-symmetric $d$-orthogonal polynomials} in these references
(with $d:=p$). An application from the normal matrix model is given in
\cite{BKnormal}.

General considerations \cite{DouMar,Kal1} show that the polynomials $Q_n$
satisfy formal \emph{multiple orthogonality relations} with respect to certain
linear functionals. Aptekarev, Kalyagin and Van Iseghem \cite{AKVI} obtain a
stronger version of this result:

\begin{theorem}\label{theorem:Qn:as:mop:0} (See \cite[Th.~1.1]{AKS}, \cite[Cor.~2]{AKVI}:)
Suppose that $a_n>0$ for all $n$ and the numbers $a_n$ are uniformly bounded.
Then the polynomials $Q_n(x)$ are multiple orthogonal with respect to the
measures $\nu_1,\ldots,\nu_p$ defined in \eqref{nuk:Helly} (see
Section~\ref{section:Nikishin}), in the sense that
\begin{equation}\label{Qn:MOP:0}
\int Q_n(x)x^m \ud\nu_j(x) = 0,
\end{equation}
for any $m\in [0:\lfloor\frac{n-j}{p}\rfloor]$ and $j\in [1:p]$.
\end{theorem}

Here $x\mapsto\lfloor x\rfloor$ denotes the \lq floor\rq\ function and we
abbreviate $[i:j]:=\{i,i+1,\ldots,j\}$. This notation will be used throughout
the paper.

The measures $\nu_1,\ldots,\nu_p$ are supported on a compact subset of the star
$S_+$. We will call them the \emph{orthogonality measures}. Aptekarev, Kalyagin
and Saff \cite{AKS} study these measures in the case where
$\sum_{n=0}^{\infty}|a_n-a|<\infty$ for some $a>0$. They obtain a formal link
with Nikishin systems. In the present paper we will extend this link to the
case of periodic $a_n$. In particular, we will obtain conditions guaranteeing
that $\nu_1,\ldots,\nu_p$ form a \emph{true}, rather than a formal, Nikishin
system.

For any $j\in [1:p],$ define the \emph{second kind function} $\Psi_{n}^{(j)}$
by
\begin{equation}\label{secondkind:def}
\Psi_{n}^{(j)}(z) := \int \frac{Q_{n}(t)}{z-t}\ud\nu_j(t),\qquad n\geq 0.
\end{equation}
Define the \emph{Riemann-Hilbert matrix} (briefly \emph{RH matrix}) $Y_n(z)$ by
\begin{equation}\label{def:Y}
Y_{n}(z) = \begin{pmatrix} Q_{n}(z) & \Psi_{n}^{(1)}(z) & \ldots &
\Psi_{n}^{(p)}(z)  \\ Q_{n-1}(z) & \Psi_{n-1}^{(1)}(z) & \ldots &
\Psi_{n-1}^{(p)}(z)
\\
\vdots & \vdots & & \vdots \\ Q_{n-p}(z) & \Psi_{n-p}^{(1)}(z) & \ldots &
\Psi_{n-p}^{(p)}(z)
\end{pmatrix}.
\end{equation}
This definition is a variant of the one in Van Assche, Geronimo and
Kuijlaars~\cite{VAGK}, see also \cite{FIK}. The matrix $Y_n(z)$ satisfies a
certain Riemann-Hilbert problem; but we will not need this here.

Denote the principal $(k+1)\times (k+1)$ minor of $Y_n(z)$ by
\begin{equation}\label{RHminor:principal:def} B_{k,n}(z) = \det\begin{pmatrix} Q_n(z) & \Psi_n^{(1)}(z) &
\ldots &
\Psi_n^{(k)}(z) \\ \vdots & \vdots & & \vdots \\
Q_{n-k}(z) & \Psi_{n-k}^{(1)}(z) & \ldots & \Psi_{n-k}^{(k)}(z)
\end{pmatrix},
\end{equation}
for $k\in [0:p]$. We call this the \emph{$k$th principal Riemann-Hilbert minor}
of $Y_{n}$. For $n<k$ we set $B_{k,n}(z)\equiv 1$. In this paper we will also
work with the determinants of more general submatrices of \eqref{def:Y}, whose
rows are not necessarily consecutive; see Section~\ref{section:geneig} and
following.


\begin{lemma}\label{lemma:RHminor:degree}
For any $k\in [0:p]$, $B_{k,n}(x)$ is a polynomial of degree
\begin{equation*}
\deg B_{k,n}\leq \frac{p-k}{p}(n-k).
\end{equation*}
\end{lemma}

\begin{proof}
First we prove that $B_{k,n}(x)$ is a polynomial. By the multi-linearity of the
determinant,
\[
B_{k,n}(z)=\int\cdots\int \det\begin{pmatrix} Q_n(z) & Q_{n}(y_{1}) & \ldots &
Q_n(y_{k}) \\ \vdots & \vdots & & \vdots \\
Q_{n-k}(z) & Q_{n-k}(y_{1}) & \ldots & Q_{n-k}(y_{k})
\end{pmatrix}\frac{\ud\nu_{1}(y_{1})\ldots \ud\nu_{k}(y_{k})}{(z-y_{1})\cdots(z-y_{k})}.
\]
The integrand is clearly a polynomial in $z$, hence $B_{k,n}$ is a polynomial.
Finally, the claim about the degree of $B_{k,n}(z)$ will be a consequence of
Prop.~\ref{prop:geneigRH} and Lemma~\ref{lemma:geneig:degree} in what follows.
(This claim may be shown in a direct way as well.)
\end{proof}

Note in particular that $\deg B_{p,n}=0$, i.e., the determinant of the full RH
matrix $Y_n(z)$ is a constant. Prop.~\ref{prop:geneigRH} will imply that this
constant is nonzero.

Define the two complementary \lq stars\rq
\begin{equation}\label{stars}
S_{\pm} := \{ x\in\cee\mid x^{p+1}\in\er_{\pm}\}.
\end{equation}
In this paper we will prove that
the zeros of $B_{k,n}$ (and of more general RH minors) are all located on the
star $S_+$ if $k$ is even and on the star $S_-$ if $k$ is odd. We will also
obtain several kinds of
interlacing relations between the zeros of the different RH minors. \\

The main focus of this paper is on the case where the recurrence coefficients
$a_n$ are asymptotically periodic of period $r\in\enn$. This means that
\begin{equation}\label{periodic:asy:twodiag}
\lim_{n\to\infty} a_{rn+j} =: b_j>0,\qquad j\in [0:r-1],
\end{equation}
for certain limiting values $b_0,\ldots,b_{r-1}>0$.

It turns out that in the asymptotically periodic case, the zeros of $Q_n$ for
$n\to\infty$ are attracted (in the sense of weak convergence) by a certain
rotationally invariant subset $\Gamma_0$ of the star $S_+$. Moreover, the zeros
asymptotically distribute themselves according to a measure $\mu_0$ on
$\Gamma_0$, which appears in the solution to a certain vector equilibrium
problem. An example of the set $\Gamma_0$ is shown in the left picture of
Fig.~\ref{fig:plotLimit}.
Below we will also introduce a family of sets $\Gamma_k$ and measures $\mu_k$,
$k\in [0:p-1]$, which will be the limiting zero distributions of the RH minors
$B_{k,n}$.


\begin{figure}[htbp]
\begin{center}\vspace{-1mm}
        \subfigure{}\includegraphics[scale=0.32]{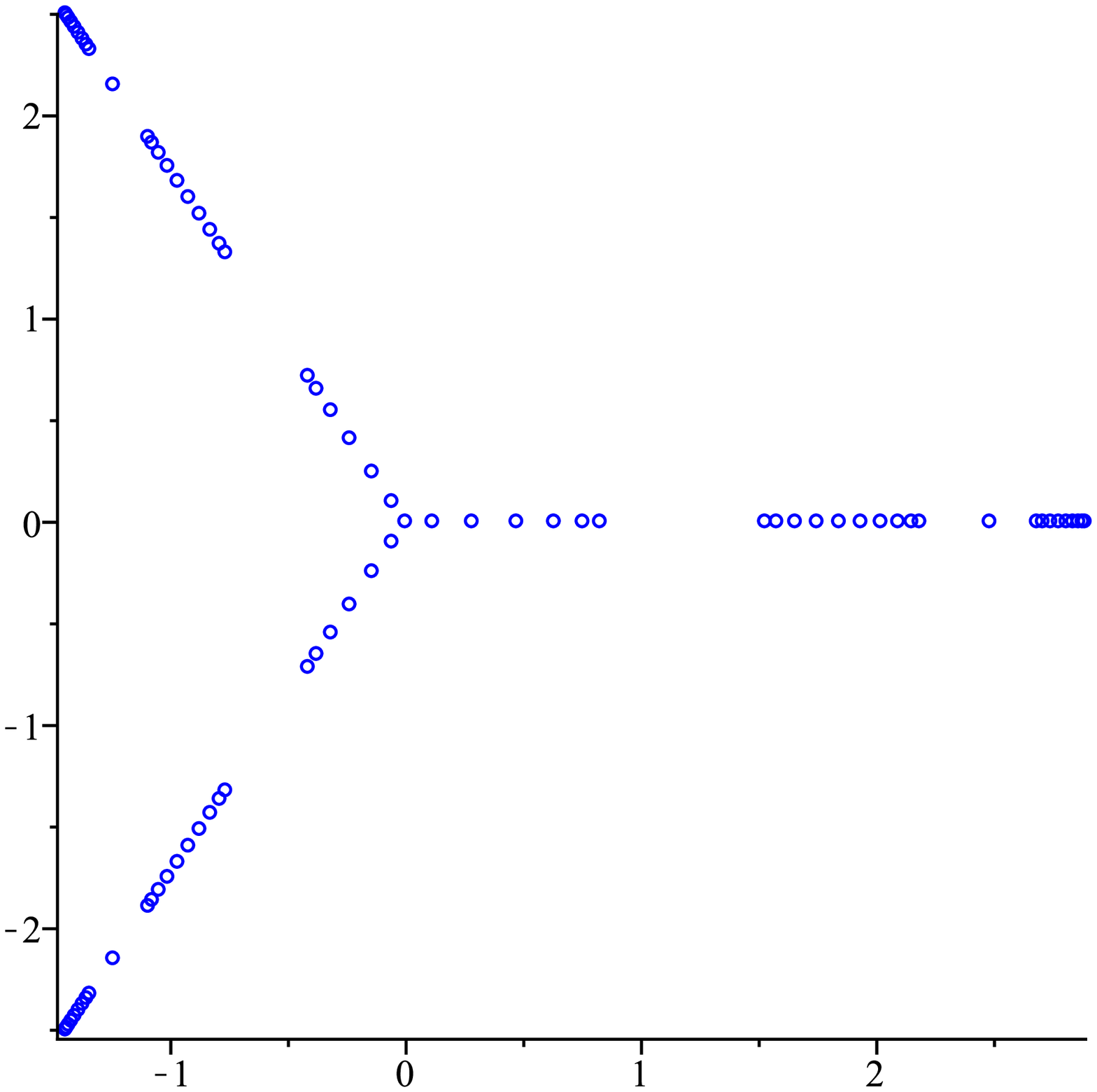}\hspace{15mm}
        \subfigure{}\includegraphics[scale=0.32]{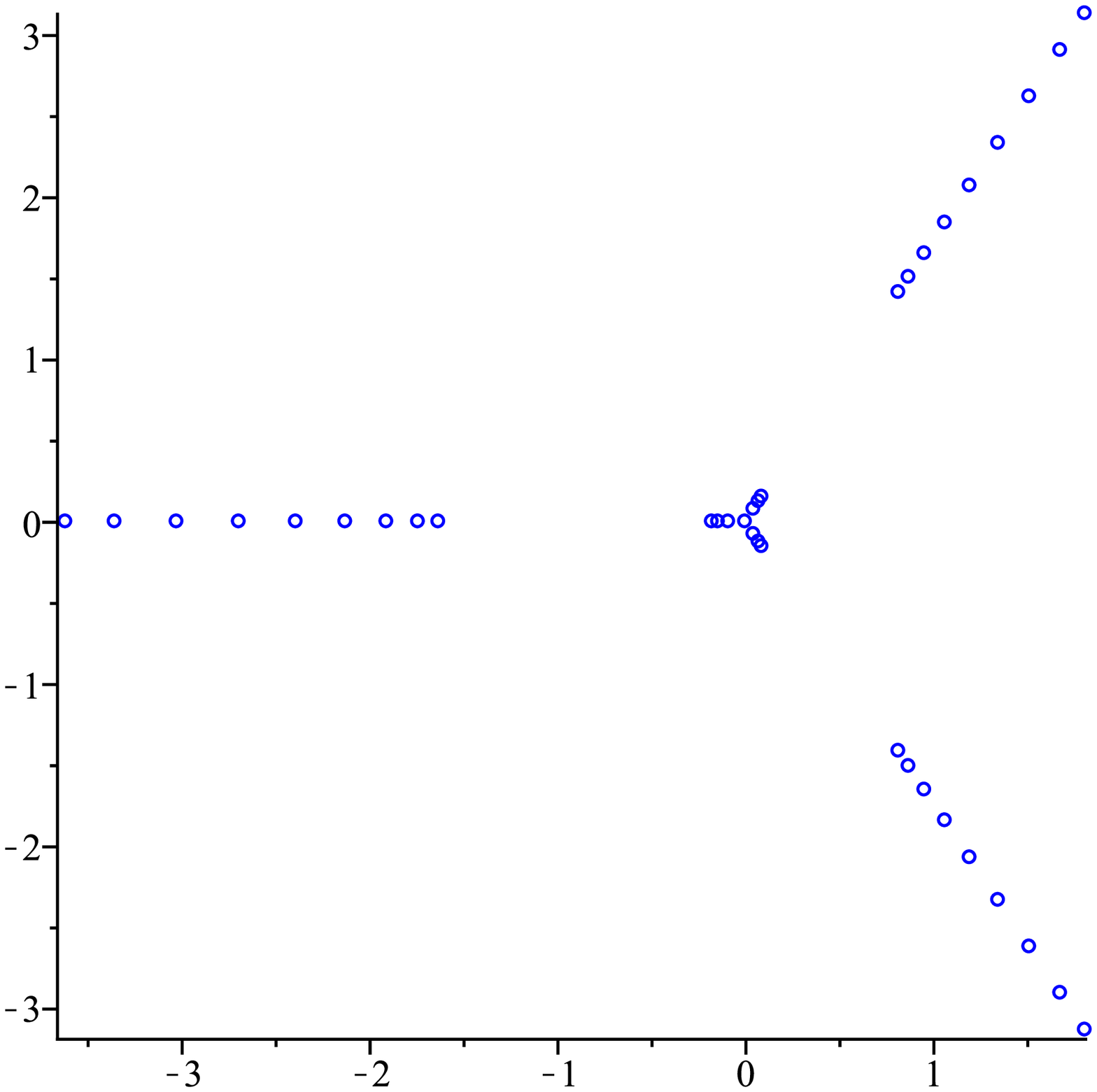}
\end{center}\vspace{-8mm}
\caption{Zeros of $Q_{80}$ (left) and $P_{1,80}$ (right) in the periodic case
with $p=2$ and period $r=8$ and $(a_0,\ldots,a_7)=(3, 1, 5, 2, 2, 9, 6, 1)$.
The zeros of $Q_n$ accumulate on a set $\Gamma_0\subset S_+$ whose intersection
with $\er$ is $[0,0.85]\cup [1.52,2.19]\cup [2.67,2.89]$ (using two digits of
precision). The zeros of $P_{1,n}$ accumulate on a set $\Gamma_1\subset S_-$
whose intersection with $\er$ is $[-3.72,-1.59]\cup [-0.17,0]$. Note that
$Q_{80}$ has an isolated zero between some of the intervals.}
\label{fig:plotLimit}
\end{figure}

Define the matrix
\begin{equation}\label{symbol}
F(z,x) := Z^{-1}+ Z^{p}\diag(b_0,\ldots,b_{r-1})-xI_r,
\end{equation}
and the algebraic curve
\begin{equation}\label{def:algcurve}
0=f(z,x) := \det F(z,x),
\end{equation}
where $Z$ denotes the cyclic shift matrix
\begin{equation}\label{def:Sshift}
Z= \begin{pmatrix}0 & z \\ I_{r-1} & 0 \end{pmatrix},
\end{equation}
and where $I_k$ denotes the identity matrix of size $k$. If $r=1$ then we put
$Z=z$ and $b_0=:b$. In that case, \eqref{def:algcurve} reduces to the algebraic
curve $z^{-1}+ bz^p-x=0$ in \cite{AKS,HeSaff}. The matrix $F(z,x)$ can be
interpreted as the \emph{symbol} of a block Toeplitz matrix. This is explained
in Section~\ref{section:widom}.

The expression $f(z,x)$ can be expanded as a Laurent polynomial in~$z$:
\begin{equation}\label{f:series}
f(z,x) = (-1)^{r-1}z^{-1} + \mathsf{f}_0(x) + \mathsf{f}_1(x)z+\cdots+ \mathsf{f}_{p}(x)z^p,
\end{equation}
where each $\mathsf{f}_k(x)$, $k\in [0:p]$, is a polynomial in $x$, and
\begin{equation}\label{fp} \mathsf{f}_p(x) \equiv \mathsf{f}_p =
(-1)^{p(r-p)}\prod_{k=0}^{r-1}b_k.
\end{equation}

The algebraic equation $f(z,x)=0$ has precisely $p+1$ roots $z_k=z_k(x)$, $k\in
[0:p]$ (counting multiplicities), and we order them by increasing modulus as
\begin{equation}\label{ordering:rootszk}
|z_0(x)|\leq |z_1(x)|\leq \cdots \leq |z_{p}(x)|
\end{equation}
for all $x\in\cee$. If $x\in\cee$ is such that two or more subsequent roots
$z_k(x)$ in \eqref{ordering:rootszk} have the same modulus then we may
arbitrarily label them so that \eqref{ordering:rootszk} is satisfied. It is
easy to see (see e.g.\ \cite[Sec.~4]{Del} or \cite[p.~102]{vMM}) that for
$x\rightarrow\infty$,
\begin{equation}\label{zk:infty:intro}
z_{0}(x)=x^{-r}+O(x^{-r-1}),\qquad z_{k}(x)= O(x^{r/p}), 
\quad k\in [1:p].
\end{equation}
More precisely, for any $x\in\cee$ there is a
permutation $(\til z_k(x))_{k=1}^p$ of the set $(z_k(x))_{k=1}^p$ so that
\begin{equation}\label{zk:infty:intro:bis}\wtil z_k(x)^{p/d} = \left(\prod_{n=0}^{r/d-1}b_{dn+(k-1\!\!\!\!\mod\! d)}\right)^{-1}
x^{r/d} (1+o(1)),\qquad k\in[1:p],
\end{equation}
as $x\to\infty$,
where $d:=\gcd\{p,r\}$. 
See \cite[p.~102]{vMM}.


\smallskip Define the sets $\Gamma_k$ by
\begin{equation}\label{Gammak}
\Gamma_k = \{x\in\cee\mid |z_{k}(x)|=|z_{k+1}(x)|\},\qquad k\in [0:p-1].
\end{equation}
It turns out that $\Gamma_k$ is a finite union of line segments on the star
$S_+$ if $k$ is even and $S_-$ if $k$ is odd: see Fig.~\ref{fig:plotLimit} and
Theorem~\ref{theorem:Gammak:star}. 
The next lemma shows that $\Gamma_k$ is rotationally invariant.

\begin{lemma}\label{lemma:rotsym} (Rotational symmetry:) 
With $\om:=\exp(2\pi\ir/(p+1))$, we have $f(z,\om x)=\om^r f(\om^{r}z,x)$.
Hence, for any $x\in\cee$ the sets $(z_k(\om x))_{k=0}^p$ and
$(\om^{-r}z_k(x))_{k=0}^p$ are equal up to permutation, and each set $\Gamma_k$
is invariant under rotations over $2\pi/(p+1)$.
\end{lemma}

\begin{proof} Recalling \eqref{symbol}--\eqref{def:Sshift}, it is easy to
see that $D^{-1}F(z,\om x)D = \om F(\om^{r} z,x)$ where
$D:=\diag(1,\om,\om^2,\ldots,\om^{r-1})$. This implies the lemma.
\end{proof}

For any $k\in [0:p-1]$, define the measure
\begin{equation} \label{measure:k} \ud\mu_k(\lam) = \frac{1}{2\pi \ir}\frac 1r \sum_{j=0}^{k}\left(
\frac{z_{j+}'(\lam)}{z_{j+}(\lam)}-\frac{z_{j-}'(\lam)}{z_{j-}(\lam)}
\right)\ud\lam
\end{equation}
supported on $\Gamma_k$. Here the prime denotes the derivative with respect to
$\lam$, and $\ud\lam$ denotes the complex line element on each line segment of
$\Gamma_k$, according to some fixed orientation of $\Gamma_k$. Moreover,
$z_{j+}(\lam)$ and $z_{j_-}(\lam)$ are the boundary values of $z_j(\lambda)$
obtained from the $+$-side and $-$-side respectively of $\Gamma_k$, where the
$+$-side ($-$-side) is the side that lies on the left (right) when moving
through $\Gamma_k$ according to its orientation. It turns out that $\mu_k$ is a
positive measure (obviously independent of the orientation given to
$\Gamma_{k}$) with total mass \cite[Sec.~4]{Del}
\begin{equation}\label{muk:totalmass}
\mu_k(\Gamma_k) = \frac{p-k}{p},\qquad k\in [0:p-1].
\end{equation}

\smallskip The measures $(\mu_k)_{k}$ are the minimizers to an equilibrium
problem that we now describe. For any measures $\mu,\nu$ on $\cee$ define their
mutual logarithmic energy as
\begin{equation*}\label{defmutualenergybis}
I(\mu,\nu) = \int_{}\int_{}\ \log\frac{1}{|x-y|} \ \ud\mu(x)\ \ud\nu(y).
\end{equation*}
The logarithmic energy of the measure $\mu$ is defined as $I(\mu)=I(\mu,\mu)$.

We call a vector of positive measures $\vec{\nu} = (\nu_{0},\ldots,\nu_{p-1})$
\emph{admissible} if $\nu_k$ has finite logarithmic energy, $\nu_k$ is
supported on $\Gamma_k$, and $\nu_k$ has total mass $\nu_k(\Gamma_k) =
\frac{p-k}{p}$, $k\in [0:p-1]$. The \emph{energy functional} $J$ is defined by
\begin{equation}\label{energyfunctional} J(\vec{\nu}) = \sum_{k=0}^{p-1}
I(\nu_k) - \sum_{k=0}^{p-2} I(\nu_k,\nu_{k+1}).
\end{equation}
The \emph{(vector) equilibrium problem} is to minimize the energy functional
\eqref{energyfunctional} over all admissible vectors of positive measures
$\vec{\nu}$. The equilibrium problem has a unique solution which is given by
the measures $\mu_k$ in \eqref{measure:k}, see \cite{Del}.




\section{Statement of results}

\subsection{Limiting zero distribution of Riemann-Hilbert minors}
\label{subsection:RHlim}

Denote the normalized zero counting measure of $B_{k,n}$, $k\in [0:p-1]$, by
\begin{align} \label{counting:measure}
\mu_{k,n} := \frac{1}{n}\sum_{x\mid B_{k,n}(x)=0}\delta_x,
\end{align}
where $\delta_x$ is the Dirac measure at $x$ and each zero is counted according
to its multiplicity. Lemma~\ref{lemma:RHminor:degree} shows that $\mu_{k,n}$
has total mass at most $(p-k)/p$. Now we state our first main theorem.

\begin{theorem}\label{theorem:muk:RHminor} 
Assume we have asymptotically periodic recurrence coefficients
\eqref{periodic:asy:twodiag}, and define $\mu_{k,n}$, $\mu_k$ as in
\eqref{counting:measure} and \eqref{measure:k}. Then for any $k\in [0:p-1]$,
the measures $\mu_{k,n}$ weakly converge to the measure $\mu_{k}$ on
$\Gamma_{k}$ as $n\to\infty$. This means that
\begin{equation}\label{weakcvg}
\lim_{n\to\infty}\int \phi(x)\ud\mu_{k,n}(x) = \int \phi(x)\ud\mu_k(x)
\end{equation}
for any bounded continuous function $\phi$.
\end{theorem}

Theorem~\ref{theorem:muk:RHminor} will be proved in
Section~\ref{section:Poincare} with the help of a \lq normal family\rq\
estimate for the ratio of two RH minors (Section~\ref{section:normal}), and
using the generalized Poincar\'e theorem. In fact, we will use a multi-column
version of the generalized Poincar\'e theorem
(Lemma~\ref{lemma:poincaretheorem2}). This approach yields not only weak
asymptotics but also ratio asymptotics for the RH minors,
as we explain in Section~\ref{section:Poincare}, see e.g.\ \eqref{ratioasy:Bkn}
or \eqref{geneig:Nikhier}. Moreover, we will see that
Theorem~\ref{theorem:muk:RHminor} remains valid with $B_{k,n}$ replaced by more
general RH minors (Remark~\ref{remark:weak:RHminors}).

We point out that Theorem~\ref{theorem:muk:RHminor} for $k=0$ could also be
obtained from the normal family arguments in \cite{BDK}, taking into account
the interlacing relations for the zeros of $Q_n$. 

Theorem~\ref{theorem:muk:RHminor} shows that the limiting zero distribution of
each Riemann-Hilbert minor $B_{k,n}$ exists and that the limiting measures are
the minimizers to a vector equilibrium problem. We have reason to believe that
a similar conclusion may hold for more general classes of multiple orthogonal
polynomials. This may be an interesting topic for further research.

\subsection{Star-like structure of $\Gamma_k$}

\begin{theorem}\label{theorem:Gammak:star}
Assume that \eqref{periodic:asy:twodiag} holds. Fix $k\in [0:p-1]$ and define
$\Gamma_k$ \eqref{Gammak} and also
\begin{equation}\label{Gammak:til}
\wtil\Gamma_k = \Gamma_k^{p+1} := \{x^{p+1}\mid x\in\Gamma_k\}.
\end{equation} Then:
\begin{itemize}
\item[(a)] $\wtil\Gamma_{k}$ is part of $\er_+$ (or $\er_-$) if $k$ is
even (or odd respectively). \item[(b)] $\wtil\Gamma_k$ is the union of $n_k$
intervals $I_{j,k}$:
\begin{equation}\label{nk:union}\wtil\Gamma_k = \bigcup_{j=1}^{n_k}
I_{j,k},\qquad\textrm{with}\ n_k= \left\lceil \frac{k+1}{p+1}r \right\rceil
-\left\lfloor \frac{kr}{p} \right\rfloor,
\end{equation}
with $x\mapsto\lceil x\rceil$ and $x\mapsto\lfloor x\rfloor$ denoting the \lq
ceiling\rq\ and \lq floor\rq\ functions. The intervals $I_{j,k}$, $j\in
[1:n_k]$ are pairwise disjoint except maybe for common endpoints.
\item[(c)] The following conditions imply that $\wtil\Gamma_k$ contains $0$ or
$\infty$:
\begin{equation}\label{Gammak:infty}
\frac{k+1}{p+1}r\not\in\enn\Rightarrow 0\in\wtil\Gamma_k,\qquad
\frac{kr}{p}\not\in\enn\cup\{0\}\Rightarrow (-1)^{k}\infty\in\wtil\Gamma_k.
\end{equation}
\end{itemize}
\end{theorem}

Theorem~\ref{theorem:Gammak:star} was formulated for the sets $\wtil\Gamma_k$
in \eqref{Gammak:til}. In terms of the original sets $\Gamma_k$, it implies
that $\Gamma_k$ lies on one of the two stars $S_+$ and $S_-$ in \eqref{stars},
depending on whether $k$ is even or odd respectively. Recall that $\Gamma_{k}$
is rotationally invariant (Lemma~\ref{lemma:rotsym}).

Theorem~\ref{theorem:Gammak:star} will be proved in
Section~\ref{subsection:Gammak:star}. In the case $r=1$ it was already obtained
by Aptekarev-Kalyagin-Saff~\cite{AKS}; note that in that case we have
$n_0=\cdots=n_{p-1}=1$, $\wtil\Gamma_0=[0,c]$ for a certain $c>0$, and
$\wtil\Gamma_k=(-1)^{k}\er_+$ for $k\in [1:p-1]$.

\begin{remark} As mentioned in the statement of the theorem, the intervals $I_{j,k}$, $j\in [1:n_k]$
in \eqref{nk:union} are pairwise disjoint except possibly for common endpoints.
We believe that such common endpoints are rare, in the sense that for a
sufficiently \lq generic\rq\ choice of the parameters $b_k>0$, $k\in [0:r-1]$,
all the endpoints of the intervals are distinct.
\end{remark}

\begin{remark} Suppose $p=2$. Then we have two values $k=0$ and $k=1$, and
\eqref{nk:union} reduces to
$$ n_0 = \left\lceil \frac{r}{3} \right\rceil,\qquad n_1 =
\left\lceil \frac{2r}{3}\right\rceil-\left\lfloor \frac{r}{2} \right\rfloor.
$$
For example, if the period $r=6$ then we have $n_0=2$ and $n_1=1$. Note that
this is the same setting as in~\cite{LopezGarcia}, but in the
latter paper there is an additional structure
on the $b_k>0$ which implies that the two intervals of $\wtil\Gamma_0$ are
tangent (and contain the origin), so that $\wtil\Gamma_0$ consists of a single
contiguous interval in that case. If the period $r=8$ then we have $n_0=3$ and
$n_1=2$: see Fig.~\ref{fig:plotLimit}.
\end{remark}



\subsection{Generalized eigenvalues and interlacing}
\label{subsection:geneig}

To obtain interlacing relations for the zeros of RH minors, we will use an
alternative representation via generalized eigenvalue determinants that we now
describe. To the recurrence \eqref{recurrencerel} we associate the Hessenberg
operator $H=(H_{i,j})_{i,j=0}^{\infty}$ with entries
\begin{equation}\label{H:entries:twodiag}
\left\{\begin{array}{llll} H_{j-1,j} &=& 1,&j\geq 1,\\
H_{j+p,j}&=&a_{j},& j\geq 0,\\
H_{i,j}&=&0,& \textrm{otherwise}.
\end{array}\right.
\end{equation}
We refer to $H$ as a \emph{two-diagonal Hessenberg matrix}. We denote with
$H_n$ its $n\times n$ leading principal submatrix:
\begin{equation}\label{H:matrix:twodiag} H_n =
(H_{i,j})_{i,j=0}^{n-1} =
\begin{pmatrix}
0 &  1  & & &  0 \\
\vdots & \ddots & \ddots & &   \\
a_0 &  & \ddots & \ddots &  \\
 & \ddots & & \ddots & 1 \\
0   & & a_{n-p-1} & \ldots & 0
\end{pmatrix}_{n\times n}.
\end{equation}
The recurrence relation \eqref{recurrencerel} can be written in matrix-vector
form as
\begin{equation}\label{recurrence:mxvector}
x(Q_0(x),Q_1(x),\ldots)^T = H (Q_0(x),Q_1(x),\ldots)^T,
\end{equation}
where the superscript ${}^T$ denotes the transpose. This implies easily that
$Q_n(x)=\det(x I_n-H_n).$
So the eigenvalues of $H_n$ are the zeros of $Q_n$.

For $k\in [0:p]$ we define the polynomial $P_{k,n}(x)$ as the determinant of
the submatrix of $H_n-x I_n$ obtained by skipping the first $k$ rows and the
last $k$ columns. Thus
\begin{equation}\label{geneig:Pkn}
P_{k,n}(x) 
= \det\begin{pmatrix} 0 & \ldots & -x & 1  & & \\
\vdots & \ddots &  & \ddots & \ddots & & \\
a_0 &  & \ddots & & \ddots & \ddots & \\
 & \ddots &  & \ddots & & \ddots & 1 \\
 & & \ddots & & \ddots & & -x \\
 &  &  & \ddots & & \ddots & \vdots \\
 & & & & a_{n-p-1} & \ldots & 0
\end{pmatrix}_{(n-k)\times (n-k)}.
\end{equation}
The $k$th \emph{generalized eigenvalues} of $H_n$ are the numbers $x\in\cee$
such that $P_{k,n}(x)=0$. For $n\leq k$ we set $P_{k,n}(x)\equiv 1$. Note that
for $k=p$ we have $P_{p,n}(x)\equiv a_{0}\cdots a_{n-p-1}>0$.

\begin{lemma}\label{lemma:geneig:degree}
For any $k\in [0:p]$, the polynomial $P_{k,n}(x)$ has degree
\begin{equation*}
\deg P_{k,n}\leq \frac{p-k}{p}(n-k).
\end{equation*}
\end{lemma}

\begin{proof}
This follows by a simple combinatorial argument; see e.g.\ \cite[Proof of
Prop.~2.5]{DK}.
\end{proof}

Lemma~\ref{lemma:geneig:degree} could be refined using the combinatorial
formulas in Section~\ref{section:normal}. This leads to an exact formula for
$\deg P_{k,n}$, depending on $n\ \textrm{mod}\ p$. We will not go into this
issue here.


The fact of the matter is the following.

\begin{proposition}\label{prop:geneigRH} (Generalized eigenvalues versus RH minors:)
For any $k\in [0:p]$, 
\begin{equation}\label{relGeneigRHminor}
B_{k,n}(x) = (-1)^{n(k+1)-\binom{k+1}{2}}c_{k}  P_{k,n}(x),
\end{equation}
cf.~\eqref{RHminor:principal:def}, where the constant $c_{k}$ depends only on the first $k$ moments of the
measures $\nu_1,\ldots,\nu_{k}$:
\begin{equation}\label{constant:geneigRH} c_{k}=(-1)^{k}
\left(\int \ud\nu_1(t)\right)\left(\int Q_1(t)\ \ud\nu_2(t)\right)\cdots
\left(\int Q_{k-1}(t)\ \ud\nu_{k}(t)\right).
\end{equation}
\end{proposition}
Note that in \eqref{relGeneigRHminor} and \eqref{constant:geneigRH}, we
should understand $\binom{1}{2}=0$ and $c_{0}=1$.

We point out that Prop.~\ref{prop:geneigRH} remains valid for arbitrary banded
Hessenberg operators, that is, for matrices $H=(H_{i,j})_{i,j=0}^{\infty}$
defined by
\begin{equation}\label{H:entries}
\left\{\begin{array}{ll} H_{j-1,j} = 1,& j\geq 1,\\ H_{j+k,j} = a_j^{(k)},&
j\geq 0,\,\, k\in [0:p],\,\,a_{j}^{(k)}\in\cee,\\
H_{i,j}=0,&\textrm{otherwise,}
\end{array}\right.
\end{equation}
so that
\begin{equation}\label{H:matrix}
H_n = (H_{i,j})_{i,j=0}^{n-1} =
\begin{pmatrix}
a_0^{(0)} & 1  & & & 0 \\
\vdots & \ddots & \ddots &  &  \\
a_0^{(p)} &  & \ddots & \ddots &  \\
  & \ddots &  & \ddots & 1 \\
0   & & a_{n-p-1}^{(p)} & \ldots & a_{n-1}^{(0)}
\end{pmatrix}_{n\times n}.
\end{equation}
We will assume that 
$a_j^{(p)}\neq 0,$ for all $j$, 
so the entries on the $p$th subdiagonal of \eqref{H:matrix} are non-zero. We
associate to $H$ the sequence of monic polynomials $(Q_{n})_{n=0}^{\infty}$
satisfying the $(p+2)$-term recurrence relation \eqref{recurrence:mxvector},
i.e.,
\begin{equation}\label{eq:recurrence}
x Q_{n}(x)=Q_{n+1}(x)+a_{n}^{(0)}Q_{n}(x)+a_{n-1}^{(1)} Q_{n-1}(x)+\cdots+a_{n-p}^{(p)}Q_{n-p}(x),\qquad n\geq 0,
\end{equation}
with initial conditions
\begin{equation}\label{initcondQnH}
Q_{-1}\equiv\cdots\equiv Q_{-p}=0,\quad Q_{0}\equiv 1.
\end{equation}
Prop.~\ref{prop:geneigRH} will be a consequence of a result proved in
Prop.~\ref{prop:geneigRH:bis} for the polynomials $Q_{n}$ satisfying
\eqref{eq:recurrence}--\eqref{initcondQnH}, assuming that these polynomials are
\emph{multiple orthogonal} with respect to a system of $p$ measures, see
Section~\ref{section:geneig} for more details.

Prop.~\ref{prop:geneigRH} shows that RH minors can be alternatively represented
as generalized eigenvalue determinants. We now state interlacing relations for
the latter.

\begin{figure}[tbp]
\begin{center}\vspace{-1mm}
        \subfigure{}\includegraphics[scale=0.32]{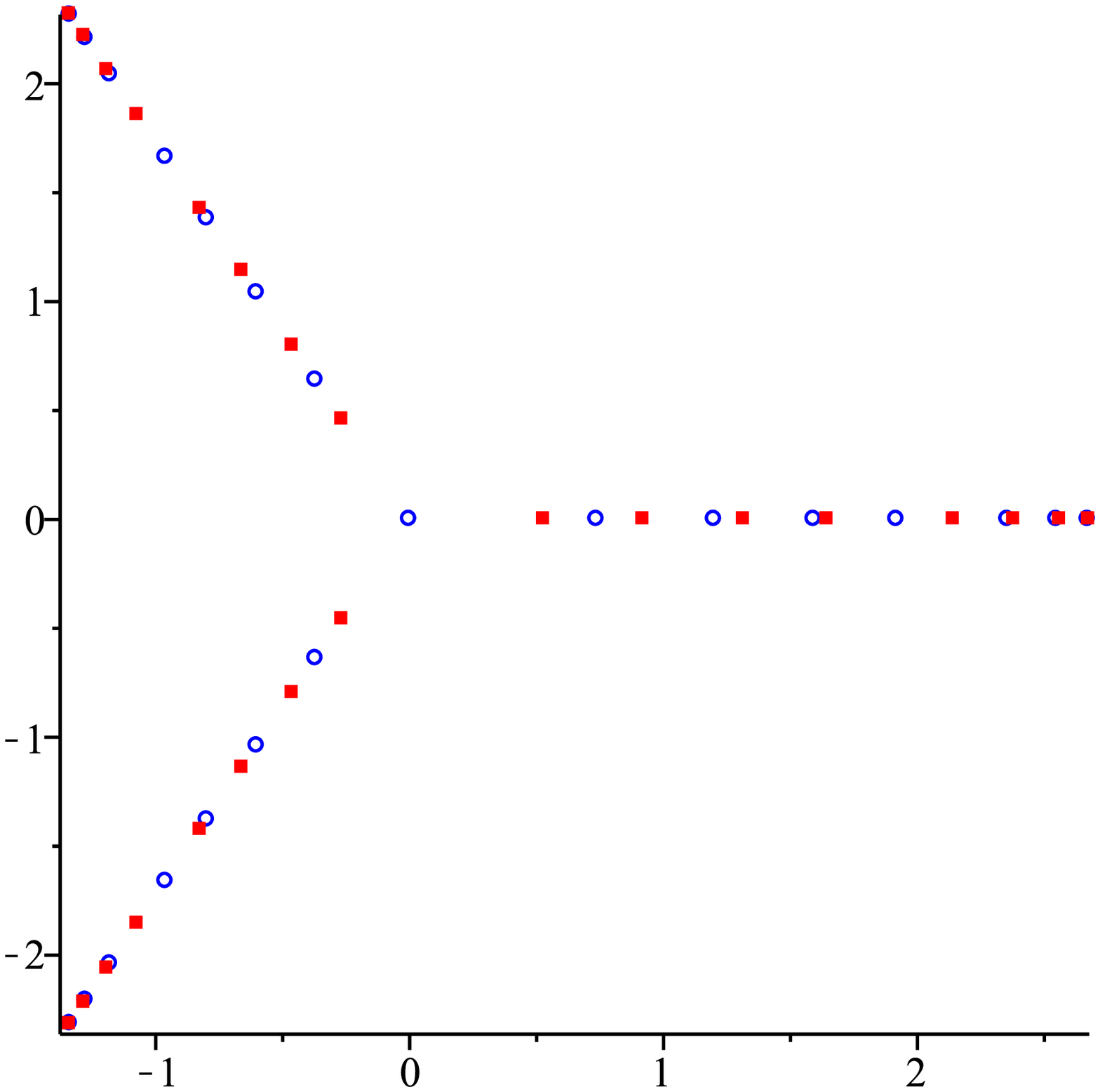}\hspace{15mm}
        \subfigure{}\includegraphics[scale=0.32]{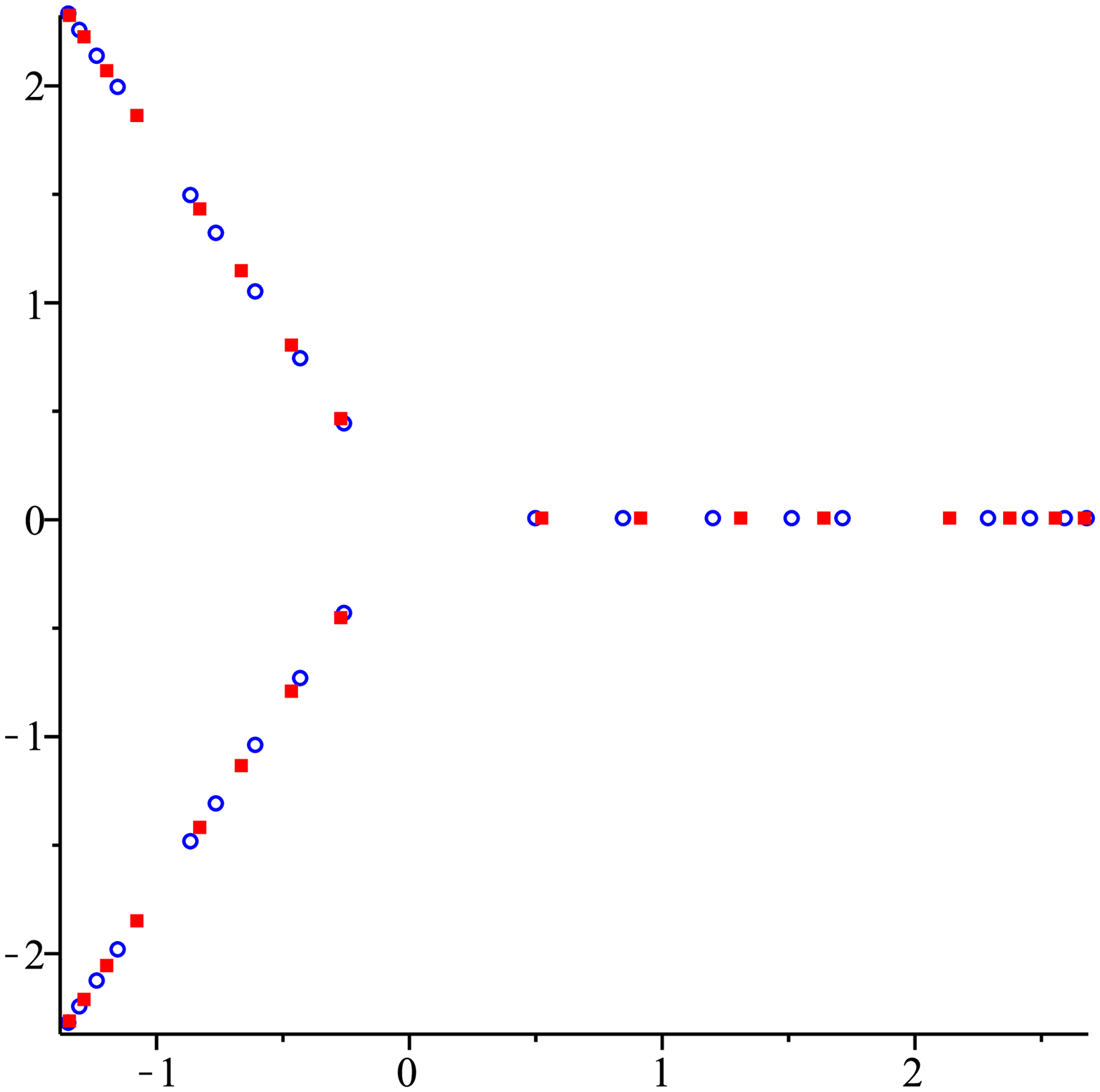}
\end{center}\vspace{-8mm}
\caption{Left picture: zeros of $Q_{23}$ (circles) and $Q_{24}$ (squares).
Right picture: zeros of $Q_{24}$ (squares) and $Q_{27}$ (circles). In these
pictures we have a two-diagonal Hessenberg matrix $H$ as in
\eqref{H:entries:twodiag}--\eqref{H:matrix:twodiag} with $p=2$ and recurrence
coefficients $(a_{0},\ldots,a_{5}) = (3,2,3,5,4,1)$ extended periodically with
periodicity $r=6$.} \label{fig:interlace1}
\end{figure}


\begin{theorem}\label{theorem:interlace} (Interlacing for generalized eigenvalues:)
Let $H$ be a two-diagonal Hessenberg matrix~\eqref{H:entries:twodiag} with
$a_j>0$ for all $j$. Fix $n\in\enn$ and $k\in [0:p-1]$. Then
\begin{itemize}
\item[(a)] We have $P_{k,n}(x) = x^{m_{k,n}}\wtil P_{k,n}(x^{p+1})$, for a
polynomial $\wtil P_{k,n}$ with $\wtil P_{k,n}(0)\neq 0$ and with
\begin{equation}\label{mkn:def} m_{k,n} =\left\{\begin{array}{ll}
(j-k)(k+1),& \textrm{ if } n\equiv j \mod (p+1),\, j\in [k:p],\\[0.3em]
(k-j)(p-k),& \textrm{ if } n\equiv j \mod (p+1),\, j\in [-1:k].
\end{array}\right.
\end{equation}
The zeros of $\wtil P_{k,n}$ all lie in $\er_+$ ($\er_-$) if $k$ is even (odd).

\item[(b)]
Denote the zeros of $\wtil P_{k,n}$ and $\wtil P_{k,n+1}$ as
$(x_i)_{i=1,2,\ldots}$ and $(y_i)_{i=1,2,\ldots}$ respectively, counting
multiplicities and ordered by increasing modulus.
We have the weak interlacing relation
$$ 0<|x_1|\leq |y_{1}|\leq |x_{2}|\leq |y_{2}|\leq \ldots
$$
if $n\equiv j\mod (p+1)$, $j\in [k:p-1]$, and
$$ 0<|y_{1}|\leq |x_{1}|\leq |y_{2}|\leq |x_{2}|\leq \ldots
$$
if $n\equiv j\mod (p+1)$, $j\in [-1:k-1]$.

\item[(c)] 
Let $(x_i)_{i=1,2,\ldots}$ be the zeros of $\wtil P_{k,n}$, as in (b), and let
$(w_i)_{i=1,2,\ldots}$ be the zeros of $\wtil P_{k,n+p+1}$, counting
multiplicities and ordered by increasing modulus. We have
$$ 0<|w_{1}|\leq |x_{1}|\leq |w_{2}|\leq |x_{2}|\leq \ldots.
$$
Note that the moduli can be removed if $k$ is even and replaced by minus signs
if $k$ is odd.
\end{itemize}
\end{theorem}

Theorem~\ref{theorem:interlace} generalizes known results for the standard
eigenvalues $k=0$ \cite{EV,Rom}. The theorem will be proved in
Section~\ref{section:interlacing}, by using the theory of totally positive
matrices and extending the approach of Eiermann-Varga \cite{EV}. See also
Theorems~\ref{theorem:interlace:kl} and \ref{theorem:interlace:gen} below for
related results.

Theorem~\ref{theorem:interlace} is illustrated in Figures~\ref{fig:interlace1}
and \ref{fig:interlace2}.

\begin{figure}[tbp]
\begin{center}\vspace{-1mm}
        \subfigure{}\includegraphics[scale=0.32]{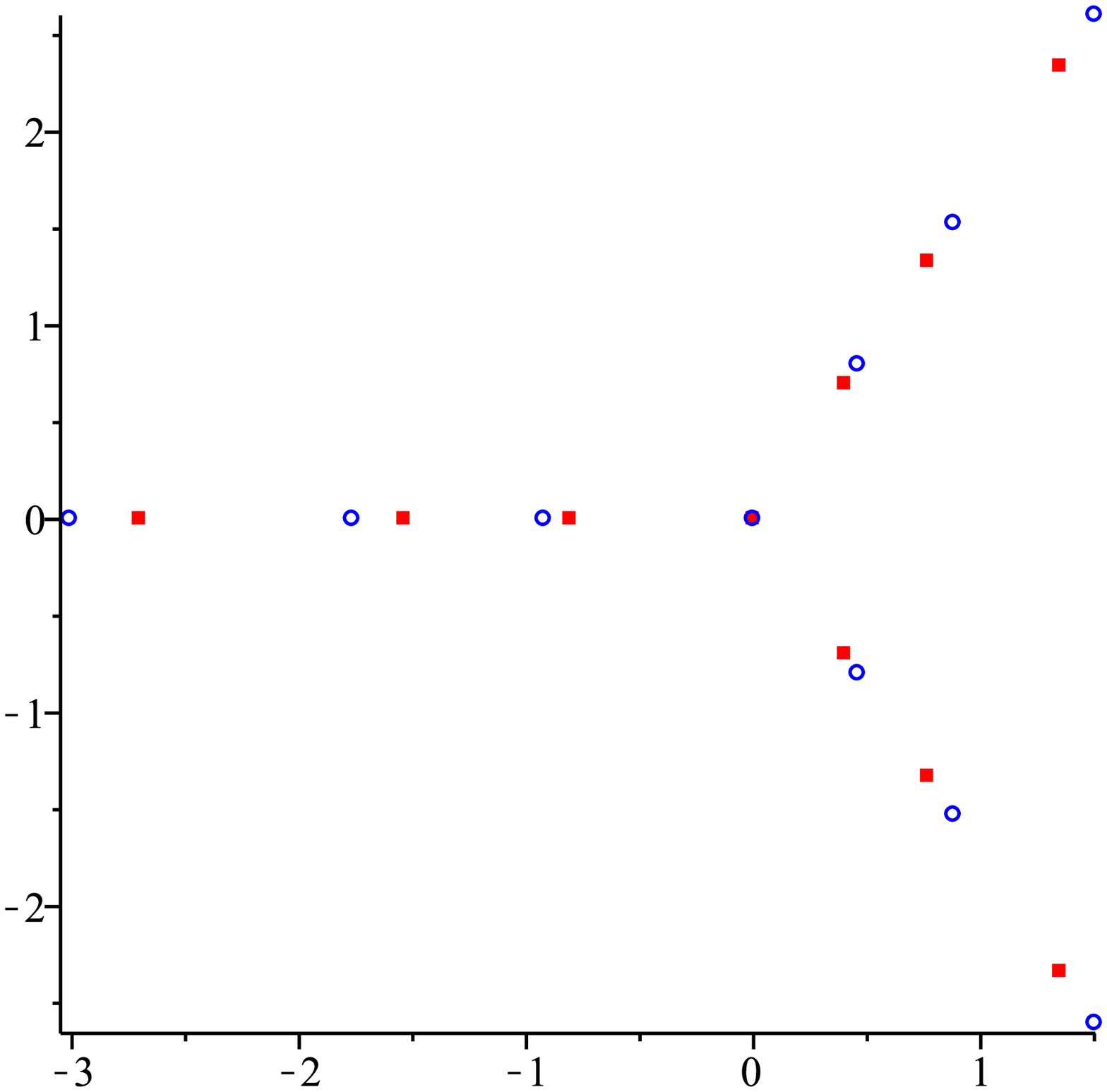}\hspace{15mm}
        \subfigure{}\includegraphics[scale=0.32]{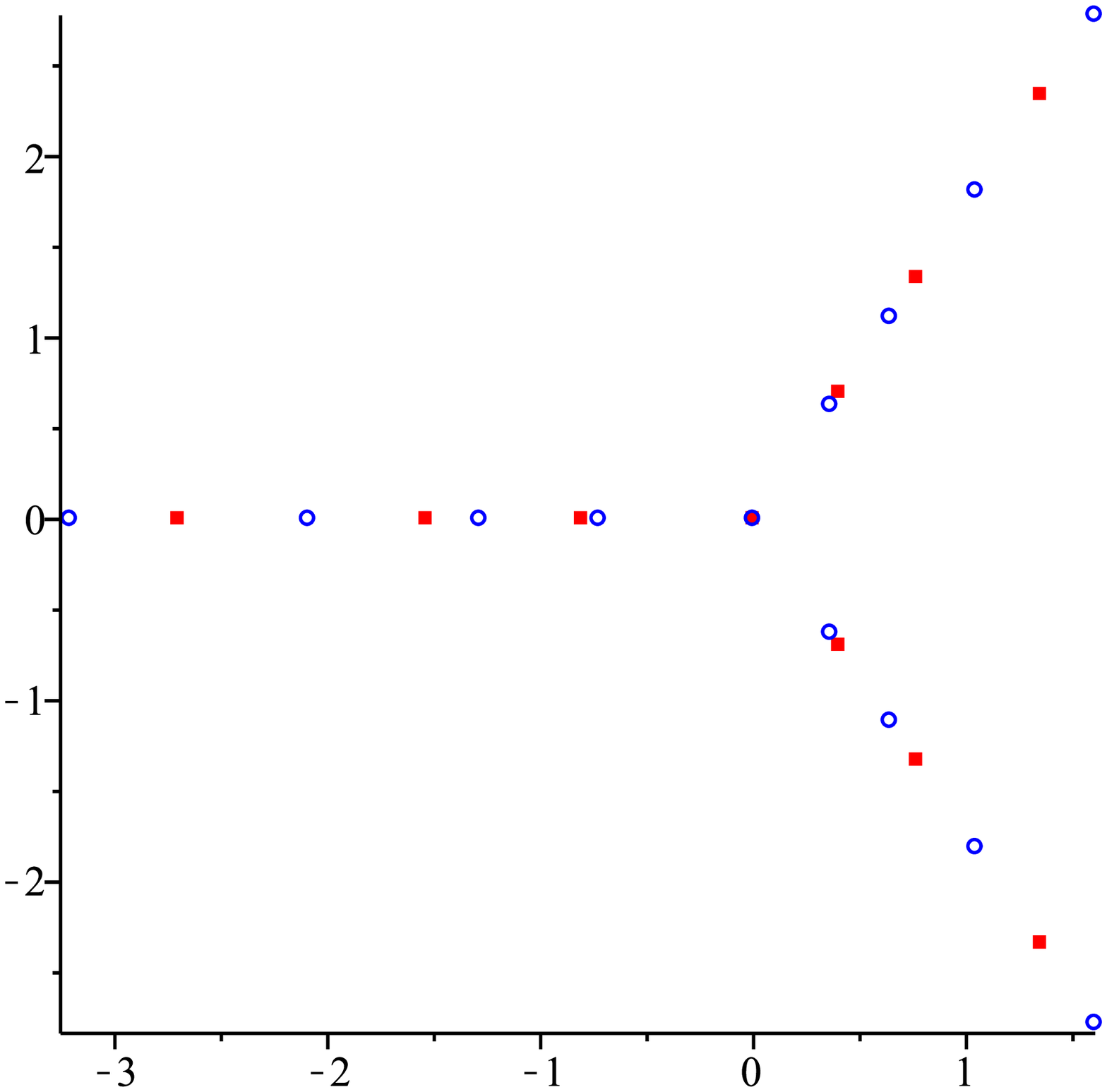}
\end{center}\vspace{-8mm}
\caption{Left picture: zeros of $P_{1,23}$ (circles) and $P_{1,24}$ (squares).
Right picture: zeros of $P_{1,24}$ (squares) and $P_{1,27}$ (circles). The
matrix $H$ is as in Figure~\ref{fig:interlace1}.} \label{fig:interlace2}
\end{figure}


Generalized eigenvalues turn out to be deeply connected to the hierarchy of
functions of the (formal) \emph{Nikishin system} generated by $H$. This will be
the topic of Section~\ref{subsection:Nik}.

\subsection{Connection with Nikishin systems}
\label{subsection:Nik}

Aptekarev-Kalyagin-Saff \cite{AKS} show that, in the trace class
$\sum_{k=0}^{\infty}|a_k-a|<\infty$ and with period $r=1$, the two-diagonal
operator $H$ generates a (formal) \emph{Nikishin system}. These objects are
only formally defined however.

In this paper we will obtain a related result. It will apply to the exactly
periodic case
\begin{equation}\label{periodic:exact:twodiag}
a_{rn+k}=a_k=b_{k},\qquad n\in\enn, \quad k\in [0:r-1].
\end{equation}
We assume without loss of generality that the period $r$ is a multiple of $p$.
We also assume that
\begin{equation}\label{ordering:as}\prod_{n=0}^{r/p-1}
a_{pn}>\prod_{n=0}^{r/p-1} a_{pn+1}>\cdots>\prod_{n=0}^{r/p-1}
a_{pn+(p-1)}.\end{equation} Under these conditions, we will show that the
polynomials $Q_n$ are multiple orthogonal with respect to a \emph{true}
Nikishin system generated by rotationally invariant measures on the stars $S_+$
and $S_-$, coming from measures on $\er_+$ or $\er_-$ with \emph{constant sign}. There
can also be possible point masses at each level of the Nikishin hierarchy.

Nikishin systems formed by measures supported on the real line were introduced
by E.M. Nikishin in~\cite{Nik}. The same definition can be easily adapted to
our context of star-like sets, as we now explain. Compare this definition with
the one given in~\cite[Section 8.1]{AKS}.

\begin{definition}\label{def:Nikishin} Let $\nu_{1},\ldots,\nu_{p}$ be a collection of $p$ complex
measures supported on the set $\Gamma_{0}\cup\mathcal{A}_{0}$, where
$\mathcal{A}_{0}\subset S_{+}\setminus\Gamma_{0}$ is a
discrete set. We say that $(\nu_{1},\ldots,\nu_{p})$ forms a Nikishin system on
$(\Gamma_{0},\ldots,\Gamma_{p-1})$ if for each $k\in [0:p-1],$ there exists a
collection of complex measures $(\nu_{l,k})_{l=k+1}^{p}$ supported on
$\Gamma_{k}\cup\mathcal{A}_{k}$, where $\mathcal{A}_{k}$ is a discrete subset
of $S_{+}\setminus\Gamma_{k}$ (if $k$ is even) or $S_{-}\setminus\Gamma_{k}$
(if $k$ is odd), with the following properties:
\begin{itemize}
\item[(a)] $(\nu_{1},\ldots,\nu_{p})=(\nu_{1,0},\ldots,\nu_{p,0})$.
\item[(b)] If $\ud\nu_{l,k}(x)=g_{l,k}(x) \ud x+\ud \nu_{l,k}^{(s)}(x)$,
$\ud \nu_{l,k}^{(s)}(x)\perp g_{l,k}(x)\ud x$, denotes the Lebesgue
decomposition of $\nu_{l,k}$, $l\in [k+1:p],$ then
    \begin{equation}\label{Nik:property}
    \frac{g_{l,k}(x)}{g_{k+1,k}(x)}=\int\frac{\ud \nu_{l,k+1}(t)}{x-t},\qquad  x\in\Gamma_{k},\quad l\in [k+2:p].
    \end{equation}
\item[(c)] For every $l\in [k+1:p],$ there exists a real measure $\til \nu_{l,k}$
with constant sign (either positive or negative), supported on $\er_{+}$
($\er_{-}$) if $k$ is even (odd), such that
    \begin{equation}\label{nukl:symmetry}
    \ud\nu_{l,k}(t) = t^{k+1-l} \ud\til\nu_{l,k}(t^{p+1}).
    \end{equation}
\end{itemize}
\end{definition}

\begin{remark}
We observe that Nikishin systems possess a hierarchical structure, with the
measures $(\nu_{1},\ldots,\nu_{p})$ forming level $0$ of the hierarchy. The
measure $\nu_{l,k}$ is said to be at the \emph{$k$th level of the Nikishin
hierarchy}. Note that \eqref{nukl:symmetry} implies that for any $k\in
[0:p-1]$, the measure $\nu_{k+1,k}$ is rotationally invariant, and the induced
measure $\til\nu_{k+1,k}$ is real with constant sign. The measures
$\nu_{k+1,k}$ are usually referred to as the \emph{generating measures} of the
Nikishin system. We are implicitly requiring in \eqref{Nik:property} that
$g_{k+1,k}(x)\neq 0$ for all but finitely many $x\in\Gamma_{k}$.
\end{remark}

Our main result is the following.

\begin{theorem}\label{theorem:Nik:property}
Let $H$ be a two-diagonal Hessenberg matrix~\eqref{H:entries:twodiag} with
exactly periodic coefficients $a_{j}>0$ satisfying
\eqref{periodic:exact:twodiag}--\eqref{ordering:as}, where the period $r$ is a
multiple of $p$. Then the orthogonality measures $(\nu_{1},\ldots,\nu_{p})$ in
Theorem~\ref{theorem:Qn:as:mop:0} form a Nikishin system on
$(\Gamma_{0},\ldots,\Gamma_{p-1})$ (Def.~\ref{def:Nikishin}). Moreover, the
star-like sets $(\Gamma_{k})_{k=0}^{p-1}$ are compact and the discrete sets
$(\mathcal{A}_{k})_{k=0}^{p-1}$ are finite.
\end{theorem}

Theorem~\ref{theorem:Nik:property} will be proved in
Section~\ref{section:Nikishin}.

\begin{remark} Theorem~\ref{theorem:Nik:property} was
stated under the condition \eqref{ordering:as}. In general, consider the set
\begin{equation}\label{product:as} \left\{\prod_{n=0}^{r/p-1}
a_{pn},\prod_{n=0}^{r/p-1} a_{pn+1},\cdots,\prod_{n=0}^{r/p-1}
a_{pn+(p-1)}\right\}.\end{equation} Eq.~\eqref{zk:infty:intro:bis} (with $d=p$)
shows that there exists a
permutation $\Pi$ of $[1:p]$ so that
\begin{equation}\label{zk:infty:0}
z_{\Pi(k)}(x) = \left(\prod_{n=0}^{r/p-1}a_{pn+k-1}\right)^{-1}
x^{r/p}(1+o(1)),\qquad k\in [1:p],
\end{equation}
for $x\to\infty$. 
This can also be seen from the derivation of \eqref{zk:infty} in
Section~\ref{section:Nikishin}. As a consequence, if the $p$ numbers in
\eqref{product:as} are pairwise distinct then all the $\Gamma_k$ are bounded.
The converse of the last statement is also true, due to
\cite[Lemma~3.3]{SchmidtSpitzer}. Now if the numbers \eqref{product:as} are
pairwise distinct but ordered in a different way than \eqref{ordering:as}, then
a variant to Theorem~\ref{theorem:Nik:property} holds. We then have an
additional constant or monomial term in the right hand
side of \eqref{Nik:property}. This is due to the fact that 
the constant $\alpha$ in Eq.~\eqref{limit} in Section~\ref{section:Nikishin}
can be nonzero in this case.
\end{remark}

We see that the key to obtaining a \emph{true} (rather than a formal) Nikishin
system is to show that the ratio between the densities \eqref{Nik:property} of
the measures at the different levels of the Nikishin hierarchy are Cauchy
transforms of measures on $S_+$ or $S_-$, associated to real measures with constant
sign on $\er_+$ or $\er_-$. We will establish this requirement via a surprising
connection with RH minors. In particular we will use the interlacing relations
between the zeros of RH minors.

Recall the generalized eigenvalue determinant $P_{k,n}(x)$
from \eqref{geneig:Pkn}. We need a more general definition. For any $1\leq k\leq l\leq p$ we define $P_{k,l,n}(x)$ as
the determinant of the submatrix obtained by skipping rows $0,1,\ldots,k-1$ and
columns $n-l,n-k+1,n-k+2,\ldots,n-1$ of $H_{n}-xI_{n}$. If $l=k$ then we
retrieve our previous definition: $P_{k,k,n}(x)\equiv P_{k,n}(x)$.

In the proof of Theorem~\ref{theorem:Nik:property} we need the following result
on the polynomials $P_{k,l,n}(x)$.

\begin{theorem}\label{theorem:interlace:kl}
(Interlacing for $P_{k,l,n}$ and $P_{k,n}$:) Let $H$ be a two-diagonal
Hessenberg matrix~\eqref{H:entries:twodiag} with $a_j>0$ for all $j$. Fix
$n\in\enn$ and $0\leq k< l\leq p$. Then
\begin{itemize}
\item[(a)] We have $P_{k,l,n}(x) = x^{k-l+m_{k,n}}\wtil P_{k,l,n}(x^{p+1})$,
with $m_{k,n}$ defined in \eqref{mkn:def} and with $\wtil P_{k,l,n}$ a
polynomial whose zeros all lie in $\er_+$ ($\er_-$) if $k$ is even (odd).

\item[(b)] Denote the zeros of $\wtil P_{k,n}(x)$ and $\wtil P_{k,l,n}(x)$ as $(x_{i})_{i=1,2,\ldots}$
and $(y_{i})_{i=1,2,\ldots}$ respectively, ordered by increasing modulus and
counting multiplicities. We have the weak interlacing relation
$$ 0\leq |y_1|\leq |x_1|\leq |y_2|\leq |x_2|\leq
\ldots.
$$
\end{itemize}
\end{theorem}

Theorem~\ref{theorem:interlace:kl} is proved in
Section~\ref{section:interlacing}. The precise way how
Theorem~\ref{theorem:interlace:kl} is used in the proof of
Theorem~\ref{theorem:Nik:property} will be explained in
Section~\ref{section:Nikishin}.

\begin{remark} The polynomial $\wtil P_{k,l,n}$ could have one, and at most one, zero at the
origin. This happens precisely when $n\equiv j\mod (p+1)$ for some $j\in
[k:l-1]$.\end{remark}

%

\subsection{Widom-type formula}

In this section we state an exact formula for the polynomials $Q_n$ in the
exactly periodic case~\eqref{periodic:exact:twodiag}. In fact, we prove the
formula for general banded Hessenberg matrices $H$ of the form
\eqref{H:entries}. We say that $H$ is \emph{exactly periodic with period $r$}
if
\begin{equation}\label{periodic:exact}a_{rn+k}^{(j)}=a_k^{(j)}=b_k^{(j)},\qquad n\in\enn,\quad
k\in [0:r-1],\quad j\in [0:p].\end{equation} Recall that we are assuming
$b_{k}^{(p)}\neq 0$ for all $k$. Define the \lq block Toeplitz symbol\rq
\begin{equation}\label{symbol:Hess}
F(z,x) = Z^{-1}+ \sum_{k=0}^p Z^{k}\diag(b_0^{(k)},\ldots,b_{r-1}^{(k)})-xI_r,
\end{equation}
with $Z$ as in \eqref{def:Sshift}. In the case of a two-diagonal Hessenberg
matrix \eqref{H:entries:twodiag} this reduces to \eqref{symbol}. Also define
$f(z,x)=\det F(z,x)$, the roots $z_k(x)$ of the algebraic equation $f(z,x)=0$
as in \eqref{ordering:rootszk}, and the sets $\Gamma_{k}$ as in \eqref{Gammak}.
Prop.~1.1 in \cite{Del} shows that $\Gamma_k$ is a finite union of analytic
arcs. Clearly, \eqref{f:series}--\eqref{fp} remain valid in this setting (with
$b_k$ replaced by $b_k^{(p)}$).


\begin{theorem}\label{theorem:Widomformula2} (Widom-type formula:)
With the above notations, let $x\in\cee$ be such that the solutions $z_{k}(x)$
of the algebraic equation $0=f(z,x)=\det F(z,x)$ are pairwise distinct. Then
for each $n\in\enn\cup\{0\}$ and for each $j\in [0:r-1],$
\begin{equation}\label{Widom:2}
Q_{rn+j}(x) :=\det(xI_{rn+j}-H_{rn+j})=\frac{(-1)^{r+j}}{\mathsf{f}_{p}}\sum_{k=0}^{p}
\frac{\det F^{r-1,j}(z_{k}(x),x)}{\prod_{i=0,i\neq
k}^{p}(z_{k}(x)-z_{i}(x))}\,z_{k}(x)^{-n-1}.
\end{equation}
Here $\mathsf{f}_p$ is defined in \eqref{fp} (with $b_k$ replaced by
$b_k^{(p)}$), and we use the notation $F^{i,j}$ to denote the submatrix of $F$
in \eqref{symbol:Hess} that is obtained by skipping the $i$th row and the $j$th
column, $i,j\in [0:r-1]$. Moreover, for all $i,j,k$, $\det F^{i,j}(z_{k}(x),x)$
is zero for only finitely many $x$.
\end{theorem}

Theorem~\ref{theorem:Widomformula2} will be proved in
Section~\ref{section:widom}, as a consequence of Widom's determinant identity
for block Toeplitz matrices \cite[Section 6]{Widom1}. From \eqref{Widom:2} and
\eqref{ordering:rootszk}--\eqref{Gammak} we also find:

\begin{corollary} \label{cor:strong:exact}
The strong asymptotic formula
\begin{equation*}
\lim_{n\to\infty} Q_{rn+j}(x)z_0(x)^{n+1} =
\frac{(-1)^{r+j}}{\mathsf{f}_p}\frac{\det
F^{r-1,j}(z_0(x),x)}{\prod_{i=1}^{p}(z_0(x)-z_i(x))},\qquad j\in [0:r-1],
\end{equation*}
holds uniformly on compact subsets of $\cee\setminus (\Gamma_0\cup\mathcal A)$
with $\mathcal A$ a finite set.
\end{corollary}

Incidentally, Aptekarev et al.~\cite{AKS} obtain strong asymptotics for $Q_n$ in the trace class
$\sum_{k=0}^{\infty}|a_k-a|<\infty$ ($a>0$) with period $r=1$. By using
Theorem~\ref{theorem:Widomformula2}, it is possible to extend these results  to
higher periods $r$. We will not go into this issue here.

\subsection{Outline of the paper}

The rest of this paper is organized as follows. In Section~\ref{section:geneig}
we prove the connection between RH minors and generalized eigenvalue
determinants and we introduce the concept of a general RH minor
$B^{(n_{0},n_{1},\ldots,n_{k})}(z)$. In Section~\ref{section:interlacing} we
prove interlacing relations for generalized eigenvalues.
Section~\ref{section:normal} contains normal family estimates for the ratio
between two RH minors. The remaining sections deal with asymptotically periodic
coefficients $a_n$. The proof of the Widom-type formula
for $Q_{n}$ in the exactly periodic case is given in Section~\ref{section:widom}.
In Section~\ref{section:Poincare} we
obtain weak and ratio asymptotics for RH minors and we prove
Theorem~\ref{theorem:Gammak:star} on the star-like structure of $\Gamma_k$. In
Section~\ref{section:Nikishin} we prove Theorem~\ref{theorem:Nik:property} on
the connection with Nikishin systems.

\section{Riemann-Hilbert minors and generalized eigenvalues}
\label{section:geneig}

In this section we prove Prop.~\ref{prop:geneigRH} in the general context of
banded Hessenberg operators $H=(H_{i,j})_{i,j=0}^{\infty}$ defined in
\eqref{H:entries}.

\smallskip

Let $(Q_{n})_{n=0}^{\infty}$ be the sequence of monic polynomials associated to
the operator $H$, i.e., the sequence satisfying
\eqref{eq:recurrence}--\eqref{initcondQnH}. We will assume in this section that
the polynomials $Q_{n}$ are multiple orthogonal with respect to a system of $p$
complex measures $\nu_{1},\ldots,\nu_{p}$ supported on a compact contour
$\Sigma\subset\cee$. This means that for every $j\in [1:p],$
\begin{equation}\label{orthog:Sigma}
\int_{\Sigma}Q_{n}(t)\,t^{m}\ud\nu_{j}(t)=0,\qquad m\in [0:\left\lfloor
\frac{n-j}{p} \right\rfloor].
\end{equation}
Define the second kind functions $\Psi_{n}^{(j)}$ as in \eqref{secondkind:def}.

For later use, we need a more general definition of generalized eigenvalues.
Let $H_{n}=(H_{i,j})_{i,j=0}^{n-1}$. As in \eqref{geneig:Pkn}, we denote by
$P_{k,n}(x)$ the determinant of the matrix obtained by skipping the first $k$
rows and the last $k$ columns of the matrix $H_n-x I_n$. Similarly we could
skip any set of $k$ different columns, not necessarily consecutive.

Let $k\in [0:p]$ and let $(n_0,n_1,\ldots,n_{k})$ be a $(k+1)$-tuple of
positive integers such that
\begin{equation}\label{geneig:ordering} 0\leq n_0<n_1<\ldots<n_k\leq n_0+p.
\end{equation} We define the \emph{generalized eigenvalue determinant associated to
$(n_0,n_1,\ldots,n_{k})$} as
\begin{equation}\label{def:general:geneig}
P^{(n_0,n_1,\ldots,n_{k})}(x) := \det (H_{n_k}-x I_{n_k})^{(0,1,\ldots,k-1;
n_0,n_1,\ldots,n_{k-1})}.
\end{equation}
That is, the polynomial $P^{(n_0,n_1,\ldots,n_k)}(x)$ is the determinant of the
submatrix obtained by skipping rows $0,1,\ldots,k-1$ and columns
$n_0,n_1,\ldots,n_{k-1}$ of $H_{n_k}-xI_{n_k}$. The \emph{generalized
eigenvalues} associated to $(n_0,n_1,\ldots,n_k)$ are the numbers $x\in\cee$
such that $P^{(n_0,n_1,\ldots,n_k)}(x)=0$. In the case $k=0$ we put $n:=n_0$
and we understand $P^{(n)}(x)=\det(H_{n}-x I_{n})=(-1)^{n}Q_{n}(x)$.

By choosing $(n_0,n_1,\ldots,n_{k})$ to be a sequence of consecutive numbers:
$$(n_0,n_1,\ldots,n_{k}) = (n-k,\ldots,n-1,n),$$ we retrieve our earlier
definition $P^{(n-k,\ldots,n-1,n)}(x) \equiv P_{k,n}(x)$. Similarly we can
retrieve $P_{k,l,n}(x)$.


In this section we prove the following connection with Riemann-Hilbert minors.

\begin{proposition}\label{prop:geneigRH:bis}
Let $H=(H_{i,j})_{i,j=0}^{\infty}$ be the banded Hessenberg matrix
\eqref{H:entries} with $a_{j}^{(p)}\neq 0$ for all $j\geq 0$. Assume that the
monic polynomials $(Q_{n})_{n=0}^{\infty}$
\eqref{eq:recurrence}--\eqref{initcondQnH} associated with $H$ satisfy the
multiple orthogonality relations \eqref{orthog:Sigma}, for some complex
measures $\nu_{1},\ldots,\nu_{p}$ supported on $\Sigma\subset \cee$. Let
$\Psi_{n}^{(j)}$ be the second kind functions \eqref{secondkind:def}. For any
$k\in [0:p]$, we have
\begin{equation}\label{geneig:RH}
c_{k}(-1)^{n_0+\ldots+n_{k}}\, P^{(n_0,n_1,\ldots,n_k)}(z) =
B^{(n_{0},n_{1},\ldots,n_{k})}(z),
\end{equation}
where
\begin{equation}\label{B:gendef}
B^{(n_{0},n_{1},\ldots,n_{k})}(z):=\det\begin{pmatrix}
Q_{n_k}(z) & \Psi_{n_k}^{(1)}(z) & \ldots & \Psi_{n_k}^{(k)}(z) \\
\vdots & \vdots & & \vdots \\
Q_{n_1}(z) & \Psi_{n_1}^{(1)}(z) & \ldots & \Psi_{n_1}^{(k)}(z) \\
Q_{n_0}(z) & \Psi_{n_0}^{(1)}(z) & \ldots & \Psi_{n_0}^{(k)}(z)
\end{pmatrix}_{(k+1)\times(k+1)},\end{equation}
and where the constant $c_{k}$ is given in \eqref{constant:geneigRH}.
\end{proposition}

The matrix in the right hand side of \eqref{B:gendef} is again a submatrix of
the RH matrix in \eqref{def:Y} (with $n=n_k$), although not necessarily a
principal submatrix. This follows from \eqref{geneig:ordering}.


\begin{proof}[Proof of Prop.~\ref{prop:geneigRH:bis}]
We prove \eqref{geneig:RH} by verifying that both sides of the equation satisfy
the same recurrence relation. Assume that $n_{k}\geq k+1$. If we apply the
cofactor expansion formula to $P^{(n_0,n_1,\ldots,n_k)}(z)$ along the last row,
we obtain
\begin{equation}\label{geneig:rec1} P^{(n_0,n_1,\ldots,n_k)}(z) = \sum_{j=1}^{p+1} (-1)^{\sigma_{j}}(H_{n_k}-z
I_{n_k})_{n_{k}-1,n_{k}-j}\,\widetilde P^{(n_0,\ldots,n_{k-1},n_k-j)}(z),
\end{equation}
where in the right hand side of \eqref{geneig:rec1}, $A_{i,j}$ denotes the
$(i,j)$ entry of a matrix $A$, the function $\widetilde{P}$ is defined in the
following way:
\[
\widetilde P^{(n_0,\ldots,n_{k-1},n_{k}-j)}(z) := \left\{\begin{array}{ll}
P^{(\til n_0,\til n_1,\ldots,\til n_k)}(z) & \qquad
\textrm{if $n_k-j\not\in\{n_0,\ldots,n_{k-1}\}$} ,\\ 0 &
\qquad\textrm{otherwise},
\end{array}\right.
\]
where $(\til n_0,\til n_1,\ldots,\til n_k)$ is obtained by ordering the entries
of $(n_0,\ldots,n_{k-1},n_k-j)$ increasingly, and where $\sigma_{j}$ is the sum
of the row and column coordinates of the entry $(H_{n_k}-z
I_{n_k})_{n_{k}-1,n_{k}-j}$ in the matrix $(H_{n_k}-z
I_{n_k})^{(0,1,\ldots,k-1; n_0,n_1,\ldots,n_{k-1})}$ (the definition of
$\sigma_{j}$ is used only when $n_k-j\not\in\{n_0,\ldots,n_{k-1}\}$). We also
put $(H_{n_k}-z I_{n_k})_{i,j}:=0$ whenever $j<0$.

To prove \eqref{geneig:rec1} we observe that the submatrix of $(H_{n_k}-z
I_{n_k})^{(0,1,\ldots,k-1; n_0,n_1,\ldots,n_{k-1})}$ obtained by skipping the
row and column occupied by the entry $(H_{n_k}-z I_{n_k})_{n_{k}-1,n_{k}-j}$,
takes the form
$$ \begin{pmatrix}(H_{\widetilde{n}_k}-z I_{\widetilde{n}_k})^{(0,1,\ldots,k-1;
\widetilde{n}_0,\widetilde{n}_1,\ldots,\widetilde{n}_{k-1})} & 0 \\ * & L
\end{pmatrix},
$$
where $L$ is a lower triangular square matrix of size $n_k-1-\wtil n_k$ with
$1$'s on the main diagonal. Hence the determinant of this submatrix equals
$P^{(\widetilde{n}_0,\widetilde{n}_1,\ldots,\widetilde{n}_{k})}(z)$, which
yields \eqref{geneig:rec1}.

Note that the recursion \eqref{geneig:rec1} is completely determined from its
initial condition (determinant of an empty matrix)
\begin{equation}\label{geneigRH:initial1} P^{(0,1,\ldots,k)}(z) \equiv 1.
\end{equation}



It is well-known (and easily checked) that the second kind functions
$\Psi_n^{(k)}$ satisfy exactly the same recursion as the polynomials $Q_n$, in
the sense that
\begin{equation}\label{secondkind:rec}
x\Psi_n^{(k)}(z) =\Psi_{n+1}^{(k)}(z)+a_{n}^{(0)}\Psi_n^{(k)}(z)+a_{n-1}^{(1)}
\Psi_{n-1}^{(k)}(z)+\cdots+a_{n-p}^{(p)}\Psi_{n-p}^{(k)}(z),\qquad n\geq k,
\end{equation}
for any $k\in [1:p]$. The recursion for the functions $\Psi_n^{(k)}(z)$ starts
only from the index $n=k$. Assume that $n_{k}\geq k+1$. Applying
\eqref{eq:recurrence} and \eqref{secondkind:rec} (with $n:=n_k-1$) to the first
row of \eqref{B:gendef} and using the linearity of the determinant, we deduce
that
\begin{equation}\label{Bfunc:rec}
B^{(n_0,n_1,\ldots,n_k)}(z) = \sum_{j=1}^{p+1} (-1)^{1+\tau_{j}}(H_{n_k}-z
I_{n_k})_{n_{k}-1,n_{k}-j}\,\widetilde B^{(n_0,\ldots,n_{k-1},n_k-j)}(z)
\end{equation}
where
\[
\widetilde B^{(n_0,\ldots,n_{k-1},n_{k}-j)}(z) := \left\{\begin{array}{ll}
B^{(\til n_0,\til n_1,\ldots,\til n_k)}(z) & \qquad \textrm{if
$n_k-j\not\in\{n_0,\ldots,n_{k-1}\}$},
\\ 0 & \qquad\textrm{otherwise},
\end{array}\right.
\]
and $\tau_{j}$ is the number of adjacent transpositions that are necessary to
order $(n_{0},\ldots,n_{k-1},n_{k}-j)$ increasingly, e.g. for
$(n_{0},n_{1},n_{2},n_{3}-j)=(1,4,5,3)$ we have $\tau_{j}=2$.

If $(\widetilde{n}_{0},\widetilde{n}_{1},\ldots,\widetilde{n}_{k})$ is obtained
by ordering $(n_{0},\ldots,n_{k-1},n_{k}-j)$ increasingly, then obviously
\begin{equation}\label{eq:ordering}
n_{0}+n_{1}+\cdots+n_{k}-(\widetilde{n}_{0}+\widetilde{n}_{1}+\cdots+\widetilde{n}_{k})=j.
\end{equation}

From \eqref{geneig:rec1} and \eqref{eq:ordering} we have
\begin{multline}\label{eq:4}
c_{k} (-1)^{n_{0}+n_{1}+\cdots+n_{k}}\,P^{(n_0,n_1,\ldots,n_k)}(z)
\\
=\sum_{j=1}^{p+1} (H_{n_k}-z
I_{n_k})_{n_{k}-1,n_{k}-j}(-1)^{\sigma_{j}}(-1)^{j}\,(-1)^{\widetilde{n}_{0}
+\widetilde{n}_{1}+\cdots+\widetilde{n}_{k}}\,c_{k}\,\widetilde
P^{(n_0,n_1,\ldots,n_k-j)}(z).
\end{multline}
We claim that
\begin{equation}\label{reltausigma}
(-1)^{1+\tau_{j}}=(-1)^{\sigma_{j}+j}.
\end{equation}
Let $j\geq 1$ and assume that $n_{k-1}<n_{k}-j$. Then $\tau_{j}=0$ so the
left-hand side of \eqref{reltausigma} is $-1$. Now, if $j$ is even then
$\sigma_{j}$ is odd and vice-versa. So \eqref{reltausigma} holds in this case.
Now let $j_1$ be such that $(-1)^{1+\tau_{j_{1}}} =(-1)^{\sigma_{j_{1}}+j_{1}}$
and $n_{k}-j_{1}=n_{l}+1$ for some $l\in [0:k-1]$. Assume further that the next
value of $j$ greater than $j_{1}$ for which $n_{k}-j\neq n_{i}$ for all $i$ is
$j=j_{1}+q+2$, $q\geq 0$. These assumptions imply that
$\tau_{j}=\tau_{j_{1}}+q+1$, $\sigma_{j}=\sigma_{j_{1}}+1$, and therefore
$(-1)^{1+\tau_{j_{1}}} =(-1)^{\sigma_{j_{1}}+j_{1}}$ implies that
\eqref{reltausigma} is valid for $j$. This completes the justification of
\eqref{reltausigma}.

It follows from \eqref{Bfunc:rec}, \eqref{eq:4} and \eqref{reltausigma} that
for each $k$, the functions $B^{(n_{0},n_{1},\ldots,n_{k})}$ and $c_{k}
(-1)^{n_{0}+n_{1}+\cdots+n_{k}}\,P^{(n_0,n_1,\ldots,n_k)}$ satisfy the same
recurrence relations. The recursion \eqref{Bfunc:rec} is also completely
determined from its initial condition $B^{(0,1,\ldots,k)}$, which is
\begin{multline*}B^{(0,1,\ldots,k)}(z)=
\det\begin{pmatrix}
Q_{k}(z) & \Psi_{k}^{(1)}(z) & \ldots & \Psi_{k}^{(k)}(z) \\
\vdots & \vdots & & \vdots \\
Q_{0}(z) & \Psi_{0}^{(1)}(z) & \ldots & \Psi_{0}^{(k)}(z)
\end{pmatrix} \\ = \det\begin{pmatrix}
z^k+O(z^{k-1}) & O(z^{-2}) & O(z^{-2}) & \ldots & O(z^{-2}) \\
O(z^{k-1}) & O(z^{-2}) & O(z^{-2}) & \ldots & C_k z^{-1}+O(z^{-2}) \\
\vdots & \vdots & \vdots & & \vdots \\
O(z^2) & O(z^{-2}) & O(z^{-2}) & \ldots & O(z^{-1}) \\
O(z) & O(z^{-2}) & C_2 z^{-1}+O(z^{-2}) & \ldots & O(z^{-1}) \\
O(1) & C_1 z^{-1}+O(z^{-2}) & O(z^{-1}) & \ldots & O(z^{-1})
\end{pmatrix}
\end{multline*}
with $C_j:=\int Q_{j-1}(t)\ud\nu_j(t)$. Expanding this determinant as a signed
sum over all permutations of $(0,1,\ldots,k)$, we see that all the terms in
this sum are $O(z^{-1})$ except for the one that corresponds to the permutation
$(0,k,\ldots,2,1)$:
$$ B^{(0,1,\ldots,k)}(z)=
(-1)^{\binom{k}{2}} C_1C_2\ldots C_{k}+O(z^{-1}).
$$
Since we already know by Lemma \ref{lemma:RHminor:degree} that the determinant
in the left hand side is a polynomial in $z$, the $O(z^{-1})$ term in the right
hand side vanishes. The value of $c_{k}$ was chosen so that
\[
c_{k} (-1)^{0+1+\cdots+k}\,P^{(0,1,\ldots,k)}(z)=B^{(0,1,\ldots,k)}(z),
\]
so the two initial conditions are the same and this concludes the proof of
\eqref{geneig:RH}.
\end{proof}

\section{Interlacing of generalized eigenvalues}
\label{section:interlacing}

In this section we prove Theorems~\ref{theorem:interlace} and
\ref{theorem:interlace:kl} on the interlacing of generalized eigenvalues. To
this end we use some results on totally positive matrices.

\subsection{Generalized eigenvalues of totally positive matrices}

A matrix $A\in\cee^{n\times m}$ is called \emph{totally positive (TP)} if the
determinant of any square submatrix of $A$ is positive, i.e.,
\begin{equation}\label{def:TNN} \det A(K,L)>0,\end{equation} for any index
sets $K\subset [0:n-1]$, $L\subset [0:m-1]$ of the same cardinality $|K|=|L|$,
where we write $A(K,L)$ for the submatrix of $A$ with rows and columns indexed
by $K$ and $L$, respectively. We emphasize that in the submatrix $A(K,L)$ the
rows and columns are positioned in the same order given in $A$. \emph{Fekete's
criterion} asserts that a sufficient condition for $A$ to be TP is that
\eqref{def:TNN} holds for all index sets $K$ and $L$ formed by
\emph{consecutive} indices, i.e., $K=\{r,r-1,\ldots,r-q+1\}$ and
$L=\{c,c-1,\ldots,c-q+1\}$ with $q:=|K|=|L|$ and for suitable integers $r,c$.

The matrix $A$ is called \emph{totally nonnegative (TNN)} if we have the
inequality $\geq$ in \eqref{def:TNN}:
\begin{equation}\label{def:TP} \det A(K,L)\geq 0,\end{equation} for all index
sets $K\subset [0:n-1]$, $L\subset [0:m-1]$ with $|K|=|L|$. It is well known
that TP matrices are dense in the class of TNN matrices. Moreover, the class of
TP (or TNN) matrices is closed under matrix multiplication.

The theory of eigenvalues for TP matrices was developed by Gantmacher-Krein
\cite{GantmacherKrein}. They showed that the eigenvalues of an $n\times n$ TP
matrix are all positive and distinct and that they strictly interlace with
those of its principal $(n-1)\times (n-1)$ submatrix. We need the following
analogue for \emph{generalized} eigenvalues of TP matrices, which are again
defined as in Section~\ref{subsection:geneig}.

\begin{proposition}\label{prop:geneigTP} (Generalized eigenvalues of TP matrices:)
Fix $0\leq k<n$ and let $M\in\cee^{(n+k)\times (n+k)}$ be a TP matrix. Then the
$k$th generalized eigenvalues of $M$ are simple, lie in $(0,\infty)$ if $k$ is
even and lie in $(-\infty,0)$ if $k$ is odd. The number of $k$th generalized
eigenvalues of $M$ is $n-k$. Moreover, the $k$th generalized eigenvalues of $M$
and its principal leading submatrix $Q\in\cee^{(n+k-1)\times (n+k-1)}$
are strictly interlacing.
\end{proposition}

\begin{proof} Let $N$ be the submatrix of $M$ obtained after
skipping the first $k$ rows and the last $k$ columns of $M$. Thus $N$ is of
size $n\times n$. Partition
\begin{equation}\label{def:M0}
N=\begin{pmatrix} A & B \\ C & D \end{pmatrix}
\end{equation}
with $C$ of size $k\times k$. By definition, the $k$th generalized eigenvalues
of $M$ are the numbers $x\in\cee$ such that
\begin{equation}\label{kthgeneigM}
\det\begin{pmatrix} A & B-x I_{n-k} \\
C & D
\end{pmatrix}=0.
\end{equation}
The assumption that $M$ is totally positive implies in particular that all the
entries of $N$ are positive. We bring $N$ to a simpler form by means of
elementary row operations. Denote
$$ G_j := I_n-\frac{N_{j,0}}{N_{j+1,0}}E_{j,j+1},
$$
where for $j,l\in [0:n-1]$, $N_{j,l}$ denotes the $(j,l)$ entry of $N$, and
where $E_{j,l}$ is the elementary matrix of size $n\times n$ whose $(j,l)$
entry is $1$ and which has all its other entries equal to zero. Multiplying $N$
on the left with the matrix $G_j$ amounts to subtracting from row $j$,
$N_{j,0}/N_{j+1,0}$ times row $j+1$. This operation eliminates the $(j,0)$
entry of $N$.

We also define
$$ \wtil G_j := \left\{\begin{array}{ll}
I_n+\frac{N_{j,0}}{N_{j+1,0}}E_{j+k,j+k+1}, & \textrm{if $j+k+1<n$,}\\
I_n, &\textrm{otherwise.}
\end{array}\right.
$$
The matrices $G_{j}$ and $\wtil G_{j}$ satisfy
\begin{equation}\label{propmatGj}
G_{j} \begin{pmatrix} 0 & I_{n-k} \\
0 & 0
\end{pmatrix} \wtil G_{j}=\begin{pmatrix} 0 & I_{n-k} \\
0 & 0
\end{pmatrix},\qquad j\in [0:n-2],
\end{equation}
where we use the same decomposition in blocks as in \eqref{def:M0}.

Consider the transformed matrix
\begin{equation}\label{def:M1tilde}
\wtil N^{(1)}:= G_{n-2}\ldots G_1G_0 N \wtil G_0\wtil G_1\ldots \wtil G_{n-2}.
\end{equation}
The matrix $\wtil N^{(1)}$ has all its entries in the first column equal to
zero except for the last one, which equals $N_{n-1,0}$. Let $N^{(1)}$  be the
matrix obtained by removing the first column and the last row of $\wtil
N^{(1)}$. Using \eqref{propmatGj}, we deduce that the $k$th generalized
eigenvalues of $M$ are the points $x\in\cee$ such that
\begin{equation}\label{kthgeneigM2}
\det\begin{pmatrix} \wtil A & \wtil B-x I_{n-k} \\
\wtil C & \wtil D
\end{pmatrix}=0,
\end{equation}
where
\[
N^{(1)}=\begin{pmatrix} \wtil A & \wtil B \\ \wtil C & \wtil D \end{pmatrix}
\]
with $\wtil C$ of size $(k-1)\times (k-1)$. Observe that compared to
\eqref{kthgeneigM}, the diagonal of $x$'s in \eqref{kthgeneigM2} is closer to
the main diagonal.

We claim that the matrix $-N^{(1)}$ is a TP matrix (note the minus sign). For
convenience we label the rows and columns of $N^{(1)}$ from $0$ to $n-2$
and from $1$ to $n-1$, respectively. 
Let $K=\{r,r-1,\ldots,r-q+1\}\subset [0:n-2]$ and
$L=\{c,c-1,\ldots,c-q+1\}\subset [1:n-1]$ be two index sets. From the fact that
$N$ is TP we have that
$$ \det N(K\cup\{r+1\},L\cup\{0\})>0,$$
i.e., the determinant of the enlarged submatrix obtained by adjoining row $r+1$
and column $0$ to $N(K,L)$ is positive.

Define
$$ \what N :=G_{r}\ldots
G_1G_0N.
$$
It is clear that \begin{equation}\label{interlacing:step2} 0<\det
N(K\cup\{r+1\},L\cup\{0\}) = \det \what
N(K\cup\{r+1\},L\cup\{0\}),\end{equation} where the last equality follows since
the row operations $G_0,\ldots,G_r$ applied to $N$ leave the determinant
invariant.

From the definition of the row operations $G_0,\ldots,G_r$, the submatrix in
the right hand side of \eqref{interlacing:step2} is zero in its first column
except for its last entry. Expanding the determinant along its first column we
therefore see that
$$ \det \what N(K\cup\{r+1\},L\cup\{0\}) = (-1)^{q}N_{r+1,0} \det\what
N(K,L),
$$
which combined with \eqref{interlacing:step2} and the TP property of $N$ yields
\begin{equation}\label{interlacing:step3} (-1)^{q} \det\what N(K,L)>0.
\end{equation}
The property \eqref{interlacing:step3} remains valid with $\what N$ replaced by
the matrix
$$ G_{n-2}\ldots G_{r+1}\what N = G_{n-2}\ldots
G_1G_0N,
$$
since the row operations $G_{r+1},\ldots,G_{n-2}$ applied to $\what N$ leave
the submatrix indexed by rows $K$ and columns $L$ invariant. Since $K$ and $L$
are arbitrary index sets, this implies by Fekete's criterion that the matrix of
size $(n-1)\times(n-1)$,
\begin{equation*}
-(G_{n-2}\ldots G_{1} G_{0} N)([0:n-2],[1:n-1]),
\end{equation*}
is TP. 
This implies in turn that
\[
-\wtil N^{(1)}([0:n-2],[1:n-1])=-N^{(1)}
\]
is also TP, since (cf. \eqref{def:M1tilde}) each of the column operations
$\wtil G_0,\wtil G_1,\ldots,\wtil G_{n-2}$ adds to a column a positive multiple
of the previous column; it is straightforward to verify that such operations
leave the total positivity of a matrix invariant.

By repeating the transformation $N^{(0)}:=N\mapsto N^{(1)}$ $k$ times, we get a
series of matrices $N^{(0)},N^{(1)},N^{(2)},\ldots,N^{(k)}$ so that the $k$th
generalized eigenvalues of $M$ are the (usual) eigenvalues of $N^{(k)}$.
Moreover, the matrix $(-1)^k N^{(k)}$ is TP. One of the assertions of the
Gantmacher-Krein theorem implies then the validity of the first statement of
the Proposition.

Finally, if we apply to the leading principal submatrix $Q$ of $M$ the
operations described above, and denote the resulting series of matrices by
$Q^{(0)},Q^{(1)},Q^{(2)},\ldots,Q^{(k)}$, then $Q^{(j)}$ will be the leading
principal submatrix of $N^{(j)}$ for each $j\in [0:k]$. In particular,
$Q^{(k)}$ is the leading principal submatrix of $N^{(k)}$ and the
Gantmacher-Krein theorem implies the interlacing property we want.
\end{proof}

Since TP matrices are dense in the class of TNN matrices, we obtain

\begin{corollary} \label{prop:geneigTNN} (Generalized eigenvalues of TNN matrices:)
Fix $0\leq k<n$ and let $M\in\cee^{(n+k)\times (n+k)}$ be a TNN matrix. Then
the $k$th generalized eigenvalues of $M$ lie in $[0,\infty)$ if $k$ is even and
lie in $(-\infty,0]$ if $k$ is odd. Denoting the $k$th generalized eigenvalues
of $M$ by $(y_i)_{i=1,2,\ldots}$ and those of its principal leading submatrix
by $(x_i)_{i=1,2,\ldots}$, both of them ordered by increasing modulus and
counting multiplicities, then we have the (weak) interlacing
$$ 0\leq |y_1|\leq |x_1|\leq |y_2|\leq |x_2|\leq \ldots.
$$
Note that the moduli can be removed if $k$ is even and replaced by minus signs
if $k$ is odd.
\end{corollary}

\begin{remark}
In the process of approximating a TNN matrix by a sequence of TP matrices, some
of the generalized eigenvalues may escape to infinity. This will always happen
for the kind of banded matrices we are interested in.
\end{remark}

\subsection{The approach of Eiermann-Varga revisited}
\label{subsection:EV}

In the proofs of Theorems~\ref{theorem:interlace} and
\ref{theorem:interlace:kl}, we will use some ideas from Eiermann-Varga
\cite{EV}, which we now review.

Consider $Q_n(x)=(-1)^n\det(H_n-xI_n)$ with $H_n$ the $n\times n$ two-diagonal
Hessenberg matrix in \eqref{H:matrix:twodiag}. Let $P:[0:n-1]\mapsto [0:n-1]$
be the permutation that sorts the indices according to their residue modulo
$p+1$, in the natural way, that is,
$$(P(0),P(1),\ldots, P(n-1))= (0,p+1,2p+2,\ldots;1,p+2,2p+3,\ldots;\ldots;p,2p+1,3p+2,\ldots).$$
Also denote with $P$ the corresponding permutation matrix such that $P\vece_j =
\vece_{P(j)}$ for $j\in [0:n-1]$. Thus $P$ has in its $j$th column the value
$1$ at position $P(j)$ and zero at all other positions. We consider the
permuted matrix $P^{-1}H_nP-xI_n$. It has a block bidiagonal structure:
\begin{equation}\label{EV:permuted}
P^{-1}H_nP-xI_n = \begin{pmatrix}X_0 & Y_0 \\ 0 & X_1 & Y_1 \\ & & \ddots & \ddots \\
 & & & X_{p-1} & Y_{p-1} \\ Y_p & & & & X_p \end{pmatrix}
\end{equation}
where $X_j = -xI_{n_j}$ with $n_j:=\left\lfloor\frac{n+p-j}{p+1}\right\rfloor$,
where $Y_j$ is the principal truncation of size $n_j\times n_{j+1}$ of the
semi-infinite bidiagonal matrix
\begin{equation}\label{bidiagonal:Yj} Y_{j,\infty} = \begin{pmatrix} 1 & \\ a_{j+1} & 1 \\ &
a_{p+j+2} & 1 \\ & & a_{2p+j+3} & 1 \\ & & & \ddots & \ddots
\end{pmatrix}
\end{equation}
for $j\in [0:p-1]$, and where $Y_p$ is the principal truncation of size
$n_p\times n_{0}$ of the matrix \begin{equation}\label{bidiagonal:Yp} Y_{p,\infty} = \begin{pmatrix} a_0 & 1 & \\
& a_{p+1} & 1 \\ & & a_{2p+2} & 1 \\ & & & \ddots & \ddots
\end{pmatrix}.
\end{equation}

\begin{lemma}\label{lemma:cyclic:product}
\begin{itemize}
\item[(a)] Let $A$ be a matrix of size $n\times n$ as in the right hand side of \eqref{EV:permuted}, with
diagonal blocks $X_j = -xI_{n_j}$, $j\in [0:p]$. Then we have
$$ \det A = (-1)^{n-n_0} x^{n-(p+1)n_0}\det(Y_0Y_1\ldots Y_p-x^{p+1}I_{n_0}).
$$
\item[(b)]
Under the same hypotheses, if we replace $X_0$ by an arbitrary square matrix of
size $n_0\times n_0$, then we have
$$ \det A = (-1)^{n-n_0} x^{n-(p+1)n_0}\det(Y_0Y_1\ldots Y_p+x^{p}X_0).
$$
\end{itemize}
\end{lemma}
\begin{proof}
Use Gaussian elimination with the blocks $X_{1},\ldots,X_{p}$ as pivots to
eliminate the blocks above the main diagonal. After these operations, the block
we obtain in the upper left corner is the matrix $X_{0}+x^{-p}\,
Y_{0}Y_{1}\ldots Y_{p}$. The exponent of $x$ is easily determined.
\end{proof}

Note that the zeros of $Q_{n}$ are the points $x$ where $\det(P^{-1}H_{n}P-x
I_{n})$ vanishes. Now we apply Lemma~\ref{lemma:cyclic:product}(a) to this
determinant. Each of the matrices $Y_0,Y_1,\ldots,Y_p$ in
\eqref{bidiagonal:Yj}--\eqref{bidiagonal:Yp} is bidiagonal with positive
entries and hence TNN. Thus also the matrix product $Y_0Y_1\ldots Y_p$ is TNN
(actually it is \emph{oscillatory} \cite{EV}). This already shows that all the
eigenvalues of $Y_0Y_1\ldots Y_p$ are in $[0,\infty)$. Taking into account the
factor $x^{p+1}$ in Lemma~\ref{lemma:cyclic:product}(a), we then see that the
zeros of $Q_n$ are all located on the star $S_+$.

Carrying on this approach a little further and using the Gantmacher-Krein
theory, one obtains the (strict) interlacing relations for the zeros of the
polynomials $Q_n$ and $Q_{n+1}$, and for $Q_n$ and $Q_{n+p+1}$. See
Eiermann-Varga \cite{EV}. Alternative proofs of the interlacing are in
\cite{HeSaff} and \cite{Rom}.

\subsection{Proofs of Theorems~\ref{theorem:interlace} and \ref{theorem:interlace:kl}}

In this section we prove Theorems~\ref{theorem:interlace} and
\ref{theorem:interlace:kl}. To this end we will rely on
Cor.~\ref{prop:geneigTNN} and the ideas in Section~\ref{subsection:EV}.

We always label rows and columns starting from $0$. We will assume throughout
the proof that $n$ is a fixed multiple of $p+1$ and we fix $k\in [0:p-1]$. The
modifications if $n$ is not a multiple of $p+1$ are discussed at the end of
this section.

\subsubsection{Proof of Theorem~\ref{theorem:interlace}(a) ($n$ a multiple of $p+1$)}
\label{subsubsection:proof1}

Recall that $P_{k,n}(x)$ is the determinant of the matrix obtained by skipping
rows $[0:k-1]$ and columns $[n-k:n-1]$ of $H_n-xI_n$. Applying the permutation
$P$ above, this is equivalent to skipping certain rows and columns of the
permuted matrix \eqref{EV:permuted}. More precisely, $P_{k,n}(x)$ is, up to its
sign, equal to the determinant of the submatrix obtained by skipping the first
row of each of the blocks $X_0,Y_0,X_1,Y_1,\ldots,X_{k-1},Y_{k-1}$ in
\eqref{EV:permuted}, and skipping the last column of each of the blocks
$X_{p},Y_{p-1},X_{p-1},Y_{p-2},\ldots,X_{p-k+1},Y_{p-k}$ (here we are using the
fact that $n$ is a multiple of $p+1$). This can be seen as follows: if we write
the submatrix of $H_{n}-xI_{n}$ as $L(H_{n}-xI_{n})R$, with $L$ and $R$
suitable submatrices of the identity matrix $I_n$ (of size $(n-k)\times n$ and
$n\times (n-k)$, respectively), and similarly the submatrix of
$P^{-1}(H_{n}-xI_{n})P$ as $\widetilde{L}P^{-1}(H_{n}-xI_{n})P\widetilde{R}$,
then $\widetilde{L}=P_{1}LP$ and $\widetilde{R}=P^{-1}R P_{2}$, for some
permutation matrices $P_{1}$ and $P_{2}$ of size $n-k$.

\smallskip Due to the above skipping of rows and columns, some of the diagonal
blocks $X_j$ in \eqref{EV:permuted} will become rectangular instead of square.
Thus we cannot apply Lemma~\ref{lemma:cyclic:product} anymore. Our goal is
therefore to make all the diagonal blocks square again. More precisely, our
goal is to get a matrix as in the right hand side of \eqref{EV:permuted} with
diagonal blocks $X_j' = -xI$, for the identity matrix of certain size, $j\in
[1:p]$, and with $X_0' =
\begin{pmatrix}0 & -xI\\ 0_{k\times k} & 0\end{pmatrix}.$ (We write $X_j',Y_j'$
to distinguish from the blocks $X_j,Y_j$ in \eqref{EV:permuted}.) We will then
be able to apply Lemma~\ref{lemma:cyclic:product}(b).
\\

Recall that in the determinantal formula for $P_{k,n}(x)$ we are skipping rows
$[0:k-1]$ of $H_n-xI_n$. Then in the first $k$ columns of this matrix there is
only one non-zero entry left, being $a_0,\ldots,a_{k-1}$ respectively.
Expanding the determinant along the columns $[1:k-1]$ (we do not touch column
$0$) we necessarily have to pick these entries. Then the determinant equals
$\pm a_1\cdots a_{k-1}$ times the determinant of the matrix obtained by
skipping the rows and columns in which the entries $a_1,\ldots,a_{k-1}$ are
standing. These are columns $[1:k-1]$ and rows $[p+1:p+k-1]$.

From the skipping of rows $[p+1:p+k-1]$, we see that in columns $[p+2:p+k-1]$
there is only one non-zero entry left, being $a_{p+2},\ldots,a_{p+k-1}$
respectively. So again the determinant picks up a factor $\pm a_{p+2}\ldots
a_{p+k-1}$, and we can proceed with the determinant of the matrix obtained by
skipping the rows and columns in which the entries $a_{p+2},\ldots,a_{p+k-1}$
are standing. These are columns $[p+2:p+k-1]$ and rows $[2p+2:2p+k-1]$.

From the skipping of rows $[2p+2:2p+k-1]$, we now have only one non-zero entry
left in each of the columns $[2p+3:2p+k-1]$, being $a_{2p+3},\ldots,a_{2p+k-1}$
respectively. We can then make a reduction as in the previous paragraphs.
Carrying on this scheme a few more steps, we are left with the following
submatrix of $H_n-xI_n$: It is obtained by skipping the rows
\begin{equation}\label{skip:1}
[0:k-1]\cup [p+1:p+k-1]\cup [2p+2:2p+k-1]\cup\ldots\cup\{(k-1)p+k-1\}
\end{equation} and the columns
\begin{equation}\label{skip:2}
[1:k-1]\cup [p+2:p+k-1]\cup [2p+3:2p+k-1]\cup\ldots\cup\{(k-2)p+k-1\}
\end{equation}
in the starting matrix $H_n-xI_n$.

We can do similar operations with the last rows and columns of $H_n-xI_n$.
Indeed, recall that in the definition of $P_{k,n}(x)$ we are skipping columns
$[n-k:n-1]$ of $H_n-xI_n$. The determinant can then be further reduced to the
one obtained by skipping the rows
\begin{equation}\label{skip:3}
\{n-(k-1)p-k\}\cup\ldots\cup [n-2p-k:n-2p-3]\cup[n-p-k:n-p-2]\cup [n-k:n-1]
\end{equation}
and the columns
\begin{equation}\label{skip:4} \{n-kp-k\}\cup \ldots\cup [n-2p-k:n-2p-2]\cup
[n-p-k:n-p-1]\cup [n-k:n-1]
\end{equation}
in the matrix $H_n-xI_n$.


Summarizing, we see that $P_{k,n}(x)$ is, up to a constant, equal to the
determinant of the submatrix of $H_{n}-x I_{n}$ obtained by skipping the rows
\eqref{skip:1} and \eqref{skip:3} and the columns \eqref{skip:2} and
\eqref{skip:4}.

The skipping of the indicated rows and columns of $H_{n}-x I_{n}$ is again
equivalent (up to a sign) to removing certain rows and columns of the permuted
matrix $P^{-1}H_nP-xI_n$ in \eqref{EV:permuted}. This leads to the formula
\begin{equation}\label{EV:gen}
P_{k,n}(x) = c\det\begin{pmatrix}X_0' & Y_0' \\ 0 & X_1' & Y_1' \\ & & \ddots & \ddots \\
 & & & X_{p-1}' & Y_{p-1}' \\ Y_p' & & & & X_p'
 \end{pmatrix},
\end{equation}
$c\neq 0$, where $X_0'$ is obtained by skipping the first $k$ rows and last $k$
columns of $X_0$; where $X_j'$, $j\in [1:p]$, is obtained by skipping the first
$\max\{k-j,0\}$ rows and columns and the last $\max\{j-p+k,0\}$ rows and
columns of $X_j$; where $Y_j'$, $j\in [0:p-1]$, is obtained by skipping the
first $\max\{k-j,0\}$ rows and $\max\{k-j-1,0\}$ columns and the last
$\max\{j-p+k,0\}$ rows and $\max\{j-p+k+1,0\}$ columns of $Y_j$; and finally
$Y_p'$ is obtained by skipping the last $k$ rows and columns of $Y_p$.

Note that each of the diagonal blocks $X_j'$, $j\in [1:p]$ in \eqref{EV:gen} is
of the form $-xI$ and moreover
\begin{equation}\label{X0hat} X_0' = \begin{pmatrix} 0 & -x I \\
0_{k\times k} & 0 \end{pmatrix}
\end{equation}
with $k$ zero columns added at the left and $k$ zero rows at the bottom. Hence
we are in a position to apply Lemma~\ref{lemma:cyclic:product}(b): this yields
\begin{equation}\label{Pkn:cyclic} P_{k,n}(x) = c x^{k(p-k)}\det\left(Y_0'Y_1'\ldots Y_p'-x^{p+1}\begin{pmatrix} 0 & I \\
0_{k\times k} & 0 \end{pmatrix}\right),
\end{equation}
$c\neq 0$. Note that each of the matrices $Y_0',Y_1',\ldots,Y_p'$ in
\eqref{EV:gen} is bidiagonal with nonnegative entries and hence TNN. Thus also
the matrix product $Y_0'Y_1'\ldots Y_p'$ is TNN. Cor.~\ref{prop:geneigTNN} and
\eqref{Pkn:cyclic} then imply that all the zeros of $P_{k,n}$ lie on the star
$S_+$ ($S_-$) if $k$ is even (odd). Finally, if we apply the Cauchy-Binet
formula to $\det (Y_{0}'Y_{1}'\ldots Y_p')$ then we see that this determinant
is the sum of a finite number of nonnegative terms with at least one term
strictly positive (for instance, the term obtained by multiplying the
determinants of the principal leading submatrices of $Y_{i}'$, $i\in[0:p]$, is
strictly positive). Noting that $m_{k,n}=k(p-k)$
if $n$ is a multiple of $p+1$, we now obtain
Theorem~\ref{theorem:interlace}(a).

\subsubsection{Proof of Theorem~\ref{theorem:interlace}(b) ($n$ a multiple of $p+1$)}
\label{subsubsection:proof2}

Now we will prove the interlacing between the zeros of $P_{k,n}$ and
$P_{k,n+1}$ in Theorem~\ref{theorem:interlace}(b), still assuming that $n$ is a
multiple of $p+1$.

Recall that in the determinantal formula for $P_{k,n+1}(x)$ we are skipping the
rows $[0:k-1]$ of $H_{n+1}-xI_{n+1}$. In exactly the same way as in
Section~\ref{subsubsection:proof1}, this leads to an iterated skipping process,
allowing us to skip the rows \eqref{skip:1} and the columns \eqref{skip:2} of
$H_{n+1}-xI_{n+1}$.

In the definition of $P_{k,n+1}(x)$ we are skipping the columns $[n-k+1:n]$ of
$H_{n+1}-xI_{n+1}$. This leads again to an iterated skipping process, allowing
us to skip the rows
\begin{equation}\label{skip:5}
\{n-(k-2)p-k+1\}\cup \ldots\cup [n-2p-k+1:n-2p-3]\cup [n-p-k+1:n-p-2]\cup
[n-k+1:n-1]
\end{equation}
and the columns
\begin{equation}\label{skip:6}
\{n-(k-1)p-k+1\}\cup\ldots\cup [n-2p-k+1:n-2p-2]\cup [n-p-k+1:n-p-1]\cup
[n-k+1:n]
\end{equation}
of $H_{n+1}-xI_{n+1}$. Note that we are \emph{not} skipping the rows
$n,n-p-1,n-2p-2,\ldots$ and the columns $n-p,n-2p-1,n-3p-2,\ldots$, although we
are allowed to do that. The reason for not skipping these rows and columns, is
because that would complicate the comparison to the formulas
\eqref{skip:3}--\eqref{skip:4} for $P_{k,n}$.

In terms of the permuted matrix \eqref{EV:gen} we obtain
\begin{equation}\label{EV:gen:2}
P_{k,n+1}(x) = c\det\begin{pmatrix}\wtil X_0 & \wtil Y_0 \\ 0 & \wtil X_1 & \wtil Y_1 \\ & & \ddots & \ddots \\
 & & & \wtil X_{p-1} & \wtil Y_{p-1} \\ \wtil Y_p & & & & \wtil X_p
 \end{pmatrix},
\end{equation}
$c\neq 0$, where the blocks $\wtil X_j,\wtil Y_j$ are obtained from the blocks
$X_j',Y_j'$ for $P_{k,n}$ in \eqref{EV:gen} by the formulas
\begin{eqnarray*}
&\wtil X_0 = \left(\begin{array}{cc} X_0'& -x\mathbf{\til e} \\ 0 &
0\end{array}\right),\\ &\wtil X_j = X_j',&\qquad j\in [1:p-k], \\
&\wtil X_j = \begin{pmatrix} X_j' & 0
\\ 0 & -x\end{pmatrix},&\qquad j\in [p-k+1:p],
\end{eqnarray*}
with $\mathbf{\til e}$ denoting the $k$th last column of the identity matrix,
and \begin{eqnarray*} &\wtil Y_0 = \left(\begin{array}{c} Y_0' \\
a_*\mathbf{e}^T\end{array}\right),\qquad \wtil Y_j = Y_j', \qquad j\in [1:p-k-1],\\
&\wtil Y_{p-k} =\begin{pmatrix}Y_{p-k}' & \vece \end{pmatrix},\qquad \wtil Y_j
=\begin{pmatrix}Y_j' & \mathbf{e} \\ 0 & a_*\end{pmatrix}, \qquad j\in
[p-k+1:p],
\end{eqnarray*} with $\vece$ denoting the last column of the identity
matrix, and with the $a_*$ certain recurrence coefficients.
Lemma~\ref{lemma:cyclic:product}(b) yields
\begin{equation}\label{Pkn:cyclic:2} P_{k,n+1}(x) =
c x^{(k-1)(p-k)}\det\left(\wtil Y_0\wtil Y_1\ldots \wtil Y_p-x^{p+1}\begin{pmatrix} 0 & I \\
0_{k\times k} & 0 \end{pmatrix}\right),
\end{equation}
$c\neq 0$. But the cyclic product $\wtil Y_0\wtil Y_1\ldots\wtil Y_p$ has
$Y_0'Y_1'\ldots Y_p'$ as its leading principal submatrix. Hence
Cor.~\ref{prop:geneigTNN} yields the required interlacing relation in
Theorem~\ref{theorem:interlace}(b). It is also easy to see, as at the end of
Section~\ref{subsubsection:proof1}, that $\det (\wtil Y_{0} \wtil Y_{1}\ldots
\wtil Y_{p})>0$.

\subsubsection{Proof of Theorem~\ref{theorem:interlace}(c) ($n$ a multiple of $p+1$)}
\label{subsubsection:proof3}

Next we prove the interlacing between the zeros of $P_{k,n}$ and $P_{k,n+p+1}$
in Theorem~\ref{theorem:interlace}(c), still assuming that $n$ is a multiple of
$p+1$. Observe that $n+p+1$ is also a multiple of $p+1$. Applying exactly the
same approach as in Section~\ref{subsubsection:proof1}, we find that
$P_{k,n+p+1}$ can be written as in the right hand side of \eqref{EV:gen:2},
where the blocks $\wtil X_j,\wtil Y_j$ are now obtained from the blocks
$X_j',Y_j'$ for $P_{k,n}$ in \eqref{EV:gen} by the formulas
\begin{eqnarray*}
&\wtil X_0 = \left(\begin{array}{cc} X_0'& -x\mathbf{\til e} \\ 0 &
0\end{array}\right),\qquad \wtil X_j = \begin{pmatrix}X_j' & 0 \\ 0 &
-x\end{pmatrix},\qquad j\in [1:p],
\end{eqnarray*}
with again $\mathbf{\til e}$ denoting the $k$th last column of the identity
matrix, and
\begin{eqnarray*}
& \wtil Y_j =\left(\begin{array}{cc}Y_j' & 0 \\ a_*\mathbf{e}^T &
1\end{array}\right),& \qquad j\in [0:p-k-1], \\ & \wtil Y_j =
\left(\begin{array}{cc}  Y_j'& \mathbf{e}\\ 0&a_*
\end{array}\right),&\qquad j\in [p-k:p],
\end{eqnarray*}
with $\mathbf{e}$ denoting the last column of the identity matrix and with
$a_*$ certain recurrence coefficients. We can again apply
Lemma~\ref{lemma:cyclic:product}(b). But the cyclic product $\wtil Y_0\wtil
Y_1\ldots\wtil Y_p$ has $Y_0'Y_1'\ldots Y_p'$ as its leading principal
submatrix. Cor.~\ref{prop:geneigTNN} then yields the required interlacing
relation in Theorem~\ref{theorem:interlace}(c).

\subsubsection{Proof of Theorem~\ref{theorem:interlace:kl} ($n$ a multiple of $p+1$)}
\label{subsubsection:proof4}

Next we prove the interlacing relations between the zeros of $P_{k,n}$ and
$P_{k,l,n}$ in Theorem~\ref{theorem:interlace:kl}, assuming that $1\leq k<l\leq
p$ and $n$ is a multiple of $p+1$. Recall that $P_{k,l,n}$ is the determinant
of the submatrix obtained by skipping the rows $[0:k-1]$ and the columns
$\{n-l\}\cup [n-k+1:n-1]$ of $H_n-xI_n$. This leads to an iterated skipping
process, allowing us to skip the rows \eqref{skip:1} and
$$ \{n-(k-2)p-(k-1)\}\cup\ldots\cup [n-2p-k+1:n-2p-3]\cup [n-p-k+1:n-p-2]\cup [n-k+1:n-1],$$
and the columns \eqref{skip:2} and
\begin{multline*}
\{n-(k-1)p-k+1\}\cup\ldots\cup [n-2p-k+1:n-2p-2]\cup [n-p-k+1:n-p-1]\cup
\{n-l\}\cup [n-k+1:n-1]
\end{multline*}
in the matrix $H_n-xI_n$. Then $P_{k,l,n}$ can be written as in the right hand
side of \eqref{EV:gen:2}, where the blocks $\wtil X_j,\wtil Y_j$ are now
obtained from the blocks $X_j',Y_j'$ for $P_{k,n}$ in \eqref{EV:gen} by the
formulas
\begin{eqnarray*}
&\wtil X_0 = \left(\begin{array}{cc} X_0'& -x\mathbf{\til
e}\end{array}\right),\qquad \wtil X_j = X_j',\qquad j\in [1:p-k],\ j\neq p-l+1, \\
&X_{p-l+1}' = \begin{pmatrix}\wtil X_{p-l+1}& -x\mathbf{e} \end{pmatrix},\qquad
\wtil X_j = \begin{pmatrix} X_j' & 0
\\ 0 & -x\end{pmatrix},\qquad j\in [p-k+1:p],
\end{eqnarray*}
and \begin{eqnarray*} &\wtil Y_j = Y_j',&\qquad j\in [0:p-k-1],\ j\neq p-l,\\
&Y_{p-l}' = \begin{pmatrix}\wtil Y_{p-l}& \mathbf{e} \end{pmatrix},& \\
&\wtil Y_{p-k} =\begin{pmatrix}Y_{p-k}' & \vece \end{pmatrix},& \\
&\wtil Y_j =\begin{pmatrix}Y_j' & \mathbf{e} \\ 0 & a_*\end{pmatrix},& \qquad
j\in [p-k+1:p],\end{eqnarray*} with the notations $\mathbf{e},\mathbf{\til
e},a_*$ as defined before.

There is now a complication for the block $\wtil X_{p-l+1}$: this is a scalar
multiple of the identity matrix but with the last column skipped. Moreover,
$\wtil Y_{p-l}$ is a submatrix of $Y_{p-l}'$ rather than the other way around.
We will resolve both issues by appending an extra row and column at the end of
some of the blocks in the matrix in the right hand side of \eqref{EV:gen:2}.
These extra rows and columns will have a triangular structure and therefore
will not influence the determinant (except for a scalar factor), as we will see
below. Here is the definition: we put
$$ \what  X_{p-l+1} := \begin{pmatrix}\wtil X_{p-l+1}& -x\mathbf{e}\end{pmatrix},
$$
thereby making this block square again. We also put
$$ \what X_{0} :=
\begin{pmatrix}\wtil X_{0}
\\ 0 \end{pmatrix},\qquad \what X_{j} := \begin{pmatrix}\wtil X_{j} & 0 \\ 0 &
-x\end{pmatrix},\qquad j\in [1:p-l],
$$
with the zeros denoting a row or column vector,
$$ \what Y_{j} := \begin{pmatrix} \wtil Y_{j} & 0 \\ 0 &
1\end{pmatrix},\qquad j\in [0:p-l-1], \qquad \what
Y_{p-l}:=\left(\begin{array}{cc}
\wtil Y_{p-l} & \mathbf{e} \\
0& a\end{array}\right),$$ for an arbitrary constant $a\neq 0$, and
\begin{eqnarray*} & \what X_j := \wtil X_j,\qquad j\in [p-l+2:p],
\\ & \what Y_j := \wtil Y_j,\qquad j\in [p-l+1:p].
\end{eqnarray*}
As mentioned, the triangular structure of the added rows and columns implies
that
\begin{equation}\label{EV:tilde}
\det\begin{pmatrix}\what X_0 & \what Y_0 \\ 0 & \what X_1 & \what Y_1 \\ & & \ddots & \ddots \\
 & & & \what X_{p-1} & \what Y_{p-1} \\ \what Y_p & & & & \what X_p
 \end{pmatrix}
= \pm a
\det\begin{pmatrix}\wtil X_0 & \wtil Y_0 \\ 0 & \wtil X_1 & \wtil Y_1 \\ & & \ddots & \ddots \\
 & & & \wtil X_{p-1} & \wtil Y_{p-1} \\ \wtil Y_p & & & & \wtil X_p
 \end{pmatrix}.
\end{equation}
(To see this, expand the determinant on the left-hand side of \eqref{EV:tilde}
along the last row of $\what X_{0}$ and $\what Y_{0}$. This row has only one
nonzero entry. This allows to delete this row and also the last column of
$\what Y_{0}$ and $\what X_{1}$. Next we expand the new determinant along the
last row of $\what X_{1}$ and $\what Y_{1}$, and so on.) So we can work with
$\what X_j,\what Y_j$ instead of $\wtil X_j,\wtil Y_j$.

Summarizing, $P_{k,l,n}$ can be written as a constant times the left hand side
of \eqref{EV:tilde}. Combining the above descriptions, we see that the blocks
$\what X_j,\what Y_j$ are obtained from the blocks $X_j',Y_j'$ for $P_{k,n}$ in
\eqref{EV:gen} by the formulas
\begin{eqnarray*}
&\what X_0 = \left(\begin{array}{cc} X_0'& -x\mathbf{\til e}\\ 0 &
0\end{array}\right),\qquad \what X_j = \begin{pmatrix} X_j' & 0
\\ 0 & -x\end{pmatrix},\qquad j\in [1:p-l]\cup [p-k+1:p],\\
& \what X_j = X_j',\qquad j\in [p-l+1:p-k],
\end{eqnarray*}
and
\begin{eqnarray*}
&\what Y_{j}=\begin{pmatrix} Y_{j}' & 0 \\ 0 & 1\end{pmatrix},&\quad
j\in [0:p-l-1],\\ &\what Y_{p-l}=\left(\begin{array}{c} Y_{p-l}' \\
a\mathbf{e}^T\end{array}\right),& \\
&\what Y_{j}= Y_{j}',&\quad j\in [p-l+1:p-k-1], \\ &\what
Y_{p-k}=\begin{pmatrix}Y_{p-k}' & \mathbf{e}\end{pmatrix},
&\\ &\what Y_{j}=\begin{pmatrix}Y_{j}' & \mathbf{e}\\
0 & a_{*}
\end{pmatrix},&\qquad j\in [p-k+1:p].
\end{eqnarray*}
Lemma~\ref{lemma:cyclic:product}(b) can be applied and yields
\begin{equation}\label{Pkln:cyclic} P_{k,l,n}(x) =
c x^{k(p-k)+k-l}\det\left(\what Y_0\what Y_1\ldots \what Y_p-x^{p+1}\begin{pmatrix} 0 & I \\
0_{k\times k} & 0 \end{pmatrix}\right),
\end{equation}
$c\neq 0$. But the principal leading submatrix of $\what Y_0\what Y_1\ldots
\what Y_p$ is precisely the matrix $Y_0'Y_1'\ldots Y_p'$. Taking into account
that $m_{k,n}=k(p-k)$ if $n$ is a multiple of $p+1$, we then obtain
Theorem~\ref{theorem:interlace:kl}.

\subsubsection{Modifications if $n$ is not a multiple of $p+1$}
\label{subsubsection:proof5}

If $n$ is not a multiple of $p+1$, we can use the same ideas but with a few
modifications. We focus on the construction for $P_{k,n}$ in
Section~\ref{subsubsection:proof1}. We can again write $P_{k,n}$ as in
\eqref{EV:gen}. But now the description of the blocks $X_j',Y_j'$ depends on
the residue $q$ of $n$ modulo $p+1$, $q\in [1:p]$. More precisely, the number
of rows and columns to be skipped at the top and at the left of each block
$X_j',Y_j'$, is exactly the same as in Section~\ref{subsubsection:proof1}. But
for the rows and columns at the bottom and at the right of each block, the
description that was given for $X_j',Y_j'$ in
Section~\ref{subsubsection:proof1} should now be applied to $X_{j+q}',Y_{j+q}'$
(where we view the subscripts modulo $p+1$).

The above description implies in particular that $X_0' =
\begin{pmatrix} 0 & -xI
\end{pmatrix}$, with $k$ zero columns added at its left, and $X_q' = \begin{pmatrix}
-xI\\ 0
\end{pmatrix}$, with $k$ zero rows added at its bottom. All the other $X_j$ are of
the form $-xI$. Thus we are not able to apply Lemma~\ref{lemma:cyclic:product}.

To get around this issue, we append $k$ extra rows and columns in the top
and/or left part of some of the blocks. We set
\begin{eqnarray*}
\wtil X_0 = \begin{pmatrix}-x E_k  \\ X_0' \end{pmatrix},\qquad \wtil X_j =
\begin{pmatrix}-xI_k & 0 \\ 0& X_j'\end{pmatrix},\qquad
j\in [1:q-1],\\
\wtil X_q = \begin{pmatrix}0 & X_q'  \end{pmatrix},\qquad \wtil X_j =
X_j',\qquad j\in [q+1:p],
\end{eqnarray*}
where $E_k$ is the submatrix formed by the first $k$ rows of the identity
matrix, and
\begin{eqnarray*}
\wtil Y_j = \begin{pmatrix}I_k & 0 \\ 0 & Y_j'\end{pmatrix},&\qquad j\in [0:q-1]\\
\wtil Y_j = Y_j',&\qquad j\in [q:p].
\end{eqnarray*}
Due to the triangular structure of the added rows and columns, they leave the
determinant invariant up to its sign. So we can replace the $X_j',Y_j'$ by the
$\wtil X_j,\wtil Y_j$.

The blocks $\wtil X_j$ are now all of the form $-xI$, except for $\wtil X_q$
which is of the form $\wtil X_q=\begin{pmatrix} 0& -xI\\ 0_{k\times k} &
0\end{pmatrix}.$ To bring the matrix to the form required by
Lemma~\ref{lemma:cyclic:product}, it suffices to apply a cyclic block
permutation. More precisely, we move each of the blocks $\wtil X_j,\wtil Y_j$,
$q$ positions to the top and to the left (in a cyclic way, thus reappearing at
the bottom or right of the matrix when crossing the top or left matrix border
respectively). We relabel the blocks in the permuted matrix as $X_j',Y_j'$ in
the usual way. Thus $X_j':=\wtil X_{j+q}$ and $Y_j':=\wtil Y_{j+q}$ where we
view the subscripts modulo $p+1$.

Summarizing, we can write $P_{k,n}$ as in \eqref{EV:gen}, where the blocks
$X_j',Y_j'$ have the form described in Lemma~\ref{lemma:cyclic:product}. The
blocks $X_j',Y_j'$ will play exactly the same role as in
Sections~\ref{subsubsection:proof1}--\ref{subsubsection:proof4}.

\smallskip The above construction for $P_{k,n}$, can also be used for the matrices associated to
$P_{k,n+1}$, $P_{k,n+p+1}$ and $P_{k,l,n}$. The blocks in these matrices differ from those
for $P_{k,n}$ only in their bottom right matrix corner. So the above
modifications, which only affect a fixed number of \emph{top} and \emph{left}
rows and columns in each block, are exactly the same for each of these
matrices. Since the interlacing relations are obtained by comparing the
\emph{bottom} and \emph{right} rows and columns, the proofs in
Sections~\ref{subsubsection:proof2}--\ref{subsubsection:proof4} then carry on
in exactly the same way as before.

The only effect of adding rows and columns at the top or at the left, is that
it changes the sizes of some of the blocks $X_j',Y_j'$. The sizes depend
explicitly on the residue $q$ of $n$ modulo $p+1$. The same then holds for the
exponents of $x$ yielded by Lemma~\ref{lemma:cyclic:product}.  The details are
straightforward.

\begin{remark}\label{remark:mkn:combin} To obtain the values for $m_{k,n}$ in
\eqref{mkn:def} and the fact that $\wtil P_{k,n}(0)\neq 0$ in
Theorem~\ref{theorem:interlace} with the above approach, one has to do a
careful bookkeeping. Another approach for checking these statements is to use
\eqref{pattern:expansion} below. It suffices there to find the patterns $s$
that yield the lowest exponent of $x$ in \eqref{pattern:expansion}, which is a
combinatorial exercise.
\end{remark}

\subsection{Interlacing for arbitrary Riemann-Hilbert minors}

The techniques in Sections~\ref{subsubsection:proof4} and
\ref{subsubsection:proof5} can be used to prove the following result for
arbitrary Riemann-Hilbert minors. It generalizes
Theorem~\ref{theorem:interlace:kl}.

\begin{theorem}\label{theorem:interlace:gen} Consider the two-diagonal Hessenberg matrix $H_n$ in
\eqref{H:matrix:twodiag}, where $n$ is sufficiently large. Let $k\in [0:p-1]$
and $\kappa\in [0:k-1]$ and consider the polynomials $P^{\mathbf{n}_{1}}(x)$
and $P^{\mathbf{n}_{2}}(x)$ $($cf. \eqref{def:general:geneig}$)$, with
\begin{align}
\label{vecn12}& \mathbf{n}_{1}=(n_{0},\ldots,n_{\kappa},n-k+\kappa+1,\ldots,n-1,n),\\
\nonumber & \mathbf{n}_{2}=(n_{0},\ldots,n_{\kappa-1},n-k+\kappa,\ldots,n-1,n),
\end{align}
where
\begin{equation}\label{kappa:assumption}n-p\leq
n_{0}<n_{1}<\cdots<n_{\kappa}<n-k+\kappa,\end{equation}
and the last $k-\kappa$ $($$k-\kappa+1$$)$ components of $\mathbf{n}_{1}$ $($$\mathbf{n}_{2}$$)$ are
taken consecutively.
\begin{itemize}
\item[$(a)$] We have $$
P^{\mathbf{n}_{1}}(x) = x^{m}\wtil P^{\mathbf{n}_{1}}(x^{p+1}),\qquad
P^{\mathbf{n}_{2}}(x)=x^{m+n_{\kappa}-(n-k+\kappa)}\wtil
P^{\mathbf{n}_{2}}(x^{p+1}),$$ for some explicit $m$ depending only on the
residues of the indices $n_0,\ldots,n_{\kappa},n$ modulo $p+1$. The zeros of
the above polynomials $\wtil P$ lie in $\er_+$ ($\er_-$) if $k$ is even (odd).
\item[$(b)$] Denote
by $(y_i)_{i=1,2,\ldots}$ and $(x_i)_{i=1,2,\ldots}$ the roots of $\wtil
P^{\mathbf{n}_{1}}(x)$ and $\wtil
P^{\mathbf{n}_{2}}(x)$ respectively, counting
multiplicities and ordered by increasing modulus. We have the weak interlacing
relation
$$ 0\leq |y_1|\leq |x_{1}|\leq |y_{2}|\leq |x_{2}|\leq \ldots.
$$
\end{itemize}
\end{theorem}

\section{Normal family estimates}
\label{section:normal}

The goal of this section is to prove the following result.

\begin{lemma}\label{lemma:normalfamily} (Normal family:)
Let $k\in [0:p]$ be fixed, and assume there exists an absolute constant $R>0$
so that
\begin{equation}\label{normal:family:cond}
R^{-1} <a_n< R,\qquad n\geq 0.
\end{equation}
Then for any compact set $\mathcal{K}\subset\cee\setminus S_+$ (if $k$ is even)
or $\mathcal{K}\subset\cee\setminus S_-$ (if $k$ is odd) and for any fixed set
of indices $i_j\in\zet$, $j\in [0:k]$, with
\begin{equation}\label{cond:indices}
i_0<i_1<\ldots<i_{k}\leq i_0+p,
\end{equation}
there exists a constant $M>0$ so that for all $x\in \mathcal{K}$ and for all
$n$ we have
\begin{equation}\label{normal:family:2}
M^{-1}<\left|P^{(n+i_0,n+i_1,\ldots,n+i_k)}(x)/P_{k,n}(x)\right|<M.
\end{equation}
\end{lemma}

Lemma~\ref{lemma:normalfamily} is proved in
Section~\ref{subsection:proof:normal}. In the proof we will need a
combinatorial expansion of the generalized eigenvalue determinant
$P^{(n_0,n_1,\ldots,n_k)}$, to which we turn now.

\subsection{Combinatorial expansion of generalized eigenvalue determinants}
\label{subsection:comb}

We start with a definition.

\begin{definition}\label{def:pattern} Let
$p\in\enn$, $k\in [0:p]$ and $n-p\leq n_0<n_1<\ldots<n_k=n$ be fixed numbers. A
\emph{pattern} is a sequence $s=(s_j)_{j=0}^{n-1}$ such that $s_j\in\{0,1\}$
for all $j$,
with boundary conditions
\begin{equation}\label{pattern:boundary1}
s_0=\ldots = s_{k-1}=1,
\end{equation}
\begin{equation} \label{pattern:boundary2} s_{j}=\left\{\begin{array}{ll} 1,& \quad j=n_0,n_1,\ldots,n_{k-1},
\\ 0,& \quad j\in [n-p:n-1]\setminus\{n_0,n_1,\ldots,n_{k-1}\},\end{array}\right.
\end{equation}
and such that the following rule holds: \begin{itemize}\item[] \emph{Pattern
rule}: For each $j\in [0:n-p-1]$ with $s_j=1$, exactly $k$ out of the $p$
numbers $s_{j+1},\ldots,s_{j+p}$ are equal to $1$.\end{itemize} We denote with
$\mathcal S$ the set of all such patterns $s$.
\end{definition}

For example, if $p=4$, $k=2$ and $(n_0,n_1,n_2=n)=(14,15,16)$ then the
following sequence is a pattern:
$(s_j)_{j=0}^{15}=(1,1,0,1,0,1,1,0,0,1,1,1,0,0,1,1)$. For instance, note that
$s_3=1$ and exactly 2 out of the $4$ numbers $\{s_4,s_5,s_6,s_7\}$ are equal to
1, namely the numbers $s_5$ and $s_6$.

For a pattern $s\in\mathcal S$, write $|s|=\sum_{j=0}^{n-p-1} s_j$. Thus $|s|$
is the number of indices $j\in [0:n-p-1]$ for which $s_j=1$. In the example
above we have $|s|=8$.

\begin{remark}\label{remark:pattern:p}
Let $s$ be a pattern and consider a group of $p$ consecutive numbers
$(s_j)_{j=m-p}^{m-1}$, $p\leq m\leq n$. Def.~\ref{def:pattern} easily implies
that $\#\{j\in [m-p:m-1]\mid s_j=1\} \in\{k,k+1\}$.
\end{remark}

\begin{remark}
In the case $k=0$, we understand that there is no initial condition
\eqref{pattern:boundary1}, and \eqref{pattern:boundary2} reduces to ask
$s_{j}=0$ for all $j\in [n-p:n-1]$.
\end{remark}

\begin{proposition}\label{prop:comb}
The generalized eigenvalue determinant $P^{(n_0,n_1,\ldots,n_k)}$ can be
written as a sum over patterns: \begin{equation}\label{pattern:expansion}
P^{(n_0,n_1,\ldots,n_k)}(x) = \sum_{s\in\mathcal
S}(-1)^{(p-k)|s|}\left(\prod_{j=0}^{n-p-1}
a_j^{s_j}\right)(-x)^{(k+1)(n-k)-(p+1)|s|-q},
\end{equation}
where $q:=\sum_{j=0}^{k-1} (n+j-k-n_j)\geq 0$.
\end{proposition}

\begin{proof} In the proof we assume that $k\geq 1$. Simpler arguments can be
applied to prove \eqref{pattern:expansion} in the case $k=0$. By definition, $P^{(n_0,n_1,\ldots,n_k)}(x)$ is the determinant
of the matrix obtained by skipping rows $0,\ldots,k-1$ and columns
$n_0,\ldots,n_{k-1}$ of the matrix $H_n-xI$, with $n=n_k$. We write this
determinant as a signed sum over all permutations $\sigma$ of length $n-k$:
\begin{equation}\label{pattern:perm}
P^{(n_0,n_1,\ldots,n_k)}(x) = \sum_{\sigma\in
S_{n-k}}(-1)^{\mathrm{sign}(\sigma)}\left(\prod_{i=k}^{n-1}
(H_n-xI)_{i,\sigma(i)}\right)
\end{equation}
where $(H_n-xI)_{i,j}$ denotes the $(i,j)$ entry of $H_n-xI$ and we view
$\sigma$ as a map $\sigma: [k:n-1]\to
[0:n-1]\setminus\{n_0,n_1,\ldots,n_{k-1}\}$, in the natural way. We will often
find it convenient to use the inverse map $\sigma^{-1}$.

Since $H_n-xI$ has only three nonvanishing diagonals, most of the terms in the
sum \eqref{pattern:perm} will be zero. To have a nonzero term we must have
$\sigma^{-1}(i)\in\{i-1,i,i+p\}$ for all $i$. For such a $\sigma$, we define a
sequence $(s_j)_{j=0}^{n-p-1}$ by $s_j=1$ if $\sigma^{-1}(j)=j+p$ (meaning that
the permutation $\sigma$ selects the entry $(H_n-xI)_{j+p,j}=a_j$) and $s_j=0$
otherwise.
We define the boundary values $(s_j)_{j=n-p}^{n-1}$ as in
\eqref{pattern:boundary2}. We claim that this sets up a bijection between the
permutations $\sigma$ leading to a nonzero term in \eqref{pattern:perm}, and
the patterns $s\in\mathcal S$.

To prove this assertion, consider the matrix obtained by skipping rows
$0,\ldots,k-1$ of $H_n-xI$. Its leading principal submatrix of size $2p-k+1$ by
$p+1$ can be partitioned in blocks as follows:

\begin{equation}\label{pattern:partition}
\begin{array}{cc} & \begin{array}{cc} \hphantom{ddd}k& p-k+1 \end{array}\\
\begin{array}{c} p-k\\ k\\ p-k+1\end{array} & \left(\begin{array}{cc} 0 & X \\
A & \hphantom{dd}Y\hphantom{dd} \\ 0 & B\end{array}\right)\end{array},\qquad
\end{equation}
where $A=\diag(a_0,\ldots,a_{k-1})$, $B=\diag(a_k,\ldots,a_{p})$,
$X=\begin{pmatrix} -x & 1 \\ & \ddots & \ddots \\ & & -x & 1\end{pmatrix}$, and
$Y$ has its top right entry equal to $-x$ and all its other entries equal to
zero.

Observe that in each of the first $k$ columns of \eqref{pattern:partition}
there is only one nonzero entry, being of the form $a_j$, $j\in [0:k-1]$ in the
block $A$. The permutation $\sigma$ has to pick these entries. Hence
$s_0=\ldots=s_{k-1}=1$, consistent with \eqref{pattern:boundary1}. In
particular, since we have to pick $a_0$, we are not allowed to choose the entry
$-x$ in the top right corner of the block $Y$.

Now in each of the first $p-k$ rows of \eqref{pattern:partition}, $\sigma$ has
to pick either the entry $-x$ or the entry $1$ from the corresponding row of
the block $X$. Since $X$ is rectangular with one more column than row, there
will be one of the chosen entries in each of the columns of $X$ except one.
This means in turn that $\sigma$ has to pick exactly one of the entries
$a_k,\ldots,a_p$ of the block $B$ in \eqref{pattern:partition}. So exactly one
of the numbers $s_k,\ldots,s_p$ equals 1 and the others equal zero.


\smallskip Now let $0\leq j<i<n-p$ be two integers with $s_j=s_i=1$ and $s_{j+1}=\ldots=s_{i-1}=0$.
Assume (by induction) that exactly $k$ out of the $p$ numbers
$s_{j+1},\ldots,s_{j+p}$ equal $1$. By the above paragraphs, this holds if
$j=0$. We will prove that exactly $k$ out of the $p$ numbers
$s_{i+1},\ldots,s_{i+p}$ equal $1$. Applying this argument iteratively, we will
then obtain that $(s_j)_{j=0}^{n-1}$ satisfies the pattern rule
(Def.~\ref{def:pattern}).

By assumption we have that exactly $k$ of the numbers $s_{j+1},\ldots,s_{j+p}$
equal $1$. By the definition of $i$, this implies that exactly $k-1$ of the
numbers $s_{i+1},\ldots,s_{j+p}$ equal $1$. So it will be enough to show that
exactly one of the numbers $s_{j+p+1},\ldots,s_{i+p}$ equals $1$.

Consider the submatrix of $H_n-xI$ obtained by extracting rows
$j+p,\ldots,i+p$:
\begin{equation}\label{pattern:submx}\left(\begin{array}{lllllllllll} a_j & & & & & -x & 1& & &  \\
\hline & a_{j+1}& & && &-x & 1 \\ && \ddots & &  & && \ddots & \ddots \\ && & a_{i-1}&&& & & -x & 1&\\
\hline & & & & a_{i} & & & & & -x & 1\end{array}\right).\end{equation} Note
that the entry $-x$ in the topmost row lies either in the same column or in a
column to the right of $a_i$. This follows from our assumptions that exactly
$k\geq 1$ of the numbers $s_{j+1},\ldots,s_{j+p}$ equal $1$, and
$s_{j+1}=\cdots=s_{i-1}=0$. Now since $\sigma$ picks the entries $a_j$ and
$a_i$, the entries $-x$ and $1$ in the first and last row of
\eqref{pattern:submx} cannot be chosen. On the other hand, $\sigma$ has to
choose one of the entries $-x$ or $1$ from each of the rows containing
$a_{j+1},\ldots,a_{i-1}$ in \eqref{pattern:submx}. Since the block formed by
the entries $-x$ and $1$ lying between the two horizontal lines in
\eqref{pattern:submx} is rectangular with one more column than row, there will
be one of the chosen entries in each of its columns (i.e., the columns
$[j+p+1:i+p]$), except one. This implies in turn that $\sigma$ has to pick
exactly one of the entries $a_{j+p+1},\ldots,a_{i+p}$. Thus exactly one of the
numbers $s_{j+p+1},\ldots,s_{i+p}$ equals $1$, proving the claim of the above
paragraph.

In the above argument we were tacitly assuming that $[j+p+1:i+p]$ is disjoint
from the set of skipped columns $\{n_0,n_1,\ldots,n_{k-1}\}$ in the definition
of $P^{(n_0,n_1,\ldots,n_k)}$. If this disjointness fails, then similar
arguments as above show that $[j+p+1:i+p]$ must contain exactly one of the
indices $n_0,n_1,\ldots,n_{k-1}$, and again, exactly one of the numbers
$s_{j+p+1},\ldots,s_{i+p}$ equals $1$ (recall \eqref{pattern:boundary2}).

\smallskip Summarizing, we have proved that each permutation $\sigma$ leading to
a nonzero term in \eqref{pattern:perm} defines a pattern $s\in\mathcal S$ with
\begin{equation}\label{pattern:sigma}\textrm{$s_j=1$
if and only if $\sigma^{-1}(j)=j+p$,\quad $j\in [0:n-p-1]$}.
\end{equation}
Conversely, we claim that each pattern $s\in\mathcal S$ leads to a unique
permutation $\sigma$ satisfying \eqref{pattern:sigma} and associated to a
nonzero term in \eqref{pattern:perm}. We call $\sigma$ the \emph{permutation
induced by $s\in\mathcal S$}. To prove its existence and uniqueness, let again
$j<i$ be two numbers with $s_j=s_i=1$, $s_{j+1}=\ldots=s_{i-1}=0$. As observed
before, exactly one of the numbers $s_{j+p+1},\ldots,s_{i+p}$ equals $1$;
denote this number by $s_{l+p}$, for suitable $l$. Skipping the corresponding
column in \eqref{pattern:submx}, the block formed by the entries $-x$ and $1$
lying between the two horizontal lines in \eqref{pattern:submx} then takes the
form $\diag(C,D)$ with
$$
C=\begin{pmatrix}
-x & 1 \\
& \ddots & \ddots \\
& & \ddots & 1 \\
&&& -x \end{pmatrix}_{(l-j-1)\times (l-j-1)},\qquad
D=\begin{pmatrix}1 \\
-x & \ddots \\
& \ddots & \ddots \\
& & -x & 1
\end{pmatrix}_{(i-l)\times (i-l)}.
$$
Note that both matrices $C$ and $D$ are square and triangular with nonzero
diagonal entries. Hence if $\sigma$ is a permutation induced by $s$, then we
should have $\sigma^{-1}(m)=m$, for $m\in [j+p+1:l+p-1]$ (so that $\sigma$
picks the diagonal entries of $C$), and $\sigma^{-1}(m)=m-1$, for $m\in
[l+p+1:i+p]$ (so that $\sigma$ picks the diagonal entries of $D$). If we follow
this rule consequently for all $j<i$ with $s_j=s_i=1$ and
$s_{j+1}=\ldots=s_{i-1}=0$, and use a similar reasoning near the top left
matrix corner \eqref{pattern:partition}, and near the bottom right matrix
corner, then we will end up with the unique permutation $\sigma$ induced by
$s\in\mathcal S$.
This proves the existence and uniqueness of $\sigma$.


\smallskip
Now let $s\in\mathcal S$ be a pattern with induced permutation $\sigma$. We
claim that $\mathrm{sign}(\sigma)=(-1)^{(p-k)|s|}$. To see this, recall that
$\mathrm{sign}(\sigma)=(-1)^K$ where $K$ is the number of pairs of column
indices $(j,i)$ in $[0:n-1]\setminus\{n_0,\ldots,n_{k-1}\}$ with $j<i$ and
$\sigma^{-1}(j)>\sigma^{-1}(i)$. Since $\sigma^{-1}(i)\in\{i-1,i,i+p\}$ for all
$i$, in our case $K$ is the number of pairs of column indices $(j,i)$ with
$j\in [0:n-p-1]$, $i\in [j:j+p]$, $\sigma^{-1}(j)=j+p$ (i.e., $s_j=1$) and
$\sigma^{-1}(i)\in\{i-1,i\}$ (i.e., $s_i=0$). But if $j\in [0:n-p-1]$ is such
that $s_j=1$ then exactly $p-k$ of the numbers $s_{j+1},\ldots,s_{j+p}$ are
$0$, by Def.~\ref{def:pattern}. Thus $K=(p-k)|s|$ and
$\mathrm{sign}(\sigma)=(-1)^{(p-k)|s|}$, proving our claim.

\smallskip Let again $s\in\mathcal S$ be a pattern with induced permutation $\sigma$.
Let $a,b,c$ be the number of indices $i\in [k:n-1]$ with $\sigma(i)=i-p$,
$\sigma(i)=i$ and $\sigma(i)=i+1$, respectively. Thus $a,b,c$ denote the number
of entries of $H_n-xI$ of the form $a_j$, $-x$ and $1$, respectively, that are
picked by $\sigma$. We have the two relations
\begin{equation}\label{pattern:two:relations}
a+b+c=n-k,\qquad pa-c+ \sum_{i=0}^{k-1} (i-n_i) = 0.
\end{equation}
The first relation is obvious. The second one follows from $\sum_{i=k}^{n-1}
(i-\sigma(i))+\sum_{i=0}^{k-1} (i-n_i)=0$, due to the facts that $\sigma$ is a
permutation and we are skipping the rows $0,1,\ldots,k-1$ and columns
$n_0,n_1,\ldots,n_{k-1}$ of $H_n-xI$. Now by adding the two relations in
\eqref{pattern:two:relations}, we obtain $b=(k+1)(n-k)-(p+1)a-q$ with $q$ as in
the statement of the proposition. This yields the exponent of $-x$ in
\eqref{pattern:expansion}. Putting together all the above observations, we
obtain \eqref{pattern:expansion}.

\end{proof}

Next we state a technical lemma on the existence of patterns with prescribed
initial part.

\begin{lemma}\label{lemma:pattern:tech}
Let $p,k,n$ and $K\geq (p+1)(k+1)+pk$ be positive integers and let
$(s_j)_{j=0}^{n-K}$ satisfy \eqref{pattern:boundary1}, with $s_j\in\{0,1\}$ for
all $j$, and such that the pattern rule holds for all $j\in [0:n-K-p]$. Then
for any indices $(n_j)_{j=0}^{k}$ with $n-p\leq n_0<n_1<\ldots<n_k=n$ one can
assign the numbers $(s_j)_{j=n-K+1}^{n-1}$ such that $(s_j)_{j=0}^{n-1}$ is a
pattern with respect to these indices (Def.~\ref{def:pattern}).
\end{lemma}

\begin{proof}
We will assume that $K=\wtil K:=(p+1)(k+1)+pk$; the case where $K>\wtil K$ is
discussed at the end of the proof.

Consider the group of $p$ consecutive numbers $(s_j)_{m-p+1}^{m}$ with $m=n-K$.
Remark~\ref{remark:pattern:p} implies that it has exactly $k$ or $k+1$ entries
equal to $1$. Assume these entries are at the positions $m-p+i_j$, $j\in
[l:k]$, with $l\in\{0,1\}$ and $1\leq i_l<\ldots<i_{k}\leq p$. We define the
next group of $p$ consecutive numbers $(s_j)_{m+1}^{m+p}$ such that it has
precisely $k$ entries equal to 1, standing at the positions $m+i_j$, $j\in
[1:k]$. This definition is valid since it satisfies the pattern rule.

Next, we define $(s_j)_{m+p+1}^{m+2p}$ such that it has 1's precisely at the
positions $m+p+\wtil i_j$ with $\wtil i_1=1$ and $\wtil i_j=i_j$ for $j\in
[2:k]$. In the next group $(s_j)_{m+2p+1}^{m+3p}$ we put 1's at the indices
$m+2p+\what i_j$ with $\what i_1=1$, $\what i_2=2$ and $\what i_j=i_{j}$, $j\in
[3:k]$. We repeat this procedure until we arrive at $(s_j)_{m+kp+1}^{m+(k+1)p}$
with 1's at its first $k$ positions and zeros elsewhere. We also define each of
the numbers $(s_j)_{m+(k+1)p+1}^{m+(k+1)p+k}$ as 1. These definitions are
compatible with the pattern rule.

Next, we use a similar strategy to arrive at the prescribed boundary conditions
\eqref{pattern:boundary2}. Setting $\wtil m:=m+(k+1)(p+1)$, we know from the
last paragraph that the group $(s_j)_{\wtil m-p}^{\wtil m-1}$ has $1$'s at its
last $k$ positions and zeros elsewhere. Then we define $(s_j)_{\wtil m}^{\wtil
m+p-1}$ with 1's at the position $\wtil m+p+n_0-n$ and at its last $k-1$
positions. Next, we define $(s_j)_{\wtil m+p}^{\wtil m+2p-1}$ with 1's at the
positions $\wtil m+2p+n_0-n$, $\wtil m+2p+n_1-n$ and at its last $k-2$
positions. Repeating this process, we end up with $(s_j)_{\wtil
m+(k-1)p}^{\wtil m+kp-1}$ having 1's at the positions $\wtil m+kp+n_j-n$, $j\in
[0:k-1]$, and zeros elsewhere. These definitions are compatible with the
pattern rule. Moreover, one checks that
$$\wtil m+kp+n_j-n=
n_j,\qquad \wtil m+kp-1 = n-1.$$ So
we obtain the desired boundary condition \eqref{pattern:boundary2}.

Finally, if $K>\wtil K:=(p+1)(k+1)+pk$ then we arbitrarily assign the numbers
$(s_j)_{j=n-K+1}^{n-\wtil K}$ so that the pattern rule is satisfied. We then
use the extended sequence $(s_j)_{j=0}^{n-\wtil K}$ and proceed in exactly the
same way as before.
\end{proof}

\subsection{Proof of Lemma~\ref{lemma:normalfamily}}
\label{subsection:proof:normal}

We write \eqref{pattern:expansion} in the form
\begin{equation}\label{pattern:expansion:bis}
P^{(n_0,n_1,\ldots,n_k)}(x) = (-x)^{(k+1)(n-k)-q}\sum_{s\in\mathcal
S}(-1)^{(k+1)|s|}\left(\prod_{j=0}^{n-p-1} a_j^{s_j}\right)y^{-|s|},
\end{equation}
with $y:=x^{p+1}$. Suppose that $k$ is odd and $y=x^{p+1}\in\er_+$. Then in the
above sum, each term is real and positive and hence no cancelation can occur.
The same happens if $k$ is even and $y=x^{p+1}\in\er_-$.

Consider the ratio of polynomials in \eqref{normal:family:2}. Both the
numerator and denominator can be written as in \eqref{pattern:expansion:bis}.
Let $s=(s_j)_{j=0}^{n+i_k-1}$ be a pattern corresponding to the indices
$n+i_0,\ldots,n+i_{k}$. Lemma~\ref{lemma:pattern:tech} implies that there
exists a pattern $\wtil s=(\wtil s_j)_{j=0}^{n-1}$ corresponding to the indices
$[n-k+1:n]$ such that $\wtil s_j=s_j$ for all $j\in [0:n-K]$, with
$K=\max\{(k+1)(p+1)+kp,-i_k+1\}$. Clearly there are only finitely many such
patterns $\wtil s$ (or $s$) with prescribed initial part $(\wtil
s_j)_{j=0}^{n-K}$ (or $(s_j)_{j=0}^{n-K}$). Now for each fixed
$y\in\cee\setminus\{0\}$ there exists a constant $M>0$ so that
\begin{equation}\label{pattern:termwise}
M^{-1}<\left|\left(y^{-|s|}\prod_{j=0}^{n+i_k-p-1}
a_j^{s_j}\right)/\left(y^{-|\til s|}\prod_{j=0}^{n-p-1} a_j^{\til
s_j}\right)\right|<M,\end{equation} uniformly in $n$. This is because the
products in the numerator and denominator are equal except for at most a finite number
(independent of $n$) of factors $a_j$ and $y$, and in view of \eqref{normal:family:cond}. We conclude
that for each fixed $x\in S_+\setminus\{0\}$ (if $k$ is odd) or $x\in
S_-\setminus\{0\}$ (if $k$ is even) there is a (new) constant $M>0$ so that
\begin{equation}\label{normal:family:weak} M^{-1} <
\left|P^{(n+i_0,n+i_1,\ldots,n+i_{k})}(x)/P_{k,n}(x)\right| < M,
\end{equation}
uniformly in $n$. This is due to the termwise estimate \eqref{pattern:termwise}
and our earlier observation that the terms in \eqref{pattern:expansion:bis} are
all real with fixed sign. This already gives us a normal family estimate on
compact sets of $S_+\setminus\{0\}$ (if $k$ is odd) or $S_-\setminus\{0\}$ (if
$k$ is even).

To obtain the full statement of Lemma~\ref{lemma:normalfamily}, we recall the
interlacing relation for the generalized eigenvalues in
Theorem~\ref{theorem:interlace:gen}. With the notations $\mathbf{n}_1$,
$\mathbf{n}_2$ as in \eqref{vecn12}, Theorem~\ref{theorem:interlace:gen} yields
the partial fraction decomposition
\begin{equation}\label{parfrac:proof}
\wtil P^{\mathbf{n}_2}(x)/ \wtil P^{\mathbf{n}_1}(x) =
\alpha_0+\sum_{i=1,2,3,\ldots}\alpha_i/(x-y_i),
\end{equation}
where the numbers $\alpha_1,\alpha_2,\ldots$ all have the same sign, which is
also the sign of $\alpha_0$ if $k$ is odd, or minus the sign of $\alpha_0$ if
$k$ is even. For convenience we will assume that $k$ is odd. In view of
\eqref{normal:family:weak} we already know that the left hand side of
\eqref{parfrac:proof} is uniformly bounded in $n$ for each fixed point
$x\in\er_+\setminus\{0\}$. Fix such an $x$. In view of the above observations
we have
$$ |\alpha_0|<M,\qquad \sum_{i=1,2,3,\ldots}|\alpha_i|/\mathrm{dist}(x,y_i)<M,
$$
uniformly in $n$, where $\mathrm{dist}$ is the Euclidean distance and we used
that $x>0$, $y_i\leq 0$ and all terms in \eqref{parfrac:proof} have positive
sign. But now for any compact set $\mathcal{K}'\subset\cee\setminus \er_-$
there exists $R>0$ so that
$$R^{-1}<\left|\mathrm{dist}(x,y_i)/\mathrm{dist}(t,y_i)\right|<R,\qquad \textrm{for all $t\in \mathcal{K}'$ and $y_i\in \er_-$}. $$
This now easily implies that
\[
\left|\wtil P^{\mathbf{n}_2}(t)/ \wtil P^{\mathbf{n}_1}(t)\right|<M,
\]
for a new $M>0$ and for all $t\in \mathcal{K}'$, uniformly in $n$. This already
shows that the family of ratios \eqref{parfrac:proof} is normal on
$\cee\setminus\er_-$. Applying \eqref{normal:family:weak} two times for
appropriate indices, we know that for each $t>0$ there exists a constant $M'>0$
such that
\[
\left|\wtil P^{\mathbf{n}_2}(t)/ \wtil P^{\mathbf{n}_1}(t)\right|>M',
\]
for all $n$. This observation and Hurwitz' theorem imply that for any compact
set $\mathcal{K}'\subset\cee\setminus\er_-$, the functions
\eqref{parfrac:proof} are also uniformly bounded from below on $\mathcal{K}'$
by a positive constant. This proves \eqref{normal:family:2} for the ratios
$P^{\mathbf{n}_2}(t)/ P^{\mathbf{n}_1}(t)$.
Similarly, using
Theorem~\ref{theorem:interlace}(b) we obtain
\begin{equation}\label{normal:family:4}
M^{-1}<\left|P_{k,n}(t)/P_{k,n+1}(t)\right|<M,
\end{equation}
for all $t$ in a compact $\mathcal{K}\subset\cee\setminus S_{-}$, uniformly in
$n$. But now the ratio of polynomials in \eqref{normal:family:2} can be written
as a product of finitely many ratios of the form \eqref{normal:family:4} or
$P^{\mathbf{n}_2}(t)/ P^{\mathbf{n}_1}(t)$, or their inverses, recall
\eqref{vecn12}. This yields Lemma~\ref{lemma:normalfamily}. $\bol$

\begin{remark}
The proof of Lemma~\ref{lemma:normalfamily} simplifies considerably when $k=0$
or $k=p$. This is because in the former case we deal with the polynomials
$Q_{n}(x)$, whose zeros are uniformly bounded on $S_{+}$, while in the latter
case the numerator and denominator in \eqref{normal:family:2} are simply
constants.
\end{remark}

\section{Proof of the Widom-type formula}
\label{section:widom}

In this section we prove Theorem~\ref{theorem:Widomformula2}. For ease
of exposition let us assume for the moment that the period $r$ is sufficiently
large: $r\geq p$. The condition \eqref{periodic:exact} implies that $H$ is a
tridiagonal block Toeplitz operator
\begin{equation}\label{H:blockToeplitz}
H=\begin{pmatrix} B_{0} & B_{-1}\\ B_{1} & B_{0} & B_{-1} \\
  & B_{1} & B_{0} & \ddots \\
  & & \ddots & \ddots
\end{pmatrix},
\end{equation}
where the blocks $B_{k}$ are of size $r\times r$ and given by
\begin{equation}\label{blocks:B01}
B_{0}=\begin{pmatrix}
b_0^{(0)} &  1  & & &  0 \\
\vdots & \ddots & \ddots & &   \\
b_0^{(p)} &  & \ddots & \ddots &  \\
 & \ddots & & \ddots & 1 \\
0   & & b_{r-p-1}^{(p)} & \ldots & b_{r-1}^{(0)}
\end{pmatrix},\quad
B_{1}=\begin{pmatrix}
0 & \ldots &  b_{r-p}^{(p)} & \ldots & b_{r-1}^{(1)} \\
\vdots & \ddots &  & \ddots & \vdots \\
\vdots &  & \ddots &  & b_{r-1}^{(p)} \\
\vdots &  &  & \ddots & \vdots \\
0 & \ldots & \ldots & \ldots & 0
\end{pmatrix},
\end{equation}
\begin{equation}\label{blocks:Bminus1} B_{-1}=\begin{pmatrix}
\mathbf{0} & 0 \\ 1 & \mathbf{0}
\end{pmatrix}_{r\times r},
\end{equation}
where $\mathbf{0}$ denotes either a row or a column vector and the $0$ in the
top right corner is a square matrix of size $r-1$.
The \emph{symbol} $F(z,x)$ of the block Toeplitz matrix \eqref{H:blockToeplitz}
is defined as \cite{BS2,Widom1}
\begin{equation*}
F(z,x) = B_{-1}z^{-1}+B_0+B_1 z-xI_r,
\end{equation*}
where $I_r$ is the identity matrix of size $r$. One checks that this definition
coincides with \eqref{symbol:Hess}. In fact, a similar reasoning can be used
also if $r<p$ \cite[Sec.~4]{Del}.

The determinant of a banded block Toeplitz matrix is given by the next result.

\begin{lemma}\label{lemma:Widomformula} (Widom's determinant identity:)
Under the assumptions of Theorem~\ref{theorem:Widomformula2}, we have for all
$n$ sufficiently large that
\begin{equation}\label{Widom:1} Q_{rn}(x) := \det
(x I_{rn}-H_{rn}) = \sum_{k=0}^{p} C_{k}(x)(z_k(x))^{-n-1},
\end{equation}
with
\begin{equation}\label{def:CS}
C_k(x) = \det\left( \frac{1}{2\pi \ir}\int_{\sigma} F(z,x)^{-1}\frac{dz}{z}
\right),
\end{equation}
where $F(z,x)$ is the symbol \eqref{symbol:Hess} and $\sigma$ is an arbitrary,
clockwise oriented, closed Jordan curve enclosing $z=0$ and the point $z_k(x)$,
but none of the other points $z_j(x)$, $j\in [0:p]$, $j\neq k$.
\end{lemma}

Lemma.~\ref{lemma:Widomformula} follows by specializing Widom's result
\cite[Theorem 6.2]{Widom1} to the present setting. We now find a more
explicit form for the coefficients $C_k(x)$.


\begin{lemma}
\label{lemma:widom:explicit} Under the assumptions of
Lemma.~\ref{lemma:Widomformula}, we have
\begin{equation}\label{widom:explicit} C_k(x) =
\frac{(-1)^r}{\mathsf{f}_p}\frac{\det F^{r-1,0}(z_k(x),x)}{\prod_{j\neq
k}(z_k(x)-z_j(x))},\qquad k\in [0:p].
\end{equation}
\end{lemma}

\begin{proof} We start from formula \eqref{def:CS}. Note that this formula involves the
matrix $$\wtil F(z,x):= F(z,x)^{-1},$$ which can be written in entrywise form
as $\wtil F(z,x)=(\wtil F_{i,j}(z,x))_{i,j=0}^{r-1}$ with
\begin{equation}\label{proof:Widomexp1} \wtil F_{i,j}(z,x)= (-1)^{i+j}\frac{\det
F^{j,i}(z,x)}{\det  F(z,x)},\qquad i,j\in [0:r-1],
\end{equation}
thanks to the well-known cofactor formula for the inverse of a matrix. 

Now we consider in more detail the numerator and denominator of
\eqref{proof:Widomexp1}. For the denominator we have
$$ \det F(z,x)\equiv f(z,x)=(-1)^{r-1}z^{-1} + O(1),\qquad z\to 0,$$ by virtue of
\eqref{def:algcurve}--\eqref{f:series} and \eqref{symbol:Hess}. For the
numerator we have
$$ (\det  F^{i,j}(z,x))_{i,j=0}^{r-1} = \frac{(-1)^{r}}{z}\begin{pmatrix} \mathbf{0} &
R
\\
0 & \mathbf{0}
 \end{pmatrix}+ O(1),\qquad z\to 0,
$$
where each $\mathbf{0}$ is a row or column vector and where $R$ is an upper
triangular matrix with $1$'s on the diagonal. Indeed, this follows due to the
particular form of \eqref{symbol:Hess}, \eqref{def:Sshift}. Observe that in the
matrix $F(z,x)$, the entries with index $(i,j)$ with $j\geq i+2$ are of order
$O(z)$ as $z\rightarrow 0$. This explains the upper triangularity of $R$.
Secondly, the presence of $1$'s in the first super-diagonal of $F(z,x)$ and of
$1/z$ in the bottom left corner of the same matrix implies that for $j=i+1$,
$\det F^{i,j}(z,x)=(-1)^{r}/z+O(1)$ as $z\rightarrow 0$.

Inserting the above expressions in \eqref{proof:Widomexp1}, we obtain
\begin{equation}\label{proof:Widomexp2} \wtil F(z,x) = \begin{pmatrix} \mathbf{0} &
\wtil R
\\
0 & \mathbf{0}
 \end{pmatrix}^T+ O(z),\qquad z\to 0,
\end{equation}
for a new upper triangular matrix $\wtil R$ with $1$'s on the diagonal, where
${}^T$ denotes the transpose.

We also need the behavior of $\wtil F(z,x)$ for $z\to z_k(x)$. Note that
$$ f(z,x) = \frac{\mathsf{f}_p}{z}\prod_{t=0}^{p}(z-z_t(x)),
$$
see \eqref{f:series}--\eqref{fp}. Hence from \eqref{proof:Widomexp1} we have
\begin{equation}\label{proof:Widomexp3} \wtil F_{i,j}(z,x)= \frac{(-1)^{i+j}}{\mathsf{f}_p}\frac{z}{z-z_k(x)}\frac{\det
F^{j,i}(z,x)}{\prod_{t\neq k}(z-z_t(x))},
\end{equation}
for $i,j\in [0:r-1]$. The factor $z-z_k(x)$ in the denominator of
\eqref{proof:Widomexp3} shows that $\wtil F_{i,j}(z,x)$ can have a simple pole
at $z=z_k(x)$. Widom \cite[Sec.~6]{Widom1} observed that the matrix with the
residues,
\begin{equation}\label{proof:Widomexp4}
\left(\mathrm{Res}_{z=z_k(x)}\ \wtil F_{i,j}(z,x)\right)_{i,j=0}^{r-1},
\end{equation}
is a rank-one matrix.

Now we can finish the proof of the lemma. With the contour $\sigma$ as in the
statement of Lemma~\ref{lemma:Widomformula}, we find from
\eqref{proof:Widomexp2} and the residue theorem that
\begin{equation*} \frac{1}{2\pi \ir}\int_{\sigma} \wtil F(z,x)\frac{dz}{z}
= -\begin{pmatrix} \mathbf{0} & \wtil R
\\
0 & \mathbf{0}
 \end{pmatrix}^T - \frac{1}{z_k(x)}\left(\mathrm{Res}_{z=z_k(x)}\
 \wtil F_{i,j}(z,x)\right)_{i,j=0}^{r-1}.
\end{equation*}
Since \eqref{proof:Widomexp4} is a rank-one matrix, simple linear algebra then
shows that
\begin{equation*} \det\left(\frac{1}{2\pi \ir}\int_{\sigma} \wtil
F(z,x)\frac{dz}{z}\right) = -\frac{1}{z_k(x)}\mathrm{Res}_{z=z_k(x)}\
 \wtil F_{0,r-1}(z,x).
\end{equation*}
Using \eqref{proof:Widomexp3}, we conclude
that
\begin{equation*} \det\left(\frac{1}{2\pi \ir}\int_{\sigma} \wtil
F(z,x)\frac{dz}{z}\right) = \frac{(-1)^r}{\mathsf{f}_p} \frac{\det F^{r-1,0}(z_k(x),x)}{\prod_{j\neq k}(z_k(x)-z_j(x))},
\end{equation*}
Comparing this with \eqref{def:CS}, we obtain the
desired formula \eqref{widom:explicit}. 
\end{proof}

\begin{proof}[Proof of Theorem~\ref{theorem:Widomformula2}]
For ease of reference in the next section, we will give the proof for the case
of a two-diagonal Hessenberg matrix \eqref{H:entries:twodiag}. It will be clear
that the same proof also works for a general Hessenberg matrix
\eqref{H:entries}. From \eqref{recurrence:mxvector} we have
\begin{equation}\label{recurrence:block}
\left(\begin{array}{ccccccccccc}
0 & \ldots &  0 & a_{rn-1-p} & \ldots & -x & 1 & & & & \\
\vdots &  & \vdots &  & \ddots & & \ddots & \ddots & & & \\
0 &  & 0 &  &  & a_{rn-1} & \ldots & -x & 1 & & \\
\vdots &  & \vdots &  &  & & \ddots & & \ddots & \ddots & \\
0 & \ldots & 0 & & & & & a_{rn+r-2-p} & \ldots & -x & 1
\end{array}\right)
\left(\begin{array}{c}
Q_{rn-r}(x)\\
\vdots\\
Q_{rn-1}(x)\\
\vdots\\
Q_{rn+r-1}(x)
\end{array}
\right)=\mathbf{0},
\end{equation}
where the matrix multiplying the column vector, which we call $M(x)$, is of
size $r\times 2r$. Let us denote by $B_n(x)$ and $C_n(x)$ the matrices formed
by the first $r$ columns and last $r$ columns of $M(x)$, respectively, i.e.
\begin{equation}\label{recurrence:block2}
B_n(x)=\left(\begin{array}{c | ccc}
  & a_{rn-p-1} &  & -x \\
 0 &  & \ddots &  \\
 &  &  & a_{rn-1} \\ \hline
0 &  & 0 &
\end{array}\right),\quad C_n(x)=\left(\begin{array}{cccccc}
1 & & & & &\\
-x & \ddots &  & & & \\
 & \ddots & \ddots & & & \\
a_{rn} &  & \ddots & \ddots & & \\
 & \ddots &  & \ddots & \ddots & \\
 & & a_{rn+r-p-2} & & -x & 1
\end{array}\right),
\end{equation}
where $0$ denotes zero blocks of appropriate sizes, such that the last $r-p-1$
rows and the first $r-p-1$ columns of $B_n(x)$ are zero. Here we are
assuming that $r\geq p+1$; the case $r\leq
p$ will be discussed in Remark~\ref{remark:smallr}. Using the vectorial
notation
\begin{equation}\label{def:vecQ}
\vecQ_n(x) :=\left( Q_{rn}(x), \ldots, Q_{rn+r-1}(x) \right)^T,
\end{equation}
we then write the recurrence \eqref{recurrence:block} as
\begin{equation}\label{recurrence:vecQ}
\vecQ_{n}(x) = A_n(x) \vecQ_{n-1}(x),\qquad \textrm{with }A_n(x) :=
-C_n^{-1}(x)B_n(x),
\end{equation}
for $n\geq 1$. Now the periodicity assumption $a_{rn+j}\equiv b_j$ implies that
$B_n(x)=:B(x)$, $C_n(x)=:C(x)$ and $A_n(x)=:A(x)$ are all independent of $n$.
By repeatedly using \eqref{recurrence:vecQ}, this yields
\begin{equation}\label{matrixform}
\vecQ_{n}(x)=A(x)^{n}\vecQ_{0}(x).
\end{equation}

Assume that $\lambda$ is a non-zero eigenvalue of $A(x)$. Then
$\det(B(x)+\lambda C(x))=0$. But now
\[
B+\lambda C=\begin{pmatrix}
\lambda &  &  & b_{r-p-1} &  &-x \\
-\lambda x & \ddots &  &  & \ddots & \\
 & \ddots & \ddots &  &  &  b_{r-1}\\
\lambda b_{0} &  & \ddots & \ddots & & \\
 & \ddots & & \ddots & \ddots & \\
 & & \lambda b_{r-p-2} & & -\lam x & \lambda
\end{pmatrix}.
\]
If we perform the following operations to $B(x)+\lambda C(x)$: divide rows $2$
to $r$ by $\lambda$, move row $1$ to the bottom and move rows $2$ to $r$ one
level up, the resulting matrix is exactly $F(1/\lambda,x)$. Therefore $\det
F(1/\lambda,x)=0$ and $\lambda=1/z_{k}(x)$ for some $k\in [0:p]$. In
conclusion, the non-zero eigenvalues of $A(x)$ are given by $1/z_{k}(x)$, $k\in
[0:p]$. Since the first $r-p-1$ columns of $A(x)$ are zero, $0$ is also an
eigenvalue of $A(x)$, with multiplicity $r-p-1$.


The eigenspace of $A(x)$ associated with the eigenvalue $1/z_{k}(x)$ is
one-dimensional and coincides with the nullspace of $F(z_{k}(x),x)$. By
Cramer's rule it is easy to see that this subspace is spanned by the vector
\begin{equation}\label{eigenvector}
\vecv_{k}(x)=\left(\det F^{r-1,0}(z_{k}(x),x), -\det F^{r-1,1}(z_{k}(x),x),
\ldots, (-1)^{r-1}\det F^{r-1,r-1}(z_{k}(x),x)\right)^T,
\end{equation}
for any $k\in [0:p]$, whenever the vector \eqref{eigenvector} is nonzero.

Let us show that the first component of
$\vecv_{k}(x)$ is zero for only finitely many $x$. Let $\mathcal R$ be the
compact Riemann surface associated to the algebraic equation $f(z,x)=0$ whose
roots are the functions $z_{k}(x)$. The collection of functions $\det
F^{r-1,0}(z_{k}(x),x)$, $k\in[0:p]$, can be seen as a single meromorphic
function defined on $\mathcal R$. (It has poles at infinity, see also
\cite{vMM}). Now \cite[Lemma 5.5]{Del} (see also \cite[Lemma 5.6]{Del}) shows
that $\mathcal R$ is connected. Hence the above meromorphic function cannot be
identically zero since this would imply by \eqref{Widom:1} and
\eqref{widom:explicit} that $Q_{rn}\equiv 0$, clearly contradictory. Hence,
each function $\det F^{r-1,0}(z_{k}(x),x)$ has only finitely many zeros in
$\mathbb{C}$.



Now define the matrices
\begin{align*}
D(x):=&\diag(z_{0}(x)^{-1},z_{1}(x)^{-1},\ldots,
z_{p}(x)^{-1},0,\ldots,0)_{r\times r},\\
V(x):=&\left(\vecv_{0}(x),\vecv_{1}(x),\ldots,\vecv_{p}(x),\mathbf{e}_{1},\ldots,\mathbf{e}_{r-p-1}\right)_{r\times
r},
\end{align*}
where $\mathbf{e}_{i}$ denotes the standard column unit vector of index $i$.
Then $A(x)=V(x)D(x)V^{-1}(x)$, and so \eqref{matrixform} gives
\begin{equation}\label{matrixform2}
\vecQ_{n}(x) = V(x)D(x)^{n}V^{-1}(x)\vecQ_{0}(x).
\end{equation}

We already know the expression of $Q_{rn}$, see \eqref{Widom:1} and
\eqref{widom:explicit}. This allows us to find that the first $p+1$ components
of the vector $V^{-1}(x)\vecQ_{0}(x)$ are
\[
\frac{(-1)^{r}}{\mathsf{f}_{p}}\frac{1}{\prod_{i\neq
0}(z_{0}(x)-z_{i}(x))}\,z_{0}(x)^{-1},\ldots,\frac{(-1)^{r}}{\mathsf{f}_{p}}\frac{1}{\prod_{i\neq
p}(z_{p}(x)-z_{i}(x))}\,z_{p}(x)^{-1}.
\]
From this observation and \eqref{matrixform2}, the desired formula
\eqref{Widom:2} follows immediately.

We have actually shown that \eqref{Widom:2} is valid for all points $x\in\cee$
satisfying two conditions, namely that the roots $z_{k}(x)$, $k\in[0:p]$ are
pairwise distinct, and the vectors $\vecv_{k}(x)$, $k\in[0:p]$, are all
nonzero. The collection of points in $\cee$ for which the first condition holds
but the second fails is finite, as we have already seen. By continuity it is
clear that formula \eqref{Widom:2} is also valid for the exceptional points in
this finite set.

With \eqref{Widom:2} at our disposal, we can prove as before that the functions
$\det F^{r-1,j}(z_{k}(x),x)$ are zero for only a finite set of $x\in\cee$. Finally, to see
that the same holds for each function $\det F^{i,j}(z_{k}(x),x)$, apply
formula \eqref{Widom:2} for the monic polynomials associated to the cyclically
permuted symbol $Z^{-i-1} F(z,x)Z^{i+1}$.
\end{proof}


\begin{remark}\label{redef:vk}
Let $x\in\cee$ be such that the values $z_{k}(x)$, $k\in [0:p],$ are pairwise
distinct. We already observed that there are at most finitely many such $x$
with the property that the vector \eqref{eigenvector} is zero. For such $x$,
$\mathbf{v}_{k}(x)$ will always denote in the next section an eigenvector of
$A(x)$ associated with the eigenvalue $1/z_{k}(x)$.
\end{remark}

\begin{remark}\label{remark:smallr} To obtain \eqref{recurrence:vecQ} we assumed that $r\geq p+1$.
If $r\leq p$ we proceed as follows. Let $m\in\enn$ be large enough so that
$\til r:=mr\geq p+1$. The matrix $H$, which is periodic of period $r$, can also
be viewed as a periodic matrix of period $\til r$. Let $\wtil F(z,x)$ be the
associated symbol. Linear algebra shows that the roots $\wtil z_k(x)$ of $\det
\wtil F(z,x)=0$ are given by $\wtil z_k(x):=z_k(x)^m$, $k\in [0:p]$. Moreover,
the null space vector $\wtil\vecv_k(x)$ such that $\wtil F(\wtil
z_k(x),x)\wtil\vecv_k(x)=\mathbf{0}$ can be constructed as follows. With
$\vecv_k$ denoting the vector of length $r$ in \eqref{eigenvector}, we define
the vector $\wtil\vecv_k$ of length $\til r$ by
\begin{equation}\label{eigenvector:comp}
\wtil\vecv_{k}(x)=\left(\vecv_k(x)^T,z_k(x)^{-1}\vecv_k(x)^T,\ldots,z_k(x)^{-m+1}\vecv_k(x)^T\right)^T.
\end{equation}
With this vector \eqref{eigenvector:comp} playing the role that was played
before by $\vecv_k(x)$, the above proof goes through in exactly the same way as
before. This leads again to the same formula \eqref{Widom:2}.
\end{remark}

\section{Ratio and weak asymptotics of Riemann-Hilbert minors}
\label{section:Poincare}

\subsection{Generalized Poincar\'e theorem}
\label{subsection:Poincare}

The following result is contained in \cite{MN}, see also \cite{Simon2}. It is
closely related to the theory of Krylov subspaces and subspace iteration in
numerical linear algebra.

\begin{lemma}\label{lemma:poincaretheorem}(Generalized Poincar\'e theorem:)
Assume that $(A_{n})_{n=1}^{\infty}, A$ are nonsingular matrices of size
$m\times m$, $m\in\enn$, and $A=\lim_{n\rightarrow\infty} A_{n}$. Suppose that
$A$ is diagonalizable with eigenvalues $\{\lambda_{i}\}_{i=1}^{m}$ satisfying
\[
|\lambda_{1}|>|\lambda_{2}|>\cdots>|\lambda_{m}|>0.
\]
Let $\vecv_{1},\ldots,\vecv_{m}$ be eigenvectors associated to the eigenvalues
$\lambda_{1},\ldots,\lambda_{m}$, respectively. Let
$(\mathbf{u}_{n})_{n=0}^{\infty}$ be a sequence of column vectors with
$\mathbf{u}_0\neq 0$, generated by the recurrence
\[
\mathbf{u}_{n}=A_{n} \mathbf{u}_{n-1},\qquad n\geq 1.
\]
Then there exists a sequence of complex numbers $(c_{n})_n$ such that $c_{n}
\mathbf{u}_{n}\rightarrow \vecv_{j}$, for some $j\in [1:m]$.
\end{lemma}

We need a multi-column version of Lemma \ref{lemma:poincaretheorem}.

\begin{lemma}\label{lemma:poincaretheorem2}
Under the same assumptions of Lemma \ref{lemma:poincaretheorem}, let
$(U_{n})_{n=0}^{\infty}$ be a sequence of matrices of size $m\times l$, $l\in
[1:m]$, with $U_{0}$ having linearly independent columns, such that
\[
U_{n}=A_{n} U_{n-1},\qquad n\geq 1.
\]
Then there exists a sequence $(C_{n})_{n=0}^{\infty}$ of invertible, upper
triangular matrices of size $l\times l$ such that
\[
\lim_{n\rightarrow\infty} U_n
C_{n}=(\vecv_{j_1},\vecv_{j_2},\ldots,\vecv_{j_l}),
\]
the matrix with columns $\vecv_{j_1},\ldots,\vecv_{j_l}$, where
$j_1,\ldots,j_l$ are $l$ distinct indices in $[1:m]$. Here the limit is defined
entrywise.
\end{lemma}

\begin{proof}
We prove this lemma by induction on $l$. For $l=1$ it reduces to Lemma
\ref{lemma:poincaretheorem}. Let us assume as induction hypothesis that the
result holds for the index $l-1$. Thus there exists a sequence of upper
triangular, invertible matrices $C_{n}$ of size $l-1$ such that, if we write
\begin{equation}\label{Poincare:Mn} M_n :=
(\mathbf{u}_{n}^{(1)},\ldots,\mathbf{u}_{n}^{(l-1)})\,C_{n},
\end{equation}
with $\mathbf{u}_n^{(i)}$ denoting the $i$th column of $U_n$, then
\begin{equation}\label{Poincare:induction}
\lim_{n\rightarrow\infty} M_n
=(\vecv_{j_{1}},\vecv_{j_{2}},\ldots,\vecv_{j_{l-1}}),
\end{equation}
where $j_1,\ldots,j_{l-1}$ are distinct indices in $[1:m]$.

For $n$ large enough, there exists a unique column vector $\mathbf{d}_n$ such
that the vector
\begin{equation}\label{Poincare:wn}
\mathbf{w}_{n}:=\mathbf{u}_{n}^{(l)}+M_n \mathbf{d}_{n}\in\cee^m
\end{equation}
does not have a contribution from the vectors $\vecv_{j_{1}},
\vecv_{j_{2}},\ldots,\vecv_{j_{l-1}}$, i.e., if we write $\mathbf{w}_{n}$ in
terms of the basis $\{\vecv_{i}\}_{i=1}^{m}$ of $\cee^m$, then the coefficients
multiplying $\vecv_{j_{1}},\ldots,\vecv_{j_{l-1}}$ are zero. To see this,
observe that finding the coefficients of $\mathbf{d}_{n}$ amounts to solve a
non-homogeneous linear system whose coefficient matrix tends to the identity
matrix, thanks to \eqref{Poincare:induction}. Observe that $\mathbf{w}_{n}\neq
\mathbf{0}$ for all $n$, because otherwise we would have a linear dependency
between the columns of $U_n$ and therefore (by the recursion $U_{n}=A_n
U_{n-1}$ with $A_n$ nonsingular) between the columns of $U_0$, contrary to the
assumptions of the lemma.

From \eqref{Poincare:Mn}, \eqref{Poincare:wn} and the recursion $U_{n}=A_n
U_{n-1}$ we have
\begin{eqnarray}
\nonumber A_n\mathbf{w}_{n-1} &=& \mathbf{u}_{n}^{(l)}+ (\mathbf{u}_{n}^{(1)},\ldots,\mathbf{u}_{n}^{(l-1)})\,C_{n-1}\mathbf{d}_{n-1} \\
\nonumber  &=& \mathbf{u}_{n}^{(l)}+ M_{n}C_{n}^{-1}C_{n-1}\mathbf{d}_{n-1} \\
\nonumber &=& \mathbf{w}_{n}+ M_{n}\left(C_{n}^{-1}C_{n-1}\mathbf{d}_{n-1}-\mathbf{d}_{n}\right) \\
\label{Poincare:fn}&=:& \mathbf{w}_{n}+ M_{n}\mathbf{f}_{n}.
\end{eqnarray}
Note that for $n$ sufficiently large, $\mathbf{f}_{n}$ is the unique column
vector for which $A_n\mathbf{w}_{n-1}-M_{n}\mathbf{f}_{n}$  has no contribution
from the vectors $\vecv_{j_{1}}, \vecv_{j_{2}},\ldots,\vecv_{j_{l-1}}$.
Equivalently, since the $\{\vecv_{i}\}_{i=1}^{m}$ are eigenvectors for $A$,
$\mathbf{f}_{n}$ is the unique column vector for which
$(A_n-A)\mathbf{w}_{n-1}-M_{n}\mathbf{f}_{n}$  has no contribution from the
vectors $\vecv_{j_{1}}, \vecv_{j_{2}},\ldots,\vecv_{j_{l-1}}$. This yields the
estimate
\begin{equation}\label{Poincare:fn:bound} |\mathbf{f}_{n}|
\leq c||A_n-A||\ |\mathbf{w}_{n-1}|
\end{equation}
for a suitable constant $c$ and for all $n$ sufficiently large, on account of
\eqref{Poincare:induction}. Here we write $|\cdot|$ for the Euclidean norm of a
vector and $||\cdot||$ for the induced matrix norm.

Define
$$ B_n :=
A_n-M_{n}\frac{\mathbf{f}_{n}\mathbf{w}_{n-1}^{H}}{\mathbf{w}_{n-1}^{H}\mathbf{w}_{n-1}},
$$
with $\mathbf{w}_{n-1}^{H}$ denoting the conjugate transpose of
$\mathbf{w}_{n-1}$. From \eqref{Poincare:fn} we get
$$ B_n \mathbf{w}_{n-1} = \mathbf{w}_{n}
$$
while \eqref{Poincare:induction}, \eqref{Poincare:fn:bound} and the fact that
$A_n\to A$ imply that
$$ ||B_n-A_n||\to 0,\qquad n\to\infty.
$$
We can now apply Lemma~\ref{lemma:poincaretheorem} to the matrices $(B_n)_n$
and the vectors $(\mathbf{w}_n)_n$. This yields a sequence of nonzero constants
$(c_n)_n$ such that $c_n\mathbf{w}_n\to\vecv_{j_l}$ for a certain index $j_l\in
[1:m]$. By the very construction of $\mathbf{w}_n$ we have that
$j_l\not\in\{j_1,\ldots,j_{l-1}\}$. The definition of the new sequence of upper
triangular matrices of size $l\times l$ is obvious.
\end{proof}

\subsection{Ratio and weak asymptotics of Riemann-Hilbert minors}

We will apply Lemma~\ref{lemma:poincaretheorem2} to the polynomials $Q_n$
generated by the three-term recurrence \eqref{recurrencerel},
\begin{equation}\label{recurrencerel:copy} xQ_n(x)=Q_{n+1}(x)+a_{n-p}
Q_{n-p}(x),\qquad n\geq p.\end{equation} We will assume that the recurrence coefficients $a_n$ are asymptotically periodic with
period $r\in\enn$, where we may assume without loss of generality that $r\geq
p+1$. The case $r\leq p$ can be handled by enlarging the period and/or by using
Remark~\ref{remark:smallr}.

We write the recurrence relation in matrix-vector form as
\eqref{recurrence:vecQ}, recalling \eqref{recurrence:block2}--\eqref{def:vecQ}.
Since the recurrence coefficients $a_n$ are asymptotically periodic with period
$r$, we have
\begin{equation}\label{poincare:limits} \lim_{n\to\infty} (A_n(x),B_n(x),C_n(x)) =
(A(x),B(x),C(x)),\quad
\end{equation}
with $A(x):=-C(x)^{-1}B(x)$ and $B(x),C(x)$ as in \eqref{recurrence:block2} but
with $a_{rn+j}$ replaced by $b_j$, $j\in [0:r-1]$, and where the limits of the
matrices are taken entrywise.

In the proof of Theorem~\ref{theorem:Widomformula2} we observed that the matrix
$A(x)$ is closely related to the block Toeplitz symbol $F(z,x)$. More
precisely, we showed that the non-zero eigenvalues of $A(x)$ are the inverted
roots $1/z_k(x)$, $k\in [0:p]$ and the corresponding eigenvectors $\vecv_k(x)$
are given by \eqref{eigenvector} (see also Remark~\ref{redef:vk}). In addition
the matrix $A(x)$ has zero as an eigenvalue of multiplicity $r-p-1$.

The second kind functions $\Psi_n^{(k)}$ are defined in \eqref{secondkind:def}.
They satisfy the same recurrence relation \eqref{recurrencerel:copy} for all
$n\geq p$. In analogy with \eqref{def:vecQ} we set
\begin{equation}\label{vecPsi}
\vecPsi_n^{(k)}(x) :=\left(\Psi_{rn}^{(k)}(x),\ldots, \Psi_{rn+r-1}^{(k)}(x)
\right)^T
\end{equation}
and we define the RH-type matrix
\begin{equation}\label{U:RHtype}
U_n(x):=\begin{pmatrix} \vecQ_{n}(x) & \vecPsi_{n}^{(1)}(x) & \ldots &
\vecPsi_{n}^{(p)}(x)
\end{pmatrix}.
\end{equation}
Then the recurrence \eqref{recurrence:vecQ} extends to
\begin{equation}\label{recurrence:RH}
U_n(x) = A_n(x) U_{n-1}(x),\qquad n\geq 1.\end{equation} Note the similarity
with Lemma~\ref{lemma:poincaretheorem2}. Hence the next result should not come
as a surprise.

\begin{proposition}\label{prop:ratio}
Let $U_n(x)$ be the RH type matrix in \eqref{U:RHtype}. For any fixed
$x\in\cee\setminus \bigcup_{j}\Gamma_{j}$, there exists a sequence of
invertible upper triangular matrices $C_n(x)$ of size $p+1$ such that
\begin{equation}\label{Poincare:RH} \lim_{n\to\infty}
\begin{pmatrix}U_{n-1}(x)\\ U_n(x)\end{pmatrix} C_n(x)
= \begin{pmatrix} \vecv_{j_{0}}(x) & \ldots & \vecv_{j_{p}}(x)\\
z_{j_0}^{-1}(x)\vecv_{j_{0}}(x) & \ldots & z_{j_{p}}^{-1}(x)\vecv_{j_{p}}(x)
\end{pmatrix}
\end{equation}
where $(j_{0},\ldots,j_{p})$ is a permutation of $[0:p]$ (depending on $x$),
and with $\vecv_k(x)$ defined in \eqref{eigenvector}, see also
Remark~\ref{redef:vk}.
\end{proposition}

\begin{proof}
Throughout the proof we will drop the $x$-dependence for convenience. We want
to apply the generalized Poincar\'e theorem
(Lemma~\ref{lemma:poincaretheorem2}) to the recurrence \eqref{recurrence:RH}.
Recall that only the last $p+1$ columns of $A_n$ are nonzero. So the matrix
$A_n$ could have zero as an eigenvalue (of multiplicity $r-p-1$), contrary to
the assumptions of Lemma~\ref{lemma:poincaretheorem2}. To resolve this issue,
we partition
$$
U_n =: \begin{pmatrix} \wtil U_{n}\\ \what U_{n} \end{pmatrix},\qquad
\vecv_j =: \begin{pmatrix}\wtil\vecv_{j}\\
\what\vecv_{j}\end{pmatrix},
$$
with $\wtil U_{n}$ and $\what U_{n}$ having $r-p-1$ and $p+1$ rows
respectively, and similarly for $\wtil\vecv_{j}$ and $\what\vecv_{j}$. We also
partition
$$ A_n=:\begin{pmatrix}
0 & \wtil A_n \\ 0 & \what A_n
\end{pmatrix},\qquad A=:\begin{pmatrix}
0 & \wtil A \\ 0 & \what A
\end{pmatrix},
$$
with $\what A_n$ and $\what A$ square matrices of size $p+1$.

Recall that the nonzero eigenvalues of $A$ are $z_j^{-1}$, $j\in [0:p]$, and
the corresponding eigenvectors are $\vecv_j$. With the above partitions, this
yields
\begin{equation}\label{partition:eig} \mathbf{0}=
(A-z_j^{-1}I)\vecv_j =\begin{pmatrix} -z_j^{-1}I & \wtil A\\ 0 & \what
A-z_j^{-1}I
\end{pmatrix}\begin{pmatrix}\wtil\vecv_{j}\\ \what\vecv_{j}
\end{pmatrix},
\end{equation}
for all $j\in [0:p]$. In particular, the matrix $\what A$ is diagonalizable
with $z_j^{-1}$ as eigenvalues, $j\in [0:p]$, and the corresponding
eigenvectors are $\what\vecv_{j}$ (note also that \eqref{partition:eig} implies
$\what\vecv_{j}\neq \mathbf{0}$).

The recursion \eqref{recurrence:RH} becomes
\begin{equation}\label{partition:rec}
\begin{pmatrix} \wtil U_{n}\\
\what U_{n} \end{pmatrix} = \begin{pmatrix} 0 & \wtil A_n \\ 0 & \what A_n
\end{pmatrix}\begin{pmatrix} \wtil U_{n-1}\\ \what U_{n-1} \end{pmatrix}.
\end{equation}
In particular, \begin{equation}\label{partition:rec2} \what U_{n} = \what
A_n\what U_{n-1},\qquad n\geq 1.
\end{equation}
We can now apply Lemma~\ref{lemma:poincaretheorem2} to the matrices $(\what
A_n)_n$ and $(\what U_{n})_n$. Observe that the matrices $\what U_{n}$ are all
nonsingular (and so the matrices $\what A_{n}$); in fact, $\det \what U_{n}(x)$
is a nonzero constant (independent of $x$), as it follows from
Prop.~\ref{prop:geneigRH}. Lemma~\ref{lemma:poincaretheorem2} yields a sequence
of invertible upper triangular matrices $C_n$ such that
\begin{equation}\label{asymp1}
\what U_{n-1}C_n\to
\begin{pmatrix} \what \vecv_{j_{0}} & \ldots & \what \vecv_{j_{p}}
\end{pmatrix}
\end{equation}
as $n\to\infty$, for a certain permutation $(j_{0},\ldots,j_{p})$ of $[0:p]$.
(Note that we write $C_n$ instead of $C_{n-1}$.) Applying
\eqref{partition:rec2} and \eqref{asymp1} we then get
$$
\what U_{n}C_{n} = \what A_n \what U_{n-1} C_{n}\to \begin{pmatrix}
z_{j_{0}}^{-1}\what\vecv_{j_{0}} & \ldots & z_{j_{p}}^{-1}\what\vecv_{j_{p}}
\end{pmatrix},
$$
as $n\to\infty$, where we used that $\what A_n\to \what A$
and $\what\vecv_{j}$ is an eigenvector of $\what A$ for the eigenvalue
$z_j^{-1}$. On the other hand, from the first block row of
\eqref{partition:rec} we have
\begin{equation*}
\wtil U_{n}C_{n} = \wtil A_n\what U_{n-1}C_{n} \to \wtil A\begin{pmatrix}
\what\vecv_{j_{0}} & \ldots & \what\vecv_{j_{p}}
\end{pmatrix}=\begin{pmatrix} z_{j_0}^{-1}\wtil\vecv_{j_{0}} & \ldots &
z_{j_{p}}^{-1}\wtil\vecv_{j_{p}}
\end{pmatrix},
\end{equation*}
as $n\to\infty$, where the equality follows from the first block row of
\eqref{partition:eig}. Finally,
$$
\wtil U_{n-1}C_{n}=\wtil A_{n-1}\what U_{n-2} C_{n}=\wtil A_{n-1} (\what{A}_{n-1})^{-1}\what U_{n-1}C_{n}\to
\begin{pmatrix} \wtil\vecv_{j_{0}} & \ldots &
\wtil\vecv_{j_{p}}
\end{pmatrix}.
$$
Combining the above limits, the proposition is proved.
\end{proof}

In principle, the indices $j_{0},\ldots,j_{p}$ in Prop.~\ref{prop:ratio} could
depend on $x$. We will see further that this is not the case; in fact we have
$j_0=0$, $j_1=1$, and so on.

Fix $k\in [0:p]$. Taking determinants of suitable $(k+1)\times (k+1)$ minors of
\eqref{Poincare:RH} and using the fact that $C_n$ is upper triangular, we find
that
\begin{align}
\lim_{n\to\infty} B_{k,rn-1}(x)\,c_n(x) &
=(-1)^{k(k+1)/2}\det\left(\vecv_{j_{0}}'(x),\ldots,
\vecv_{j_{k}}'(x)\right),\label{ratio:minor1}\\
\lim_{n\to\infty} B_{k,rn+r-1}(x)\,c_n(x) & =
(-1)^{k(k+1)/2}\left(z_{j_0}^{-1}\ldots z_{j_{k}}^{-1}\right)\det\left(
\vecv_{j_{0}}'(x),\ldots,\vecv_{j_{k}}'(x)\right),\label{ratio:minor2}
\end{align}
for any fixed $x\in\cee\setminus \bigcup_{j}\Gamma_{j}$, where $c_n$ denotes
the determinant of the principal $(k+1)\times (k+1)$ submatrix of $C_n$, and
where the vector $\vecv_j'$ consists of the last $k+1$ entries of $\vecv_j$.
For convenience we introduce the following notation:
\begin{equation}\label{def:Sk}
S_{k}:=\left\{\begin{array}{ll}
S_{+}, & \mbox{for}\,\,k\,\, \mbox{even},\,\,k\in [0:p],\\[0.3em]
S_{-}, & \mbox{for}\,\,k\,\, \mbox{odd},\,\,k\in [0:p].
\end{array}
\right.
\end{equation}

\begin{lemma}\label{lemma:nonzero:det}
Let $k\in [0:p]$, and let $x$ be a fixed point in
$\cee\setminus(\bigcup_{j}\Gamma_{j}\cup S_{k})$. Then in
\eqref{ratio:minor1}--\eqref{ratio:minor2} we have
\begin{equation} \label{determinant:nonzero:vj}\det
\left(\vecv_{j_{0}}'(x),\ldots,\vecv_{j_{k}}'(x)\right)\neq 0.\end{equation}
\end{lemma}

\begin{proof} Let $\what \vecv_j$
consist of the last $p+1$ rows of $\vecv_j$, as in the proof of
Prop.~\ref{prop:ratio}. Recall that the columns of the matrix
\begin{equation}\label{determinant:nonzero:proof}
\left(\what \vecv_{j_0}(x),\ldots,\what \vecv_{j_{k}}(x)\right)
\end{equation}
are linearly independent, since they are eigenvectors corresponding to distinct
eigenvalues of the matrix $\what A$ (see the proof of Prop.~\ref{prop:ratio}).
In particular, there exist $k+1$ row indices such that the minor obtained by
selecting these rows in \eqref{determinant:nonzero:proof} is nonzero. Denoting
the value of this minor with $\kappa(x)\neq 0$ and taking the determinant of
the corresponding $(k+1)\times (k+1)$ minor in \eqref{Poincare:RH}, we get
\begin{equation*} \lim_{n\to\infty}
B^{(n_0,n_1,\ldots,n_k)}(x)\,c_n(x) = \pm \kappa(x) \neq 0,
\end{equation*}
for suitable indices $n_i$ with $rn-p-1\leq n_0<n_1<\ldots<n_k\leq rn-1$, with
again $c_n$ the determinant of the principal $(k+1)\times (k+1)$ submatrix of
$C_n$. Comparing this to \eqref{ratio:minor1}, we get
$$ \lim_{n\to\infty}
B_{k,rn-1}(x)/B^{(n_0,n_1,\ldots,n_k)}(x)=\pm\det
\left(\vecv_{j_{0}}'(x),\ldots, \vecv_{j_{k}}'(x)\right)/\kappa(x).
$$
Now if \eqref{determinant:nonzero:vj} fails, then this limit would be zero,
thereby contradicting Lemma \ref{lemma:normalfamily} (see also \eqref{geneig:RH}).
\end{proof}

\begin{remark}\label{remark:detnonzero:gen}
The above proof shows that \emph{any} $(k+1)\times
(k+1)$ minor of
$$\begin{pmatrix} \vecv_{j_{0}}(x) & \ldots & \vecv_{j_{k}}(x)\\
z_{j_0}^{-1}(x)\vecv_{j_{0}}(x) & \ldots & z_{j_{k}}^{-1}(x)\vecv_{j_{k}}(x)
\end{pmatrix},$$ 
obtained by selecting $k+1$ rows with the difference between the smallest and
largest row index not exceeding $p$, is nonzero if
$x\in\cee\setminus(\bigcup_{j}\Gamma_{j}\cup S_{k})$. Also
recall Remark~\ref{redef:vk}.
\end{remark}

By Lemma~\ref{lemma:nonzero:det}, we can take the ratio of \eqref{ratio:minor1}
and \eqref{ratio:minor2} and get the pointwise limit
\begin{equation*}
\lim_{n\to\infty} B_{k,rn-1}(x)/B_{k,rn+r-1}(x) = z_{j_0}(x)\ldots
z_{j_{k}}(x),\qquad x\in\cee\setminus \big(\bigcup_j \Gamma_j\cup S_{k}\big).
\end{equation*}
Similar arguments can be applied for the other residue classes modulo $r$,
showing that
\begin{equation}\label{ratioasy:prodSk:bis}
\lim_{n\to\infty} B_{k,n}(x)/B_{k,n+r}(x) = z_{j_0}(x)\ldots
z_{j_{k}}(x),\qquad x\in\cee\setminus\big(\bigcup_j \Gamma_j \cup S_{k}\big).
\end{equation}

\begin{proposition}\label{prop:jk}
In Prop.~\ref{prop:ratio} we have for any fixed $x\in\cee\setminus(S_{+}\cup S_{-})$,
\begin{equation}\label{Sk:principle} (j_{0},\ldots,j_{p})=(0,\ldots,p).
\end{equation}
\end{proposition}

Prop.~\ref{prop:jk} will be proved in Section~\ref{subsection:Gammak:star}. In
the latter section we also prove Theorem~\ref{theorem:Gammak:star}, in
particular we show that $\Gamma_{k}\subset S_{k}$ for all $k\in[0:p]$. Note
that we did not use Theorem~\ref{theorem:Gammak:star} so far.

Lemma~\ref{lemma:normalfamily} and
\eqref{ratioasy:prodSk:bis}--\eqref{Sk:principle} imply that, uniformly on
compact subsets of $\cee\setminus S_{k}$,
\begin{equation}\label{ratioasy:Bkn}
\lim_{n\to\infty} B_{k,n}(x)/B_{k,n+r}(x) = z_{0}(x)\ldots z_{k}(x).
\end{equation}

We are now ready for the \begin{proof}[Proof of
Theorem~\ref{theorem:muk:RHminor}] For any measure $\mu$ in $\cee$ denote its
logarithmic potential $U^{\mu}(x)$ as
\begin{equation}\label{def:logpot}
U^{\mu}(x) = -\int\log|x-s|\ \ud\mu(s).
\end{equation}
We will prove formula \eqref{weakcvg} for each sequence $\mu_{k,n}$ with $n$ of
the form $rm+l$, $l\in [0:r-1]$ fixed. Denote with $\kappa_{n}$ the leading
coefficient of the polynomial $B_{k,n}(x)$. We have uniformly for $x$ in
compact subsets of $\cee\setminus S_k$ that
\begin{multline}\label{eq:weakconv1}
\lim_{n\to\infty} \left(U^{\mu_{k,n}}(x)-\frac{1}{n}\log|\kappa_{n}|\right) =
-\lim_{n\to\infty}\frac{1}{n} \log
|B_{k,n}(x)| \\
= -\lim_{m\to\infty}\frac{1}{rm+l} \sum_{j=1}^m \log
|B_{k,rj+l}(x)/B_{k,r(j-1)+l}(x)| = \frac{1}{r}\log \prod_{j=0}^{k} |z_j(x)|=
U^{\mu_{k}}(x)+c,
\end{multline}
with $c$ a constant independent of $x$, and $\mu_{k}$ the measure in
\eqref{measure:k}. Here the first equality is obvious from the definitions
\eqref{def:logpot} and \eqref{counting:measure}, the second one follows by
telescopic cancelation, the third one follows by \eqref{ratioasy:Bkn}
and the fourth one by \cite[Prop.~5.10]{Del}. It is easy to see that
$\log|\kappa_{n}|/n$ is bounded from below as a function of $n$, due to
\eqref{normal:family:cond} and Prop.~\ref{prop:comb}.

Let $C_{0}(S_{k})$ denote the space of continuous functions on the star $S_{k}$
that vanish at infinity. Since $\|\mu_{k,n}\|\leq (p-k)/p$ for all $n$ (cf.
Lemma~\ref{lemma:RHminor:degree}), it follows from the Banach-Alaoglu theorem
that we can extract a subsequence from $\mu_{k,n}$ that converges in the
weak-star topology to a finite measure $\nu$ supported on $S_{k}$.
Let $x_{0}\in\cee\setminus S_{k}$ be a fixed point. From the weak-star
convergence and \eqref{eq:weakconv1} we deduce that for every
$x\in\cee\setminus S_{k}$,
\begin{equation}\label{eq:weakconv2}
\int \log \left|\frac{x_{0}-s}{x-s}\right|\ \ud
\nu(s)=U^{\mu_{k}}(x)-U^{\mu_{k}}(x_{0})=\frac{1}{r}\, \Re \Big(\log
\prod_{j=0}^{k}z_{j}(x)\Big)+\wtil c,
\end{equation}
where $\log \prod_{j=0}^{k}z_{j}(x)$ is a holomorphic branch of the logarithm
of $\prod_{j=0}^{k}z_{j}(x)$ and $\wtil c$ is some constant. Note that
$\phi(s):=\log \left|\frac{x_{0}-s}{x-s}\right|\in C_{0}(S_{k})$, so the
weak-star convergence indeed applies.

We claim that
\begin{equation}\label{eq:weakconv3}
\int_{S_{k}} \log (1+|s|)\,\ud \nu(s)<\infty.
\end{equation}
This will be justified at the end of the proof and now we complete the argument
as follows. From \eqref{eq:weakconv3} we obtain that $U^{\nu}$ is well-defined
and superharmonic in $\cee$, and in particular we can replace the first
integral in \eqref{eq:weakconv2} by $U^{\nu}(x)-U^{\nu}(x_{0})$. Note also that
the last relation in \eqref{eq:weakconv1}, which is in fact valid for all
$x\in\cee$, implies that $U^{\mu_{k}}$ is continuous everywhere in the complex
plane. This in turn implies, using \eqref{eq:weakconv2} and the
superharmonicity of $U^{\nu}$, that $U^{\nu}$ is bounded on every compact
segment of $S_{k}$. Now we are in a position to apply Theorem II.1.4 from
\cite{SaffTotik}, which gives $\mu_{k}=\nu$.

Now we justify \eqref{eq:weakconv3}. This is equivalent to say that
$U^{\nu}(x)>-\infty$ for any fixed $x\in \cee\setminus S_{k}$. It is clear that
we can construct a non-increasing sequence of functions $(k_{m}(y))_{m\in\enn}$
in $C_{0}(S_{k})$ satisfying $k_{m}(y)=\log\,(1/|x-y|)$ whenever $\log
(1/|x-y|)\geq -m$ and $k_{m}(y)\geq -m$ for all $y\in S_{k}$. Applying a
standard monotone convergence theorem argument to this sequence $k_{m}$
together with \eqref{eq:weakconv1} and the weak-star convergence to $\nu$, it
is easy to deduce that $U^{\nu}(x)\geq U^{\mu_{k}}(x)+c$, for some other
constant $c$.

Summarizing, we obtain that \eqref{weakcvg} is valid for every $\phi\in
C_{0}(S_{k})$. Since $\|\mu_{k,n}\|\leq (p-k)/p$ and $\|\mu_{k}\|=(p-k)/p$, the
convergence in the weak-star topology of $\mu_{k,n}$ to $\mu_{k}$ implies that
the sequence $\mu_{k,n}$ is tight. This implies again by a standard argument
that \eqref{weakcvg} is also valid for bounded continuous functions on $S_{k}$.
\end{proof}

\begin{remark}\label{remark:weak:RHminors}
Due to the interlacing properties described in
Theorem~\ref{theorem:interlace:gen}, it is easy to see that the conclusion of
Theorem~\ref{theorem:muk:RHminor} remains valid for the zeros of general
Riemann-Hilbert minors $P^{(n+i_{0},n+i_{1},\ldots,n+i_{k})}$, with
$i_{j}\in\mathbb{Z}$, $j\in[0:k]$, a fixed set of indices
satisfying~\eqref{cond:indices}.
\end{remark}

For later use, we state the next lemma.

\begin{lemma}\label{lemma:PklnPk}
Fix $0\leq k<l\leq p$. Uniformly for $x$ in compact
subsets of $\cee\setminus S_{k}$, we have
\begin{equation}\label{geneig:Nikhier}
\lim_{n\to\infty}\frac{B_{k,l,rn}(x)}{B_{k,rn}(x)} =
\begin{vmatrix}
\til{f}_{0}(z_0(x),x) & \ldots & \til{f}_{0}(z_{k}(x),x) \\
\vdots & & \vdots \\
\til{f}_{k-1}(z_{0}(x),x) & \ldots & \til{f}_{k-1}(z_k(x),x)\\
\til{f}_{l}(z_{0}(x),x) & \ldots & \til{f}_{l}(z_k(x),x)
\end{vmatrix}/\begin{vmatrix}
\til{f}_{0}(z_0(x),x) & \ldots & \til{f}_{0}(z_k(x),x) \\
\vdots & & \vdots \\
\til{f}_{k-1}(z_{0}(x),x) & \ldots & \til{f}_{k-1}(z_k(x),x)\\
\til{f}_{k}(z_{0}(x),x) & \ldots & \til{f}_{k}(z_k(x),x)
\end{vmatrix},
\end{equation}
where the functions $\til{f}_{j}(z,x)$ are defined in \eqref{def:fktilde}, and
$B_{k,l,n}(x):=B^{(n-l,n-k+1,\ldots,n-1,n)}(x)$. Denoting with $\wtil{\mathcal
A}_k$ the set of $x\in\cee$ for which the denominator in \eqref{geneig:Nikhier}
is zero, then the set $\wtil{\mathcal A}_k$ is finite. For any $x\in
\wtil{\mathcal A}_k\setminus S_{k}$ the right hand side of
\eqref{geneig:Nikhier} has a removable pole at $x$.
\end{lemma}

\begin{proof}
Remark~\ref{remark:detnonzero:gen} shows that the determinants in the numerator
and denominator in \eqref{geneig:Nikhier} are both nonzero if
$x\in\cee\setminus(S_+\cup S_-)$. For any fixed such $x$, we obtain
\eqref{geneig:Nikhier} by taking determinants of suitable submatrices in
Prop.~\ref{prop:ratio} (with $j_i=i$ for all $i$).
Lemma~\ref{lemma:normalfamily} shows the convergence holds uniformly on compact
subsets of $\cee\setminus S_{k}$.

Finally, let us prove that $\wtil{\mathcal A}_k$ is finite. Let $\mathcal R$ be
the compact Riemann surface with $(k+1)!\binom{p+1}{k+1}$ sheets which are
labeled by the ordered $(k+1)$-tuples $(i_0,i_1,\ldots,i_k)$ in $[0:p]$. On the
sheet $(i_0,i_1,\ldots,i_k)$ we cut away all the sets $\Gamma_{i_j}$ and
$\Gamma_{i_{j}-1}$, $j\in [0:k]$. If we cross such a cut, then we move to the
sheet labeled by $(\til i_0,\til i_1,\ldots,\til i_k)$ where $z_{\wtil i_j}$ is
the analytic continuation of $z_{i_j}$ through the cut. Now in the denominator
of \eqref{geneig:Nikhier} we can replace the role of $z_0,\ldots,z_k$ by
$z_{i_0},\ldots,z_{i_k}$. The collection of all these functions yields a
meromorphic function on $\mathcal R$. Since we already know that this function
is not identically zero on the sheet $(0,1,\ldots,k)$, it can indeed have only
finitely many zeros on that sheet. (Note that $\mathcal R$ can be disconnected;
in that case we restrict ourselves to the connected component(s) involving the
sheet $(0,1,\ldots,k)$).
\end{proof}

\subsection{Proofs of Proposition~\ref{prop:jk} and Theorem~\ref{theorem:Gammak:star}}
\label{subsection:Gammak:star}

In this section we prove Prop.~\ref{prop:jk} and
Theorem~\ref{theorem:Gammak:star}. Note that we did not use
Theorem~\ref{theorem:Gammak:star} prior to the statement of
Prop.~\ref{prop:jk}. Moreover, it is a general fact \cite[Prop.~1.1]{Del} that
each set $\Gamma_k$ associated to a Hessenberg matrix $H$ consists of a finite
union of analytic arcs.

For each $k\in [0:p],$ the sequence $\left(B_{k,n}(x)/B_{k,n+r}(x)\right)_{n}$
for $n$ tending to infinity converges pointwise for
$x\in\cee\setminus(\bigcup_j\Gamma_j\cup S_{k})$, by
\eqref{ratioasy:prodSk:bis}, and it is a normal family in $\cee\setminus S_{k}$
by Lemma~\ref{lemma:normalfamily}, recall the definition of $S_{k}$
in~\eqref{def:Sk}. Therefore, the convergence in \eqref{ratioasy:prodSk:bis} is
in fact uniform on compact subsets of $\cee\setminus S_{k}$ and the limit
function $z_{j_0}(x)\ldots z_{j_{k}}(x)$ is analytic there. Applying this
observation subsequently for $k\in [0:p]$, we see that for each $x\in\cee$,
there exists a permutation $(\til z_j(x))_{j=0}^{p}$ of the set
$(z_j(x))_{j=0}^{p}$ so that $\til z_0$ is analytic in $\cee\setminus S_+$,
$\til z_0\til z_1$ is analytic in $\cee\setminus S_-$, $\til z_0\til z_1\til
z_2$ is analytic in $\cee\setminus S_+$, $\til z_0\til z_1\til z_2\til z_3$ is
analytic in $\cee\setminus S_-$, and so on (alternatingly with $S_+$ and
$S_-$). In fact $\til z_i:=z_{j_i}$, $i\in [0:p]$. We also deduce
from~\eqref{zk:infty:intro} that as $x\rightarrow\infty$, $\til
z_0(x)=x^{-r}+O(x^{-r-1})$ and $\til z_j(x)=O(x^{r/p})$, $j\in [1:p]$, since it
is clear that for $x$ sufficiently large, $\til z_{0}(x)=z_{0}(x)$.

\subsubsection{Proofs of Prop.~\ref{prop:jk} and
Theorem~\ref{theorem:Gammak:star}(a)}

The proof will proceed in a very similar way to the one in \cite[Section
4]{DLL}. For convenience, we will list the main highlights of the proof but we
will sometimes refer to \cite{DLL} for the details. We will assume without loss
of generality that $r$ is a multiple of $p+1$, see also
Remark~\ref{remark:rmultipleofplpusone} below.

We already observed that for each $k\in [0:p]$,
\begin{equation}\label{ratioasy:prodSk:bisbis}
\lim_{n\to\infty} B_{k,n}(x)/B_{k,n+r}(x)=\pm\lim_{n\to\infty}
P_{k,n}(x)/P_{k,n+r}(x)=\til z_{0}(x)\ldots \til z_{k}(x),\qquad
x\in\cee\setminus S_{k}.
\end{equation}
It follows from \eqref{ratioasy:prodSk:bisbis} and
Theorem~\ref{theorem:interlace}(a) that the functions $\til z_{i}$ satisfy the
symmetry property $\til z_{i}(\omega x)=\til z_{i}(x)$, where $\omega=\exp(2\pi
\ir/(p+1))$; recall that $r$ is a multiple of $p+1$. In accordance to this
symmetry property, we define the functions
\[
\til y_{k}(x)=\til z_{k}(x^{1/(p+1)}),\qquad k\in [0:p],
\]
where we take the principal branch of $x^{1/(p+1)}$. (By the symmetry property,
the choice of the branch is irrelevant.) Note
that $\til y_{k}(x)$ is analytic in $\cee\setminus\er$.

Let us introduce the notation
\[
\er_{k}:= (-1)^{k}\er_{+}, \qquad k\in [0:p].
\]
We now define a measure $s_k$ on $\er_{k}$ with density
\begin{equation}\label{rho:0} \ud s_k(x)=\frac{1}{2\pi \ir}\frac{p+1}{r}\left( \frac{\til
y_{k,+}'(x)}{\til y_{k,+}(x)}- \frac{\til y_{k,-}'(x)}{\til y_{k,-}(x)} \right)
\ud x,\qquad x\in \er_{k},
\end{equation}
$k\in [0:p-1]$, where the prime denotes the derivative with respect to $x$, and
where the $+$ and $-$ subscripts stand for the boundary values obtained from
the upper or lower half of the complex plane, respectively. Note that the
measure \eqref{rho:0} is well-defined except for finitely many $x$, and its
density is integrable near each endpoint of its support \cite{DK}. We claim
that $s_k$ is a real-valued (possibly signed) measure on $\er_{k}$ with total
mass
\begin{equation}\label{rho:1}
s_k(\er_{k}):=\int_{\er_{k}} \ud s_k(x) = \frac{p-k}{p},\qquad k\in [0:p-1].
\end{equation}
Indeed, since the polynomials $P_{k,n}$~\eqref{geneig:Pkn} have real
coefficients, it follows from \eqref{ratioasy:prodSk:bisbis} that $\til
z_i(\bar x)=\overline{\til z_i(x)}$, where the bar denotes complex conjugation.
This shows that $s_{k}$ is a real-valued measure. For
$x\in\cee\setminus\er_{k}$, we have
\begin{equation}\label{contourdef:1}
\int_{\er_{k}}\frac{\ud s_{k}(t)}{x-t}=\frac{1}{2\pi
\ir}\frac{p+1}{r}\sum_{j=0}^{k}\int_{\er_{k}}\frac{1}{x-t}\,\left(\frac{\til
y_{j,+}'(t)}{\til y_{j,+}(t)}- \frac{\til y_{j,-}'(t)}{\til
y_{j,-}(t)}\right)\ud t =-\frac{p+1}{r}\sum_{j=0}^{k}\frac{\til y'_{j}(x)}{\til
y_{j}(x)},
\end{equation}
where in the first equality we used the fact that $\sum_{j=0}^{k-1} \til
y_{j}'(x)/\til y_{j}(x)=(\log\prod_{j=0}^{k-1} \til y_{j}(x))'$ is analytic
across $\er_{k}$, and the second equality follows by contour deformation and
the residue theorem. From the behavior of the
functions $\til y_{j}(x)$ near infinity we see that the right hand side of
\eqref{contourdef:1} behaves as $\frac{p-k}{p}x^{-1}+o(x^{-1})$ as
$x\to\infty$. This implies \eqref{rho:1}.

We then obtain from \eqref{rho:0}--\eqref{rho:1} that
\begin{equation}\label{totalmass:1} \frac{1}{\pi}\frac{p+1}{r}\int_{\er_{k}}
\Im\left(\frac{\til y_{k,+}'(x)}{\til y_{k,+}(x)}\right) \ud x =
\frac{p-k}{p},\qquad k\in [0:p-1],
\end{equation}
with $\Im$ denoting the imaginary part of a complex number.

As in \cite{DLL}, we now turn to the construction of a second collection of
auxiliary measures. The functions $z_{k}(x)$ are unambiguously defined in the
complement of $\bigcup_{k=0}^{p-1}\Gamma_{k}$, which is a finite union of
analytic arcs. Consider the functions
\[
y_{k}(x):=z_{k}(x^{1/(p+1)}),\qquad k\in [0:p],
\]
where we take the principal branch of $x^{1/(p+1)}$. The $\pm$-boundary values
of $y_{k}$ and $y_{k}'$ are well-defined at almost every point $x\in
\wtil\Gamma_k:=\Gamma_k^{p+1}$. This allows us to introduce the measures
\begin{equation}\label{rho:3}
\ud\sigma_{k}(x):=\frac{1}{2\pi \ir}\frac{p+1}{r}\sum_{j=0}^{k}\left(
\frac{y_{j,+}'(x)}{y_{j,+}(x)}-\frac{y_{j,-}'(x)}{y_{j,-}(x)}\right)\ud
x,\qquad x\in\wtil\Gamma_k.
\end{equation}
The measure $\sigma_k$ is closely related to the measure $\mu_k$ in
\eqref{measure:k}. In fact, for any Borel set $B$,
\begin{equation}\label{measuresk:root}
\sigma_{k}(B) = \mu_k(h^{-1}(B))
\end{equation}
where $h$ is the map $x\mapsto x^{p+1}$. In particular, $\sigma_k$ is a
positive measure and
\begin{equation}\label{contourdef:2}
\sigma_{k}(\wtil\Gamma_k) = \frac{p-k}{p},\qquad k\in [0:p-1].
\end{equation}
Alternatively, \eqref{contourdef:2} could be proved directly by using the same
argument as in \eqref{rho:1}.


Now let us take a fixed open interval $J\subset\er$ that does not contain any
intersection points or endpoints of the analytic arcs constituting
$\wtil\Gamma_k$, for every $k$. We also ask $J$ not to contain isolated
intersection points of the sets $\wtil\Gamma_{k}$ with the real axis. Thus
there exists an open connected set $U\subset\cee$ such that $U\cap \er=J$ and
moreover $U\cap\wtil\Gamma_k$ is either empty or equal to $J$, for any $k\in
[0:p-1]$. The boundary values $y_{k,+}(x)$ for $x\in J$ are then uniquely
defined and they vary analytically with $x$.

On the interval $J$, there exist indices $ 0\leq m_1<m_2<\ldots <m_L< p$ such
that
\begin{multline}\label{clusters} |y_{0,+}(x)|=\ldots =
|y_{m_1,+}(x)|<|y_{m_1+1,+}(x)|=\ldots = |y_{m_2,+}(x)|<\ldots
\\ <|y_{m_L+1,+}(x)|=\ldots = |y_{p,+}(x)|,
\end{multline}
for all $x\in J$. We define $m_0:=-1$ and $m_{L+1}:=p$.

We will see later that $m_{k+1}-m_k\in\{1,2\}$ for all $k$, i.e., each
``cluster" $|y_{m_k+1,+}(x)|=\ldots = |y_{m_{k+1},+}(x)|$ in \eqref{clusters}
can only have length 1 or 2. The Cauchy-Riemann equations imply \cite{DLL}
\begin{equation}\label{pairing:1} \Im\left( \frac{y_{m_k+1,+}'(x)}{y_{m_k+1,+}(x)} \right)\geq \ldots \geq
\Im\left( \frac{y_{m_{k+1},+}'(x)}{y_{m_{k+1},+}(x)} \right),\qquad x\in J,
\quad k\in [0:L],
\end{equation}
and the numbers in \eqref{pairing:1} satisfy the pairing
\begin{equation}\label{pairing:2} \Im\left(
\frac{y_{m_k+j,+}'(x)}{y_{m_k+j,+}(x)} \right) = -\Im\left(
\frac{y_{m_{k+1}+1-j,+}'(x)}{y_{m_{k+1}+1-j,+}(x)} \right),\qquad
j=1,\ldots,m_{k+1}-m_k.
\end{equation}
The underlying
 reason for \eqref{pairing:2} is that for any $x\in\er$, the
numbers $y_{k,+}(x)$, $k\in [0:p]$, are either real or they come in complex
conjugate pairs. This is trivial if $x>0$. If $x<0$ it can be seen
e.g.\ with the help of
Lemma~\ref{lemma:rotsym}, taking into account that $r$ is a multiple of $p+1$.


Now we get
\begin{multline*}
\sum_{k=0}^{p-1} \frac{p-k}{p}\ \ \geq\ \ \sum_{k=0}^{p-1}\sigma_k(\er)\ \
\geq\ \ \frac{1}{2\pi}\frac{p+1}{r}\int_{\er}\sum_{k=0}^{p} \left|\Im\left(
\frac{y_{k,+}'(x)}{y_{k,+}(x)} \right)\right|\ud x\\ =
\frac{1}{2\pi}\frac{p+1}{r}\int_{\er}\sum_{k=0}^{p} \left|\Im\left( \frac{\til
y_{k,+}'(x)}{\til y_{k,+}(x)} \right)\right|\ud x \\ =
\frac{1}{\pi}\frac{p+1}{r}\sum_{k=0}^{p-1} \int_{\er_{k}}\left|\Im\left(
\frac{\til y_{k,+}'(x)}{\til y_{k,+}(x)} \right)\right|\ud x\geq
\sum_{k=0}^{p-1} \frac{p-k}{p},
\end{multline*}
where the first relation uses \eqref{contourdef:2} and the positivity of
$\sigma_k$, the second relation follows exactly like in \cite[Sec.~4]{DLL},
the third one follows since the numbers $\til y_k$ form a permutation of the
$y_k$, and the fifth relation is a consequence of \eqref{totalmass:1}. Finally,
the fourth relation uses that on $\er_+$ we have $\Im\left( \frac{\til
y_{2k,+}'(x)}{\til y_{2k,+}(x)} \right)=-\Im\left( \frac{\til
y_{2k+1,+}'(x)}{\til y_{2k+1,+}(x)} \right)$ and on $\er_-$ we have $\Im\left(
\frac{\til y_{2k,+}'(x)}{\til y_{2k,+}(x)} \right)=-\Im\left( \frac{\til
y_{2k-1,+}'(x)}{\til y_{2k-1,+}(x)} \right)$; for $\er_+$ this follows since
$$\sum_{j=2k}^{2k+1} \Im\left( \frac{\til y_{j,+}'(x)}{\til y_{j,+}(x)}
\right)= \frac{1}{2\ir}\sum_{j=2k}^{2k+1}\left( \frac{\til y_{j,+}'(x)}{\til
y_{j,+}(x)}-\frac{\til y_{j,-}'(x)}{\til y_{j,-}(x)} \right) =0,\qquad
x\in\er_+,
$$
due to the fact that $(\log \wtil y_{2k}\wtil y_{2k+1})'$ is analytic on
$\er_+$.

From the above chain of inequalities we obtain:

\begin{lemma}\label{lemma:chain}
\begin{enumerate}
\item[$(a)$] We have
\begin{equation}\label{Gammak:Delta}\wtil\Gamma_k\subset\er,\qquad
k\in [0:p-1].\end{equation}
\item[$(b)$] Clusters of length $\geq 3$ in \eqref{clusters} cannot occur.
\end{enumerate}
\end{lemma}

The proof is exactly as in \cite{DLL}.

From Lemma~\ref{lemma:chain}(a)--(b) we see that
$\wtil\Gamma_k\cap\wtil\Gamma_{k-1}$ contains at most finitely many points. We
also have that $(\log y_0\ldots y_{k-1})'$ is analytic in
$\cee\setminus\wtil\Gamma_{k-1}$, so in particular this holds on the interior
of each interval of $\wtil\Gamma_k$. Then \eqref{rho:3}--\eqref{contourdef:2}
imply that
\begin{equation}\label{totalmass:a:1}
\frac{1}{\pi}\frac{p+1}{r}\int_{\wtil\Gamma_{k}}
\Im\left(\frac{y_{k,+}'(x)}{y_{k,+}(x)}\right) \ud x =
\sigma_k(\wtil\Gamma_k)=\frac{p-k}{p},\qquad k\in [0:p-1].
\end{equation}
The measure $\sigma_{k}$ is non-trivial on each subarc of $\wtil \Gamma_{k}$ (see the proof of Lemma~\ref{lemma:chain}(a)).
From the positivity of $\sigma_k$ we also have
\begin{equation}\label{totalmass:a:3}
\Im\left(\frac{y_{k,+}'(x)}{y_{k,+}(x)}\right)\, \left\{\begin{array}{ll} >
0,&\qquad
x\in\inter(\wtil\Gamma_k),\\=-\Im\left(\frac{y_{k-1,+}'(x)}{y_{k-1,+}(x)}\right)\leq
0,& \qquad x\in\wtil\Gamma_{k-1},\\ =0,& \qquad
x\in\er\setminus\left(\wtil\Gamma_k\cup
\wtil\Gamma_{k-1}\right),\end{array}\right.
\end{equation}
where $\inter(\wtil\Gamma_k)$ denotes the interior of $\wtil \Gamma_{k}$ in the
topology of $\mathbb{R}$, where the first equality uses \eqref{pairing:2} and
Lemma~\ref{lemma:chain}(b).

Recall that $\wtil y_k(x)$ is analytic for $x\in\cee\setminus\er$. By
Lemma~\ref{lemma:chain}(a) the same holds for the function $y_k(x)$. Thus for
each fixed $k\in [0:p]$ we have that $\wtil y_k(x)=y_{j_k}(x)$ for all
$x\in\cee\setminus\er$ and for a certain $j_k$ which is independent of $x$.
From \eqref{totalmass:a:1}--\eqref{totalmass:a:3} and \eqref{totalmass:1} we
now easily find by induction on $k=0,1,\ldots$ that $j_k=k$ and moreover
$\wtil\Gamma_k\subset\er_k$. This proves Proposition~\ref{prop:jk} and
Theorem~\ref{theorem:Gammak:star}(a). $\bol$

\begin{remark}\label{remark:rmultipleofplpusone} In the above proof
we assumed that $r$ is a multiple of $p+1$. For general $r$, the symmetry
properties take the form $\wtil z_k(\omega x)=\omega^{-r} \wtil z_k(x)$ and
$z_k(\omega x)=\omega^{-r} z_k(x)$ (Recall Lemma~\ref{lemma:rotsym}). Then the
functions $\wtil y_k$ and $y_k$ have a jump on the whole of $\er_-$. But the
logarithmic derivatives do not have such a jump, therefore the proof goes
through in exactly the same way as above.
\end{remark}

\subsubsection{Proof of Theorem~\ref{theorem:Gammak:star}(b)--(c)}

Fix $k\in [0:p-1]$ and let $r$ be arbitrary. Let $I$ be an interval of
$\wtil\Gamma_k\subset\er_k$. We will assume that $\wtil\Gamma_k$ is not the
whole set $\er_+$ or $\er_-$ since otherwise there is nothing to prove. We
claim that
\begin{equation}\label{totalmass:interval}
\frac{r}{p+1}\sigma_k(I) = \left\{\begin{array}{ll} \in\enn,& \textrm{if
}I\cap\{0,\infty\}=\emptyset,\\
\in\enn/(p+1),& \textrm{if
}0\in I,\\
\in\enn/p,& \textrm{if }\infty\in I.
\end{array}\right.
\end{equation}
Let us assume this for the moment. By breaking each interval $I$ of
$\wtil\Gamma_k$ in smaller subintervals if necessary, in such a way that
\eqref{totalmass:interval} remains valid, we may assume that the left hand side
of \eqref{totalmass:interval} always lies in the range $(0,1]$. Then we have
from the total mass of $\sigma_k$ in \eqref{contourdef:2} that
\begin{equation}\label{intervals:1} a+\frac{b}{p+1}+ \frac{c}{p}\ =\
\frac{r}{p+1}\frac{p-k}{p}\ =\ \frac{k+1}{p+1}r-\frac{kr}{p},
\end{equation}
where $a$ denotes the number of intervals $I$ of $\wtil\Gamma_k$ for which the
left hand side of \eqref{totalmass:interval} equals $1$, and where $b\in [0:p]$
and $c\in [0:p-1]$ are nonzero \emph{only if} there is an interval $I\subset
\wtil\Gamma_k$ containing $0$ or $\infty$ respectively and with the left hand
side of \eqref{totalmass:interval} being $<1$. Since $p$ and $p+1$ are coprime,
from \eqref{intervals:1} we deduce that
$$\frac{b}{p+1}=
\frac{k+1}{p+1}r-\left\lfloor\frac{k+1}{p+1}r\right\rfloor,\qquad\textrm{and }
\frac{c}{p}= \left\lceil\frac{kr}{p}\right\rceil-\frac{kr}{p}.$$ Inserting this
in \eqref{intervals:1} we get
$$ a= \left\lfloor
\frac{k+1}{p+1}r\right\rfloor-\left\lceil\frac{kr}{p}\right\rceil.
$$
We then find for the total number $n_k$ of intervals of $\wtil\Gamma_k$ that
$$ n_k = \left\lfloor \frac{k+1}{p+1}r \right\rfloor -\left\lceil
\frac{kr}{p} \right\rceil+\mathbf{1}_{b\neq 0} +\mathbf{1}_{c\neq 0}=
\left\lceil \frac{k+1}{p+1}r \right\rceil -\left\lfloor \frac{kr}{p}
\right\rfloor,
$$
where the indicator function $\mathbf{1}_{x\neq 0}$ equals 1 if $x\neq 0$ and
zero otherwise, and where the second equality uses that $b\neq 0$ if and only
if $(k+1)r/(p+1)\not\in\enn$ and similarly $c\neq 0$ if and only if
$kr/p\not\in\enn\cup\{0\}$. This proves
Theorem~\ref{theorem:Gammak:star}(b)--(c).

Finally we prove \eqref{totalmass:interval}. Due to \eqref{measuresk:root} it
will be enough to prove that
\begin{equation}\label{totalmass:interval:star}
r\mu_k(J) = \left\{\begin{array}{ll} \in\enn,& \textrm{if
}J\cap\{\infty\}=\emptyset,\\
\in\enn/p,& \textrm{if }\infty\in J,
\end{array}\right.
\end{equation}
for any connected component $J$ of $\Gamma_k\subset S_k$. Thus $J$ is either a
line segment on $S_k\setminus\{0\}$ (there are $p+1$ rotations of such a
segment), or it is a set of the form $J=\{x\in S_k\mid |x|\leq a\}$ for some
$a>0$.

The first statement of \eqref{totalmass:interval:star} follows from
\cite[Prop~2.10]{Del}. Let us check it directly if $J\subset S_k$ is a line
segment of the form $[a,b]$ with $a,b\not\in\{0,\infty\}$. From the definition
\eqref{measure:k} of $\mu_k$ it is easy to see that
\begin{equation}\label{totalmass:interval:1}r\mu_k(J) = \frac{1}{2\pi}\lim_{x\to b,x\in J}
\left(\arg \prod_{j=0}^k z_{j,+}(x)-\arg \prod_{j=0}^k z_{j,-}(x)\right)
\end{equation}
where we take the argument function $\arg$ so that $\arg \prod_{j=0}^k
z_{j}(x)$ is continuous in $U\setminus J$ with $U$ a complex neighborhood of
$[a,b)$ ($U$ excludes $b$). This is possible since $\prod_{j=0}^k z_{j}(x)$ is analytic and
nonzero in $\cee\setminus \Gamma_k$. But then the expression between brackets
in \eqref{totalmass:interval:1} is an integral multiple of $2\pi$, yielding the
first statement in \eqref{totalmass:interval:star}.

To prove the second statement of \eqref{totalmass:interval:star}, note that
\eqref{totalmass:interval:1} remains valid if $b$ lies at $\infty$. In that
case, the behavior of the functions $z_k(x)$ near infinity in
\eqref{zk:infty:intro:bis} implies that the expression between brackets in \eqref{totalmass:interval:1} is
an integral multiple of $2\pi/p$. This proves \eqref{totalmass:interval:star}.
$\bol$

\section{Nikishin system}
\label{section:Nikishin}

In this section we prove Theorem~\ref{theorem:Nik:property} on the connection
with Nikishin systems. We start by recalling some ideas in \cite{AKVI}.

\subsection{Multiple orthogonality relations}

For any $l\in [0:p]$, we define the sequence of monic polynomials
$(Q_{n,l}(x))_{n=l}^{\infty}$ by the recurrence relation
\begin{equation}\label{recurrencerel:k} xQ_{n,l}(x)=Q_{n+1,l}(x)+a_{n-p}\,
Q_{n-p,l}(x),\qquad n\geq l,\end{equation} with initial conditions
\begin{equation}\label{initialcond:k}
Q_{l,l}(x)\equiv 1,\qquad Q_{l-1,l}(x)\equiv\cdots\equiv Q_{l-p,l}(x)\equiv
0.\end{equation} Note that $\deg Q_{n,l}=n-l$ and that 
$Q_{n,0}(x)\equiv Q_n(x)$. Moreover, the $p+1$ sequences
$(Q_{n,l}(x))_{n=0}^{\infty}$ form a basis for the space of all solutions
$(q_{n})_{n=0}^{\infty}$ to the difference equation
\[
x q_{n}=q_{n+1}+a_{n-p}\,q_{n-p},\qquad n\geq p.
\]

\begin{lemma}\label{lemma:nuk} (The measures $\nu_1,\ldots,\nu_p$; see \cite{AKVI}:)
Suppose that $a_n>0$ for all $n$ and the numbers $a_n$ are uniformly bounded.
There exists an increasing sequence of positive integers $(n_j)_{j=0}^{\infty}$
such that for any fixed $l\in [1:p]$, we have
\begin{itemize}
\item[(a)]
\begin{equation}\label{nuk:Helly}
\lim_{j\to\infty} \frac{Q_{(p+1)n_j,l}(x)}{Q_{(p+1)n_j}(x)} = \int
\frac{\ud\nu_l(t)}{x-t},
\end{equation}
uniformly for $x$ in compact subsets of $\cee\setminus S_+$, where $\nu_l$ is a
compactly supported measure on $S_+$.
\item[(b)] The moments of $\nu_l$ are uniquely determined from the condition
\eqref{nuk:Helly}, independently of the choice of the sequence
$(n_j)_{j=0}^{\infty}$.
\item[(c)]
The measure $\nu_l$ can be written as
\begin{equation}\label{nuk:symmetry}
\ud\nu_l(t) = t^{1-l} \ud\til\nu_l(t^{p+1}),
\end{equation}
for a compactly supported, positive measure $\til\nu_l$ supported on $S_+^{p+1}
= \er_+$. Thus for $l=1$ the measure $\nu_l$ is rotationally invariant under
rotations over $2\pi/(p+1)$ while for $l>1$ it is rotationally invariant up to
a monomial factor.\end{itemize}
\end{lemma}

Lemma~\ref{lemma:nuk} was shown by Aptekarev-Kalyagin-Van Iseghem \cite{AKVI}.
The key fact for \eqref{nuk:Helly} is that for any $n\in\enn$ and $l\in [1:p]$
the zeros of $Q_{(p+1)n,l}(x)$ and $Q_{(p+1)n}(x)$ interlace (in a suitable
sense) on the star-like set $S_+$. The existence of a sequence
$(n_j)_{j=0}^{\infty}$ such that \eqref{nuk:Helly} holds then follows from the
Helly selection theorem, see \cite{AKVI}.

Recall from Theorem~\ref{theorem:Qn:as:mop:0} that the polynomials $Q_n(x)$ are
multiple orthogonal with respect to the measures $\nu_1,\ldots,\nu_p$ defined
in \eqref{nuk:Helly}, in the sense of \eqref{Qn:MOP:0}.

\smallskip

We will assume throughout this section that we are in the exactly periodic
case~\eqref{periodic:exact:twodiag} and that~\eqref{ordering:as} holds. We can
assume that the sequence $n_j$ in \eqref{nuk:Helly} is such that each
$(p+1)n_j$ is a multiple of $r$; this follows directly from the freedom in
choosing a convergent subsequence in the proof of Lemma \ref{lemma:nuk} in
\cite{AKVI}. For a fixed $l\in [0:p]$, let
\[
T_{n}(x):=Q_{n+l,l}(x),\qquad n\geq 0.
\]
By \eqref{recurrencerel:k} and \eqref{initialcond:k}, this sequence satisfies
\begin{equation}\label{recurrenceT}
x T_{n}(x)=T_{n+1}(x)+a_{n+l-p}T_{n-p}(x),\qquad n\geq 0,
\end{equation}
with initial conditions
\begin{equation}\label{initialcondT}
T_{0}(x)\equiv 1,\qquad T_{-1}(x)\equiv\cdots\equiv T_{-p}(x)\equiv 0.
\end{equation}
It follows that the block Toeplitz symbol associated with
$(T_{n})_{n=0}^{\infty}$ is $Z^{-l} F(z,x) Z^{l}$, with $Z$ and $F$ given by
\eqref{def:Sshift} and \eqref{symbol}, respectively.

Using Theorem~\ref{theorem:Widomformula2} and simple considerations, we deduce
that for any index sequence $(n_j)_{j=0}^{\infty}$ as described above, we have
\begin{equation}\label{def:fk1}
\lim_{j\to\infty} \frac{Q_{(p+1)n_j,l}(x)}{Q_{(p+1)n_j}(x)}
=\lim_{j\rightarrow\infty}\frac{T_{(p+1)n_{j}-l}(x)}{Q_{(p+1)n_j}(x)}=
\frac{f_l(z_0(x),x)}{f_0(z_0(x),x)},\qquad l\in [0:p],
\end{equation}
uniformly on compact subsets of $\cee\setminus S_{+}$,
where the functions $f_l$ are the following minors of the block Toeplitz
symbol:
\begin{equation}\label{fk:def}
\begin{array}{l}
f_0(z,x)=(-1)^{r} z^{-1}\det F^{r-1,0}(z,x),\\[0.1cm]
f_{l}(z,x)=(-1)^{l} \det F^{l-1,0}(z,x),\qquad l\in [1:r].
\end{array}
\end{equation}
Note that the functions $f_{l}(z_{0}(x),x)$, $l\in[0:p]$, are analytic in
$\mathbb{C}\setminus\Gamma_{0}$, and that $f_0(z,x)$ has an extra factor
$z^{-1}$ in comparison to the other functions $f_l(z,x)$. Let
$$
\mathcal{A}_{0}:=\{x\in\cee\setminus\Gamma_{0}:
\frac{f_{l}(z_{0}(x),x)}{f_{0}(z_{0}(x),x)}\,\,\mbox{has a non-removable pole
at}\,\,x\mbox{ for some $l\in [1:p]$}\}.
$$
From the statement of Theorem~\ref{theorem:Widomformula2} we know that the set
$\mathcal{A}_{0}$ is finite. Moreover, since the functions
$Q_{n,l}(x)/Q_{n}(x)$ are analytic on $\mathbb{C}\setminus S_{+}$, we deduce
from \eqref{def:fk1} that $\mathcal{A}_{0}\subset S_{+}$.


\subsection{Formal Nikishin system}

In this section we will introduce a hierarchy of functions $f_{l,k}$, $0\leq k<
l\leq p$, which will be identified later as the Cauchy transforms of certain
measures that form the different layers of a Nikishin system on
$(\Gamma_{0},\ldots,\Gamma_{p-1})$.

\smallskip

From \eqref{nuk:Helly} and \eqref{def:fk1} we obtain
\begin{equation}\label{Cauchytransfnu}
\int \frac{\ud\nu_l(t)}{x-t} =
\frac{f_l(z_0(x),x)}{f_0(z_0(x),x)}=:f_{l,0}(z_0(x),x),\qquad l\in [1:p],
\end{equation}
for $x\in\cee\setminus (\Gamma_{0}\cup\mathcal{A}_{0})$. The measures $\nu_{l}$
and the functions $f_{l,0}$ will form layer $0$ of the Nikishin hierarchy, as
we will show later in this section. We also deduce that the measures $\nu_{l}$
are supported on $\Gamma_{0}\cup\mathcal{A}_{0}$.

We consider the functions $$
\frac{f_{l,0}(z_{0,+}(x),x)-f_{l,0}(z_{0,-}(x),x)}{f_{1,0}(z_{0,+}(x),x)-f_{1,0}(z_{0,-}(x),x)},\qquad
x\in\Gamma_0,$$ $l\in [2:p]$. (These expressions will be well-defined as we
will show in the next section.) The relations $z_{0,\pm}=z_{1,\mp}$ on
$\Gamma_0$ imply that these functions can be meromorphically extended to
$\cee\setminus\Gamma_1$; we denote the resulting functions by
\begin{equation}\label{cauchy:hier2}
f_{l,1}(z_0(x),z_1(x),x) :=
\frac{f_{l,0}(z_{1}(x),x)-f_{l,0}(z_{0}(x),x)}{f_{1,0}(z_{1}(x),x)-f_{1,0}(z_{0}(x),x)},
\end{equation}
and we observe that $f_{l,1}$ is a symmetric function of its two arguments
$z_0$ and $z_1$. This will form layer $1$ of the Nikishin hierarchy.

Next we consider the functions
$$\frac{f_{l,1}(z_{0}(x),z_{1,+}(x),x)-f_{l,1}(z_{0}(x),z_{1,-}(x),x)}{f_{2,1}(z_{0}(x),z_{1,+}(x),x)
-f_{2,1}(z_{0}(x),z_{1,-}(x),x)},\qquad x\in\Gamma_1,$$ $l\in [3:p]$. They can
be extended to $\cee\setminus\Gamma_2$ by the functions
\begin{equation}\label{cauchy:hier3}
f_{l,2}(z_0(x),z_1(x),z_2(x),x) :=
\frac{f_{l,1}(z_{0}(x),z_{2}(x),x)-f_{l,1}(z_{0}(x),z_{1}(x),x)}{f_{2,1}(z_{0}(x),z_{2}(x),x)
-f_{2,1}(z_{0}(x),z_{1}(x),x)},
\end{equation}
and we observe that $f_{l,2}$ is a symmetric function of its three arguments
$z_0,z_1,z_2$ (this is a bit harder to see now). This will form layer $2$ of
the Nikishin hierarchy.

We can continue this procedure and set \begin{multline*}
f_{l,k}(z_0(x),\ldots,z_{k}(x),x) \\ =
\frac{f_{l,k-1}(z_0(x),\ldots,z_{k-2}(x),z_k(x),x)-f_{l,k-1}(z_0(x),\ldots,z_{k-2}(x),z_{k-1}(x),x)}
{f_{k,k-1}(z_0(x),\ldots,z_{k-2}(x),z_k(x),x)-f_{k,k-1}(z_0(x),\ldots,z_{k-2}(x),z_{k-1}(x),x)},
\end{multline*}
for $l\in [k+1:p]$ and $k\in [1:p-1]$, using induction on $k$. It allows a
determinantal formula:

\begin{lemma}\label{lemma:linalgdet} Consider the functions $f_l$,
$l\in [0:p]$ in \eqref{fk:def}. Define the hierarchy of functions $f_{l,k}$,
$0\leq k< l\leq p$, as explained above. Abbreviating $f_l(z_k(x)) :=
f_l(z_k(x),x)$ for each $l$, we have
\begin{equation}\label{cauchy:hierl:det}
f_{l,k}(z_0(x),\ldots,z_k(x),x) = \left|\begin{array}{@{\hspace{0cm}} c
@{\hspace{0.2cm}} c @{\hspace{0.2cm}} c @{\hspace{0cm}}}
f_{0}(z_0(x)) & \ldots & f_{0}(z_{k}(x)) \\
\vdots & & \vdots \\
f_{k-1}(z_{0}(x)) & \ldots & f_{k-1}(z_k(x))\\
f_{l}(z_{0}(x)) & \ldots & f_{l}(z_k(x))
\end{array}\right|/\left|\begin{array}{@{\hspace{0cm}} c @{\hspace{0.2cm}} c @{\hspace{0.2cm}} c @{\hspace{0cm}}}
f_{0}(z_0(x)) & \ldots & f_{0}(z_k(x)) \\
\vdots & & \vdots \\
f_{k-1}(z_{0}(x)) & \ldots & f_{k-1}(z_k(x))\\
f_{k}(z_{0}(x)) & \ldots & f_{k}(z_k(x))
\end{array}\right|.
\end{equation}
We also have
\begin{multline}\label{cauchy:hierl:detbis}
f_{l,k,+}(z_0(x),\ldots,z_k(x),x)-f_{l,k,-}(z_0(x),\ldots,z_k(x),x) =
\\ -\frac{\left|\begin{array}{@{\hspace{0cm}} c @{\hspace{0.1cm}} c @{\hspace{0.1cm}} c @{\hspace{0cm}}}
f_{0}(z_0(x)) & \ldots & f_{0}(z_{k-1}(x)) \\
\vdots & & \vdots \\
f_{k-1}(z_{0}(x)) & \ldots & f_{k-1}(z_{k-1}(x))\\
\end{array}\right|\left|\begin{array}{@{\hspace{0cm}} c @{\hspace{0.1cm}} c @{\hspace{0.1cm}} c @{\hspace{0cm}}}
f_{0}(z_0(x)) & \ldots & f_{0}(z_{k+1}(x)) \\
\vdots & & \vdots \\
f_{k}(z_{0}(x)) & \ldots & f_{k}(z_{k+1}(x))\\
f_{l}(z_{0}(x)) & \ldots & f_{l}(z_{k+1}(x))
\end{array}\right|}{\left|\begin{array}{@{\hspace{0cm}} c @{\hspace{0.1cm}} c @{\hspace{0.1cm}} c @{\hspace{0cm}}}
f_{0}(z_0(x)) & \ldots & f_{0}(z_{k}(x)) \\
\vdots & & \vdots \\
f_{k}(z_{0}(x)) & \ldots & f_{k}(z_{k}(x))
\end{array}\right|\left|\begin{array}{@{\hspace{0cm}} c @{\hspace{0.1cm}} c @{\hspace{0.1cm}} c @{\hspace{0.1cm}} c @{\hspace{0cm}}}
f_{0}(z_0(x)) & \ldots & f_{0}(z_{k-1}(x)) & f_{0}(z_{k+1}(x)) \\
\vdots & & \vdots & \vdots \\
f_{k}(z_{0}(x)) & \ldots & f_{k}(z_{k-1}(x)) & f_{k}(z_{k+1}(x))
\end{array}\right|},
\end{multline}
for $x\in\Gamma_{k}$, where in the right hand side of
\eqref{cauchy:hierl:detbis} we define the values $z_j(x)$ as the limiting
values obtained from the $+$-side of $\Gamma_{k}$ (picking another labeling of
the $z_j(x)$ so that \eqref{ordering:rootszk} holds can only change the sign in
\eqref{cauchy:hierl:detbis}), and where we set the determinant of an empty
matrix as~$1$.
\end{lemma}

\begin{proof} The determinantal formulas follow by induction on $k=0,1,2,\ldots$ by means of
a basic linear algebra calculation using Sylvester's determinant identity
\cite{Gant}. See also \cite[Sec.~8]{AKS}.
\end{proof}

\begin{remark}\label{analyticity:flk}
As in the proof of Lemma~\ref{lemma:PklnPk}, we see that the
denominator in the right-hand side
of~\eqref{cauchy:hierl:det} vanishes only for finitely many $x\in\cee$. Taking
into account the relations $z_{i,\pm}=z_{i+1,\mp}$ on $\Gamma_{i}$,
$i\in[0:k-1]$, we see that the ratio in~\eqref{cauchy:hierl:det} is in fact
analytic in $\mathbb{C}\setminus(\Gamma_{k}\cup\mathcal{A}_{k})$, where
$\mathcal{A}_{k}$ is a finite set in $\mathbb{C}\setminus\Gamma_{k}$.
\end{remark}

\begin{lemma}\label{cor:degree} Let $r$ be a multiple of $p$ and assume the ordering \eqref{ordering:as}.
For any $0\leq k<l\leq p$ there exists $C\neq 0$ such that the following
asymptotics hold for $x\to\infty$:
$$ \begin{vmatrix}
f_{0}(z_0(x)) & \ldots & f_{0}(z_{k}(x)) \\
\vdots & & \vdots \\
f_{k-1}(z_{0}(x)) & \ldots & f_{k-1}(z_k(x))\\
f_{l}(z_{0}(x)) & \ldots & f_{l}(z_k(x))
\end{vmatrix}/\begin{vmatrix}
f_{0}(z_0(x)) & \ldots & f_{0}(z_k(x)) \\
\vdots & & \vdots \\
f_{k-1}(z_{0}(x)) & \ldots & f_{k-1}(z_k(x))\\
f_{k}(z_{0}(x)) & \ldots & f_{k}(z_k(x))
\end{vmatrix}
= C x^{k-l}(1+O(x^{-p-1})).
$$
\end{lemma}

The proof of Lemma~\ref{cor:degree} is postponed to
Section~\ref{subsection:proof:degree}.

For convenience, we define a new symbol:
\begin{equation}\label{def:Fhat}
\what{F}(z,x):=P_{r} F(z,x)^{T} P_{r},
\end{equation}
where $P_{r}$ is the $r\times r$ permutation matrix that consists of $1$'s in
the main antidiagonal and $0$'s elsewhere, i.e., the $(i,j)$ entry of $P_r$
equals $\delta_{i+j-r+1}$, for $i,j\in [0:r-1]$.
Note that $\what{F}(z,x)$ is the reflection of $F(z,x)$ with respect to its
main antidiagonal. We can rewrite \eqref{fk:def} as
\begin{equation}\label{fk:def2}
\begin{array}{l}
f_0(z,x)=(-1)^{r} z^{-1}\det \what{F}^{r-1,0}(z,x),\\[0.1cm]
f_{l}(z,x)=(-1)^{l} \det \what{F}^{r-1,r-l}(z,x),\qquad l\in [1:r].
\end{array}
\end{equation}
We also define the functions
\begin{equation}\label{def:fktilde}
\begin{array}{l}
\til{f}_{0}(z,x)=(-1)^{r}z^{-1}\det F^{r-1,0}(z,x),\\[0.1cm]
\til{f}_{l}(z,x)=(-1)^{l}\det F^{r-1,r-l}(z,x),\qquad l\in [1:r].
\end{array}
\end{equation}

\subsection{Proof of Theorem~\ref{theorem:Nik:property}}
\label{subsection:Nikproof}

Lemma~\ref{lemma:linalgdet} asserts that the measures $\nu_j$ form a
\emph{formal} Nikishin system in the sense of \cite{AKS}. To prove
Theorem~\ref{theorem:Nik:property}, we will now show that they are a
\emph{true} Nikishin system.

\begin{theorem}\label{theorem:mukl} (Nikishin property.)
Let $H$ be the two-diagonal Hessenberg matrix~\eqref{H:entries:twodiag}, with
entries $a_{n}>0$ that satisfy
\eqref{periodic:exact:twodiag}--\eqref{ordering:as}. For each pair of indices
$k,l$ with $0\leq k< l\leq p$,
\begin{equation}\label{Cauchytransf:property}
f_{l,k}(z_0(x),\ldots,z_k(x),x)=\int\frac{\ud\nu_{l,k}(t)}{x-t},
\end{equation}
for a measure $\nu_{l,k}$ supported on $\Gamma_k\cup\mathcal{A}_{k}$, where
$\mathcal{A}_{k}$ is a finite subset of $S_{k}\setminus\Gamma_{k}$
($S_{-}\setminus\Gamma_{k}$) if $k$ is even (odd). The measure $\nu_{l,k}$
takes the form \eqref{nukl:symmetry}, for a measure $\til\nu_{l,k}$ with
constant sign supported on $\er_{+}$ ($\er_{-}$) if $k$ is even (odd).
\end{theorem}
\begin{proof}
By \eqref{Cauchytransfnu} we know that \eqref{Cauchytransf:property} is valid
for $k=0$ and $l\in [1:p]$. We also know already that the measures $\nu_{l,0}$
are supported on $\Gamma_{0}\cup\mathcal{A}_{0}$, with $\mathcal{A}_{0}$ a
finite subset of $S_{+}\setminus\Gamma_{0}$, and that \eqref{nukl:symmetry}
holds for $k=0$.

In what follows we are going to work with the polynomials associated with the
symbol $\what{F}$ introduced in \eqref{def:Fhat}. That is, we consider now the
two-diagonal Hessenberg operator $\what{H}$ with periodic structure whose first
$r$ coefficients are given in the following order:
\[
a_{r-p-1}, a_{r-p-2},\ldots, a_{1}, a_{0}, a_{r-1},\ldots, a_{r-p}.
\]
We associate with the new operator $\what{H}$ the polynomials $P_{k,l,n}$ as
defined in Section~\ref{subsection:Nik}. These are the polynomials we will
employ below.

Let $1\leq k< l\leq p$. Applying Theorem~\ref{theorem:interlace:kl},
\begin{equation*}\label{proofNik:a}
\frac{P_{k,l,n}(x)}{P_{k,n}(x)}=x^{k-l}\,\frac{\wtil P_{k,l,n}(x^{p+1})}{\wtil
P_{k,n}(x^{p+1})},
\end{equation*}
where the zeros of $\wtil P_{k,n}$ and $\wtil P_{k,l,n}$ lie in $\er_{+}$
($\er_{-}$) if $k$ is even (odd), and are weakly interlacing.

Let us denote by $d_{n}$ the degree of $\wtil P_{k,n}$. Thanks to the weak
interlacing property, we know that either $\deg \wtil P_{k,l,n}=d_{n}$ or $\deg
\wtil P_{k,l,n}=d_{n}+1$. In any case, we can write
\[
\frac{\wtil P_{k,l,n}(z)}{z \wtil
P_{k,n}(z)}=\alpha_{-1,n}+\sum_{i=0}^{d_{n}}\frac{\alpha_{i,n}}{z-x_{i,n}},
\]
where $x_{0,n}:=0$, $\{x_{i,n}\}_{i=1}^{d_{n}}$ denotes the zeros of $\wtil
P_{k,n}$, and for $i\geq 0$, we set $\alpha_{i,n}=0$ if $z-x_{i,n}$ is a common
factor of the numerator and denominator. It also follows from the interlacing
property that all the coefficients $\{\alpha_{i,n}\}_{i=0}^{d_{n}}$ have the
same sign.


Assume for the moment that $k$ is even, so the zeros of $\wtil P_{k,n}$ lie in
$\er_{+}$. Let $\nu_{l,k,n}$ be the discrete measure supported on
$\{x_{i,n}\}_{i=0}^{d_{n}}$ with mass $\alpha_{i,n}$ at $x_{i,n}$. Hence
\begin{equation}\label{Cauchy:discrete}
\frac{\wtil P_{k,l,n}(z)}{z \wtil
P_{k,n}(z)}=\alpha_{-1,n}+\int\frac{\ud\nu_{l,k,n}(t)}{z-t}.
\end{equation}
It is easy to check that $\alpha_{-1,n}\leq 0$ if $\nu_{l,k,n}\geq 0$, and
$\alpha_{-1,n}\geq 0$ if $\nu_{l,k,n}\leq 0$. Therefore the function
\eqref{Cauchy:discrete} maps $(-\infty,0)$ into $(-\infty,0)$ if $\nu_{l,k,n}$
is positive, and maps $(-\infty,0)$ into $(0,\infty)$ if $\nu_{l,k,n}$ is
negative. Moreover, it maps the upper half plane into the lower half plane
(upper half plane) if $\nu_{l,k,n}$ is positive (negative).

An important ingredient in our proof is formula \eqref{geneig:Nikhier}, which
certainly applies in our situation. We should apply this formula for the
$B$-polynomials associated with the operator $\what{H}$ (or the symbol
$\what{F}$). Therefore, taking into account
\eqref{fk:def2}--\eqref{def:fktilde}, the determinants in the right-hand side
of \eqref{geneig:Nikhier} are in this situation constructed with the functions
$f_{k}(z,x)$.

We know by Lemma~\ref{lemma:PklnPk} that
\[
\lim_{n\rightarrow\infty}\frac{\wtil P_{k,l,n}(z)}{z \wtil P_{k,n}(z)}=:G(z),
\]
uniformly on compact subsets of $\cee\setminus[0,\infty)$. Since $G\not\equiv
0$, the measures $\nu_{l,k,n}$ are all positive for $n$ sufficiently large, or
they are all negative for $n$ sufficiently large. Therefore $G$ is an analytic
function in $\cee\setminus[0,\infty)$ that satisfies one of the following
properties:
\begin{itemize}
\item[$(a)$] $G(x)<0$ for all $x\in (-\infty,0)$ and $G$ maps the upper half plane into the lower half plane,
\item[$(b)$] $G(x)>0$ for all $x\in (-\infty,0)$ and $G$ maps the upper half plane into the upper half plane.
\end{itemize}
Applying Theorem A.4 from \cite{KreinNud}, we deduce that
\[
G(z)=\alpha+\int\frac{\ud \hat{\nu}_{l,k}(t)}{z-t}, \qquad
z\in\cee\setminus[0,\infty),
\]
where $\alpha\in\er$ and $\hat{\nu}_{l,k}$ is a measure with constant sign
supported on $\er_{+}$.

In particular, applying the above equations together with
\eqref{geneig:Nikhier}, \eqref{geneig:RH}, and \eqref{cauchy:hierl:det}, we
obtain
\begin{equation}\label{limit}
\lim_{n\rightarrow\infty}\frac{P_{k,l,n}(x)}{P_{k,n}(x)}
=x^{k-l+p+1}\Big(\alpha+\int\frac{\ud\hat{\nu}_{l,k}(t)}{x^{p+1}-t}\Big)=(-1)^{l-k}f_{l,k}(z_{0}(x),\ldots,z_{k}(x),x),
\end{equation}
uniformly on compact subsets of $\cee\setminus S_{+}$. Now,
Lemma~\ref{cor:degree} implies that $\alpha=0$. Finally, using
\[
\frac{x^{k-l+p+1}}{x^{p+1}-t}=\frac{1}{p+1}\sum_{m=0}^{p}\frac{(e^{\frac{2\pi
\ir m}{p+1}}s)^{k-l+1}}{x-e^{\frac{2\pi \ir m}{p+1}}s},\qquad t=s^{p+1},
\]
we deduce
\[
x^{k-l+p+1}\int_{\er_{+}}\frac{\ud \hat{\nu}_{l,k}(t)}{x^{p+1}-t}
=\frac{1}{p+1}\int_{S_{+}}\frac{s^{k-l+1}}{x-s}\,\ud\hat{\nu}_{l,k}(s^{p+1}).
\]
This justifies \eqref{Cauchytransf:property} and \eqref{nukl:symmetry} with
$\til{\nu}_{l,k}:=\frac{(-1)^{l-k}}{p+1}\,\hat{\nu}_{l,k}$. and we see from
\eqref{limit} that the function $f_{l,k}(z_{0}(x),\ldots,z_{k}(x),x)$ has no
singularities outside $S_{+}$. The proof is analogous for odd values of $k$. It
is clear from the analyticity of $f_{l,k}(z_{0}(x),\ldots,z_{k}(x),x)$ on
$\cee\setminus(\Gamma_{k}\cup\mathcal{A}_{k})$, see
Remark~\ref{analyticity:flk}, that the measure $\nu_{l,k}$ is supported on
$\Gamma_{k}\cup\mathcal{A}_{k}$.
\end{proof}

\noindent\textit{Proof of Theorem~\ref{theorem:Nik:property}:} We have
precisely shown in Theorem~\ref{theorem:mukl} that if
$\ud\nu_{l,k}(x)=g_{l,k}(x) \ud x+\ud \nu_{l,k}^{(s)}(x)$ denotes the Lebesgue
decomposition of $\nu_{l,k}$, then for $l\in [k+2:p]$,
\[
\frac{g_{l,k}(x)}{g_{k+1,k}(x)}=\frac{f_{l,k,+}(z_{0}(x),\ldots,z_{k}(x),x)-f_{l,k,-}(z_{0}(x),\ldots,z_{k}(x),x)}
{f_{k+1,k,+}(z_{0}(x),\ldots,z_{k}(x),x)-f_{k+1,k,-}(z_{0}(x),\ldots,z_{k}(x),x)},\qquad
x\in\Gamma_{k},
\]
is expressible as the Cauchy transform of a measure $\nu_{l,k+1}$ supported on
the star complementary to $\Gamma_{k}$. \hfill $\square$



\subsection{Proof of Lemma~\ref{cor:degree}}
\label{subsection:proof:degree}

In this section we will prove Lemma~\ref{cor:degree}. First we establish the
following result.

\begin{lemma}\label{lemma:degree} (Asymptotics of $f_j(z_k,x)$:)
Let $r$ be a multiple of $p$ and assume the ordering \eqref{ordering:as}. The
functions $f_j(z_k(x),x)$ in \eqref{fk:def} behave for $x\to\infty$ as
\begin{equation}\label{degree:f0}
\begin{array}{ll}
f_0(z_0(x),x) =(-1)^{r} x^{r}+O(x^{r-p-1}),& \\ f_0(z_k(x),x)=O(x^{r-p-1}),&
k\in [1:p],
\end{array}
\end{equation}
and
\begin{equation}\label{degree:fj}
\begin{array}{ll}
f_j(z_0(x),x) =(-1)^{r} x^{r-j}+O(x^{r-j-p-1}),& \\
f_j(z_k(x),x)=C_{j,k}\,x^{r-j}+O(x^{r-j-p-1}),& k\in [1:j],\\
f_j(z_k(x),x)=O(x^{r-j-p-1}),& k\in [j+1:p],
\end{array}
\end{equation}
for $j\in [1:p]$, for certain constants $C_{j,k}\neq 0$, $k\leq j$.
\end{lemma}

Note that the $O$-terms jump with powers of $x^{-p-1}$ rather than $x^{-1}$.
This is due to the rotational symmetry under rotations with $\exp(2\pi
\ir/(p+1))$.

Lemma~\ref{lemma:degree} implies that for $l\geq k$,
\begin{multline*}
\begin{pmatrix}
f_{0}(z_0(x)) & \ldots & f_{0}(z_{k}(x)) \\
\vdots & & \vdots  \\
f_{k-1}(z_{0}(x)) & \ldots & f_{k-1}(z_k(x))\\
f_{l}(z_{0}(x)) & \ldots & f_{l}(z_k(x))
\end{pmatrix} = \diag(x^{r},x^{r-1},\ldots,x^{r-k+1},x^{r-l})\\ \times\begin{pmatrix}
1 & O(x^{-p-1}) & O(x^{-p-1}) & \ldots & O(x^{-p-1}) & O(x^{-p-1}) \\
1 & C_{1,1} & O(x^{-p-1}) & \ldots & O(x^{-p-1})& O(x^{-p-1}) \\
1 & C_{2,1} & C_{2,2} & \ldots & O(x^{-p-1})& O(x^{-p-1})\\
\vdots & \vdots & \vdots & & \vdots & \vdots\\
1 & C_{k-1,1} & C_{k-1,2} & \ldots & C_{k-1,k-1} & O(x^{-p-1})\\
1 & C_{l,1} & C_{l,2} & \ldots & C_{l,k-1} & C_{l,k}
\end{pmatrix}
\end{multline*}
where each $C_{j,k}$ is a non-zero constant. Therefore,
\begin{equation}\label{det:estimate}
\begin{vmatrix}
f_{0}(z_0(x)) & \ldots & f_{0}(z_{k}(x)) \\
\vdots & & \vdots  \\
f_{k-1}(z_{0}(x)) & \ldots & f_{k-1}(z_k(x))\\
f_{l}(z_{0}(x)) & \ldots & f_{l}(z_k(x))
\end{vmatrix} = Cx^{r+(r-1)+\ldots+(r-k+1)+(r-l)}(1+O(x^{-p-1})),
\end{equation}
for some constant $C\neq 0$. Taking ratios of such determinants, we then get the desired
Lemma~\ref{cor:degree}.

In the rest of this section we prove \eqref{degree:f0}--\eqref{degree:fj}.
First of all, the statements involving $z_0(x)$ follow easily from the
Widom-type formula \eqref{Widom:2} (applied to the antidiagonal reflected
symbol \eqref{def:Fhat}) taking
into account that $z_0(x)\sim x^{-r}$ and $z_1(x),\ldots,z_{p}(x)=O(x^{r/p})$
for $x\to\infty$, and that $z_0(x)\ldots z_{p}(x)=(-1)^{r+p}/\mathsf{f}_p$.

Next, we prove \eqref{degree:f0}--\eqref{degree:fj} for the functions $z_k(x)$
with $k\geq 1$.  Let $\vece_j\in\cee^r$ be the standard basis vector which has
all its entries equal to zero except for the entry in position $j$, which is
equal to $1$. Let $P$ be the permutation matrix of size $r\times r$ which acts
on the vectors $\vece_j$ by the rule
$$P\vece_{ap+b}=\vece_{br/p+a},$$ for any $a\in [0:r/p-1]$ and
$b\in [0:p-1]$. Let $D$ be the $r\times r$ diagonal matrix
\begin{equation}\label{D:degree}D:=\diag\left(I_{p}, z^{\frac pr}I_{p},
z^{\frac{2p}{r}}I_{p},\ldots,z^{\frac{r-p}{r}}I_{p}\right).\end{equation} We
conjugate the block Toeplitz symbol $F(z,x)$ by the matrices $D$ and $P$. This
results in the following matrix:
\begin{equation}\label{symbol:permuted:modp}
PDF(z,x)D^{-1}P^{-1}
= \begin{pmatrix} A_0 & I & 0 & \ldots & 0 & 0 \\
0 & A_1 & I & \ddots & 0 & 0 \\
\vdots & \ddots & \ddots & \ddots & \ddots & \vdots \\
0 & 0 & 0 & A_{p-3} & I & 0
\\
0 & 0 & 0 & 0 & A_{p-2} & I
\\ Z & 0 & 0 & 0 & 0 & A_{p-1}
\end{pmatrix},
\end{equation}
with $Z := z^{-\frac pr}\begin{pmatrix} 0 & I_{r/p-1} \\ 1 & 0
\end{pmatrix},$
and \begin{equation}\label{permuted:p} A_j = \begin{pmatrix}
-x & 0 & 0 & \ldots & 0 & a_{r-p+j}z^{p/r} \\
a_{j}z^{p/r} & -x & 0 & \ddots & 0 & 0 \\
0 & a_{p+j}z^{p/r} & -x & \ddots & \ddots & 0 \\
\vdots & \ddots & \ddots & \ddots & \ddots & \vdots \\
0 & 0 & 0 & a_{r-3p+j}z^{p/r} & -x & 0 \\
0 & 0 & 0 & 0 & a_{r-2p+j}z^{p/r} & -x
\end{pmatrix},
\end{equation}
for $j\in [0:p-1]$. Note that each of the blocks in
\eqref{symbol:permuted:modp} is a square matrix of size $r/p$ by $r/p$.

Fix $k\in [1:p]$. We already know that $z_k(x)\sim C_k x^{r/p}$ as
$x\to\infty$. Hence $z_k^{p/r}(x)\sim C_k^{p/r} x$. From
\eqref{symbol:permuted:modp}--\eqref{permuted:p} and the fact that $\det
F(z_k(x),x)=0$ we see that (see also \eqref{zk:infty:0})
$$ C_k\in\left\{\prod_{n=0}^{r/p-1}
a_{pn}^{-1},\prod_{n=0}^{r/p-1} a_{pn+1}^{-1},\ldots,\prod_{n=0}^{r/p-1}
a_{pn+(p-1)}^{-1}\right\}.
$$
Combining this with \eqref{ordering:as} and \eqref{ordering:rootszk}, we obtain
that for $x\to\infty$,
\begin{equation}\label{zk:infty}
\begin{array}{l} z_1(x) = \left(\prod_{n=0}^{r/p-1}a_{pn}\right)^{-1}
x^{r/p}(1+O(x^{-p-1})),\\
\qquad\vdots \\
z_{p}(x) = \left(\prod_{n=0}^{r/p-1}a_{pn+p-1}\right)^{-1}
x^{r/p}(1+O(x^{-p-1})).\end{array}
\end{equation}

Now we consider $\det F^{j-1,0}(z,x)$, i.e., the determinant obtained by
skipping the $j$th row and the first column of $F(z,x)$, $j\in [1:p]$. Clearly,
this determinant is not influenced by the conjugation with the diagonal matrix
$D$ in \eqref{D:degree}, in the sense that $\det F^{j-1,0}(z,x) = \det
(DFD^{-1})^{j-1,0}(z,x)$. For a matrix $A$ denote with $\wtil A$ the matrix
obtained by skipping the first row of $A$ and with $\what A$ the matrix
obtained by skipping the first column of $A$. Then from
\eqref{symbol:permuted:modp} we obtain
\begin{equation*} \det F^{j-1,0}(z,x)=
\pm\det\begin{pmatrix} \what A_0 & I \\
& A_1 & I \\
& & \ddots & \ddots \\
& & & \wtil A_{j-1} & \wtil I &  &   &  \\
& & &  & A_{j} & I &  &  \\
& & &  & & \ddots &  \ddots &  \\
& & &  &  &  & A_{p-2} & I
\\ \what Z & & &  &   &  & & A_{p-1}
\end{pmatrix}.
\end{equation*}
Now by repeated Gaussian elimination with the identity matrices $I$ as pivots,
the above determinant can be brought to the form
\begin{equation}\label{symbol:permuted:modp:skipped} \det F^{j-1,0}(z,x)=\pm\det\begin{pmatrix}
\pm\wtil A_{j-1}\ldots A_1\what A_0 & \wtil  I \\
\what Z & \pm A_{p-1}A_{p-2}\ldots A_{j}
\end{pmatrix}.
\end{equation}

Fix $k\in [1:p]$. To obtain the dominant behavior of
\eqref{symbol:permuted:modp:skipped} for $z=z_k(x)$ as $x\to\infty$, we should
only use the $(1,1)$ and the $(2,2)$ blocks in
\eqref{symbol:permuted:modp:skipped}. Note that both blocks are square. The
determinant of the $(2,2)$ block can be simply factored as $(\det A_j)(\det
A_{j+1})\ldots(\det A_{p-1})$ with
\begin{equation}\label{degree:est1} \det A_i(z=z_k(x)) =
\left\{\begin{array}{ll}
C_{i,k}\, x^{r/p}+O(x^{r/p-p-1}), & \textrm{if }k\neq i+1,\\
O(x^{r/p-p-1}), & \textrm{otherwise},
\end{array}\right.
\end{equation} for some $C_{i,k}\neq 0$, thanks to \eqref{zk:infty}. The determinant of
the $(1,1)$ block can be expanded by means of the Cauchy-Binet formula:
\begin{equation}\label{cauchy:binet} \det\left(\wtil A_{j-1}\ldots A_1\what A_{0}\right) =
\sum_{m_1,\ldots,m_{j-1}=0}^{r/p-1} (\det A_{j-1}^{m_j,m_{j-1}})\ldots (\det
A_{1}^{m_2,m_1})(\det A_{0}^{m_{1},m_0}),
\end{equation}
where the sum runs over all $(j-1)$-tuples of integers $(m_1,\ldots,m_{j-1})$,
each of them ranging between $0$ and $r/p-1$, with boundary conditions
$m_0=m_j:=0$. We remind the reader that $A^{i,j}$ denotes the submatrix of $A$
obtained by deleting row $i$ and column $j$. Clearly,
$$
\det A_{i}^{m_{i+1},m_i}(z=z_k(x)) = \widetilde{C}_{i,k}\,
x^{r/p-1}+O(x^{r/p-p-2}),\qquad \widetilde{C}_{i,k}\neq 0,
$$
for all $i\in [0:j-1]$. Using this in \eqref{cauchy:binet} we get
\begin{equation}\label{degree:est2} \det\left( \wtil A_{j-1} \ldots
A_1\what A_{0}\right)(z=z_k(x)) = C_{j,k}\,x^{\frac{jr}{p}-j}(1+O(x^{-p-1})),
\end{equation}
as $x\to\infty$, for a new constant $C_{j,k}$. This constant $C_{j,k}$ is
nonzero, since cancelation of the leading order terms in the sum in
\eqref{cauchy:binet} cannot occur. This is due to
Lemma~\ref{lemma:basic:twodiag} below.

By combining \eqref{symbol:permuted:modp:skipped}, \eqref{degree:est1} and
\eqref{degree:est2}, we obtain the desired asymptotics in \eqref{degree:fj} for
$z=z_k(x)$ with $k\in [1:p]$.

The asymptotics in \eqref{degree:f0} can be proved by a similar argument. We
now have the relation $z^{-1}\det F^{r-1,0}(z,x) = z^{-p/r}\det (DFD^{-1})^{r-1,0}(z,x)$.
For a matrix $A$ denote now with $\wtil A$ the matrix obtained by skipping the
last (rather than the first) row of $A$ and denote again with $\what A$ the
matrix obtained by skipping the first column of $A$. Then by Gaussian
elimination we get the following analogue of
\eqref{symbol:permuted:modp:skipped}:
$$ z^{-1}\det F^{r-1,0}(z,x) = \pm z^{-p/r}\det(z^{-p/r}I\pm\wtil A_{p-1}\ldots A_1\what
A_0),
$$
where $z^{-p/r}I$ arises as the submatrix obtained by skipping the last row and
the first column of $Z$. Clearly, the dominant behavior for $z=z_k(x)$ as
$x\to\infty$ comes from $\pm z^{-p/r}\det(\wtil A_{p-1}\ldots A_1\what A_0)$.
This determinant can be evaluated using Cauchy-Binet in the same way as before.
Then we easily get the asymptotics in \eqref{degree:f0} for $z=z_k(x)$ with
$k\in [1:p]$. This ends the proof of Lemma~\ref{lemma:degree}. $\bol$

To conclude this section, we state the following lemma which was used above.

\begin{lemma}\label{lemma:basic:twodiag} Let $A$ be an $n\times n$ matrix of the form
$$A=\begin{pmatrix}-b_0 & & & & a_{n-1} \\
a_0 & -b_1 \\
& a_1 & \ddots \\
& & \ddots & -b_{n-2} \\
& & & a_{n-2} & -b_{n-1}
\end{pmatrix},$$ with $a_k,b_k>0$ for all $k\in [0:n-1]$. Denote with
$A^{k,l}$ the submatrix obtained by skipping the $k$th row and the $l$th column
of $A$. Then
$$ (-1)^{n+k+l+1}\det A^{k,l}>0.
$$
\end{lemma}

\begin{proof} Straightforward verification.
\end{proof}

\end{document}